\documentclass[10pt]{article}

\usepackage{amsfonts}
\usepackage{amsmath}
\usepackage{amssymb}
\usepackage{amsthm}
\usepackage[alphabetic]{amsrefs}
\usepackage{arydshln}
\usepackage{calc}
\usepackage{graphicx}
\usepackage[hidelinks]{hyperref}
\usepackage[left=1in,top=1in,right=1in,foot=1in]{geometry}
\usepackage{todonotes}
\usepackage{microtype}
\usepackage{tikz-cd}
\usepackage{dynkin-diagrams}
\usepackage{lscape}

\usepackage[stable, symbol]{footmisc}

\newcommand{\diag}{\mathrm{diag}}

\DeclareMathOperator{\spn}{span}

\DeclareMathOperator{\codim}{codim}

\newcommand{\Restriction}{\upharpoonright}

\def\into{\mathrel{\hookrightarrow}}
\def\onto{\mathrel{\twoheadrightarrow}}
\DeclareMathOperator{\im}{\mathrm{im}}
\newcommand{\sets}[2]{\left\{#1\,\middle|\,#2\right\}}
\newcommand{\genrel}[2]{\left\langle #1\,\middle|\,#2\right\rangle}
\newcommand{\bangles}[1]{\left\langle #1 \right\rangle}
\newcommand{\sett}[1]{\left\{#1\right\}}

\newcommand{\stab}[2]{\mathrm{Stab}_{#1}(#2)}

\newcommand{\rquotient}{\backslash}

\DeclareMathOperator{\End}{End}

\newcommand{\Hom}[2]{\mathrm{Hom}(#1, #2)}
\newcommand{\Ext}{\mathrm{Ext}}
\newcommand{\HomOver}[3]{\mathrm{Hom}_{#1}(#2, #3)}

\DeclareMathOperator{\Mat}{Mat}

\newcommand{\id}{\mathrm{id}}
\newcommand{\opp}{^\mathrm{opp}}
\DeclareMathOperator{\Ind}{Ind}
\DeclareMathOperator{\Irr}{Irr}

\newcommand{\Grp}{\mathbf{Grp}}

\newcommand{\Rep}{\mathbf{Rep}}
\newcommand{\Mod}{\mathbf{Mod}}

\newcommand{\Coh}{\mathrm{Coh}}

\DeclareMathOperator{\Spec}{Spec}

\newcommand{\Pp}{\mathbb{P}}

\newcommand{\git}{\mathbin{
  \mathchoice{/\mkern-6mu/}
    {/\mkern-6mu/}
    {/\mkern-5mu/}
    {/\mkern-5mu/}}}

\newcommand{\Bb}{\mathcal{B}}

\newcommand{\pt}{\mathrm{pt}}

\newcommand{\X}{\mathcal{X}}
\newcommand{\HH}{\mathbf{H}}

\newcommand{\Ff}{\mathcal{F}}
\newcommand{\Gg}{\mathcal{G}}
\newcommand{\Ee}{\mathcal{E}}

\DeclareMathOperator{\supp}{supp}

\newcommand{\trace}[2]{\mathrm{trace}\left({#1}\,,{#2}\right)}

\newcommand{\N}{\mathbb{N}}
\newcommand{\Z}{\mathbb{Z}}
\newcommand{\C}{\mathbb{C}}
\newcommand{\Q}{\mathbb{Q}}

\newcommand{\F}{\mathbb{F}}
\newcommand{\Oo}{\mathcal{O}}

\newcommand{\mM}{\mathfrak{m}}

\newcommand{\G}{\mathbf{G}}

\newcommand{\Waff}{W_{\mathrm{aff}}}
\newcommand{\M}{\mathbf{M}}

\newcommand{\cc}{\mathbf{c}}

\newcommand{\GL}{\mathrm{GL}}
\newcommand{\PGL}{\mathrm{PGL}}
\newcommand{\SL}{\mathrm{SL}}

\newcommand{\SO}{\mathrm{SO}}
\newcommand{\Sp}{\mathrm{Sp}}

\newcommand{\OO}{\mathrm{O}}

\newcommand{\Spin}{\mathrm{Spin}}
\newcommand{\Gm}{{\mathbb{G}_\mathrm{m}}}

\newcommand{\Aut}{\mathrm{Aut}}
\newcommand{\Sn}{\mathfrak{S}}
\newcommand{\An}{\mathfrak{A}}

\newcommand{\St}{\mathrm{St}}
\newcommand{\triv}{\mathrm{triv}}
\newcommand{\sgn}{\mathrm{sgn}}

\newcommand{\bq}{\mathbf{q}}

\newtheorem{theorem}{Theorem}
\newtheorem*{theorem*}{Theorem}
\newtheorem{cor}{Corollary}
\newtheorem{prop}{Proposition}
\newtheorem{lem}{Lemma}

\theoremstyle{definition}
\newtheorem{dfn}{Definition}
\newtheorem*{dfn*}{Definition}
\theoremstyle{remark}
\newtheorem{ex}{Example}
\newtheorem*{ex*}{Example}
\newtheorem{rem}{Remark}

\newcommand{\Gd}{G^\vee}

\newcommand{\Hr}{\overline{\HH}^{\mathrm{rigid}}}
\newcommand{\Hrq}{\overline{H}^{\mathrm{rigid}}_q}

\newcommand{\Hrs}{\overline{H}^{\mathrm{rigid}}_{\mathrm{sub}}}
\newcommand{\phir}{\overline{\phi}^{\mathrm{rigid}}}

\newcommand{\Jr}{\overline{J}^{\mathrm{rigid}}}
\newcommand{\Jrs}{\overline{J}^{\mathrm{rigid}}_{\mathrm{sub}}}
\newcommand{\RH}{\mathcal{R}}
\newcommand{\RHr}{\mathcal{R}_{\mathrm{rigid}}}
\newcommand{\RHdi}{\mathcal{R}_{\mathrm{diff-ind}}}

\newcommand{\cl}{\mathrm{cl}}

\newcommand{\hhh}{\mathrm{HH}_0}

\newcommand{\Pin}{\mathrm{Pin}}

\newcommand{\rank}{\mathrm{rank}}

\newcommand{\Zur}{Z_{G^\vee}(u)^{\mathrm{red}}}
\newcommand{\Yext}{\mathbf{Y}}

\newcommand{\Zz}{\mathcal{G}}
\newcommand{\Y}{\mathbf{Y}}

\newcommand{\oO}{\mathbb{O}}

\newcommand{\std}{\mathrm{std}}

\newcommand{\mMu}{\mathcal{M}_u}

\newcommand{\bA}{{\mathbb A}}

\newcommand{\bF}{{\mathbb F}}

\newcommand{\bZ}{{\mathbb Z}}

\newcommand{\cA}{{{\mathcal A}}}
\newcommand{\cB}{{{\mathcal B}}}
\newcommand{\cC}{{{\mathcal C}}}

\newcommand{\cG}{{{\mathcal G}}}
\newcommand{\cH}{{{\mathcal H}}}

\newcommand{\cN}{{{\mathcal N}}}

\newcommand{\cT}{{{\mathcal T}}}

\newcommand{\cW}{{{\mathcal W}}}
\newcommand{\cX}{{{\mathcal X}}}
\newcommand{\cZ}{{{\mathcal Z}}}

\newcommand{\sslash}{\mathbin{/\mkern-6mu/}}
\newcommand{\AdRumynin}{{{\mathrm A}{\mathrm d}_{\cG^0}}}
\newcommand{\Spe}{{{\mathrm S}{\mathrm p}{\mathrm e}{\mathrm c}}}
\newcommand{\GG}{{\cG\!\sslash\! \cG}}
\newcommand{\VG}{{V_{\gamma}}}
\newcommand{\wG}{{\widetilde{\cG}}}

\providecommand{\keywords}[1]{\textbf{\textit{Keywords---}} #1}

\title{The asymptotic Hecke algebra and rigidity \\ \large With an appendix by Dmitriy Rumynin}
\date{\today}
\author{Stefan Dawydiak \thanks{Mathematical Institute, Universit\"{a}t Bonn, Bonn 53111 Germany; email \texttt{dawydiak@math.uni-bonn.de}}}

\begin{document}
\maketitle
\begin{abstract}
We reprove the surjectivity statement of Braverman-Kazhdan's spectral description of Lusztig's asymptotic 
Hecke algebra $J$ in the context of $p$-adic groups. The proof is based on Bezrukavnikov-Ostrik's
description of $J$ in terms of equivariant $K$-theory. As a porism, we prove that 
the action of $J$ extends from the non-strictly positive unramified characters to the complement of a finite 
union of divisors, and that the trace pairing between the Ciubotaru-He rigid cocentre of an 
affine Hecke algebra with equal parameters and the rigid quotient of its Grothendieck group is 
perfect whenever the parameter $q$ is not a root of the Poincar\'{e} polynomial of the finite 
Weyl group.

Without recourse to $K$-theory, we prove a weak version of Xi's description of $J$ in type $A$.

As an application of relationship between $J$ and the rigid cocentre, we prove that
the formal degree of a unipotent discrete series representation of a connected reductive
$p$-adic group $G$ with a split inner form has denominator dividing the Poincar\'{e} polynomial of the Weyl 
group of $G$. Additionally, we give formulas for $t_w$ in terms of inverse and spherical Kazhdan-Lusztig 
polynomials for $w$ in the lowest cell.
\end{abstract}
%
\tableofcontents

\section{Introduction}
\subsection{The asymptotic Hecke algebra and $p$-adic groups}
Let $F$ be a non-archimedean local field with residue field $\F_q$, and $\G$ be a connected reductive group 
defined and split over $F$. Let $\HH$ be the affine Hecke algebra over $\mathcal{A}=\C[\bq^{\pm \frac{1}{2}}]$ 
attached to the affine Weyl group of $\G$, and $H$ be its specialization its specialization at $\bq=q$, the 
Iwahori-Hecke algebra $H(G,I)$ of $G=\G(F)$. In \cite{BK}, Braverman and Kazhdan proposed that the asymptotic
Hecke algebra $J$ attached by Lusztig to $\HH$, and hence to $H$, should be an algebraic analogue 
of (the Iwahori part of) the Harish-Chandra Schwartz algebra $\mathcal{C}(G)$, by defining a map 
$\eta\colon J\to \mathcal{E}_{J}$, fitting into a diagram
\begin{equation}
\label{eqn BK summary diagram}
\begin{tikzcd}
f\arrow[dd, mapsto]&H\arrow[dd, "\sim"]\arrow[dr, hook, "\phi_q"]\arrow[rr, hook]&&\mathcal{C}(G,I)\arrow[dd, "\sim"]&f\arrow[dd, mapsto]\\
&&J\arrow[ur, "\tilde{\phi}", hook]\arrow[d, "\eta", hook]&&\\
\pi(f)&\mathcal{E}^I\arrow[r, hook]&\mathcal{E}^I_J\arrow[r, hook]&\mathcal{E}_t^I&\pi(f),
\end{tikzcd}
\end{equation}
where $\mathcal{E}_J$ is a subring of $\mathcal{C}(G)$ defined
spectrally via the Paley-Wiener theorem. In \cite{Plancherel}, we proved injectivity of $\eta$.
In the present paper, we prove surjectivity\footnote{In \textit{loc. cit.}, a previous version of the present paper is referred to as proving surjectivity for all but four cells occurring only in types $E_n$.}.

The method of proof is essentially as advocated in \cite{BK}. Let $u\in\G^\vee(\C)$ be a unipotent conjugacy class in the dual group of $\G$. Then Bezrukavnikov-Ostrik \cite{BO} describe the corresponding direct 
summand of $J$ as $J_u=K_{\Zur}(\Y_u\times\Y_u^{\opp})$, where $\Zur$ is the reductive part of the centralizer of $u$,
and $\Y_u$ is a finite $\Zur$-set together with the data of a central extension of each stabilizer in $\Zur$. 
We use \cite{BO} to construct idempotent elements 
$t_{d,\rho}$ of $J_u$ yielding 
\begin{equation}
\label{eqn lower modification intro}
\begin{tikzcd}
J_u\arrow[r, "\eta_u", hook]&\mathcal{E}_{J,u}\arrow[r, "\iota_u", hook]&\bigoplus_{i,j}\Mat_{\dim\pi_{ij}^I}\left(\Oo\left(\mathfrak{o}_{\sigma_{ij}}\git W_{M_i}\right)\right)=:\mathcal{M}_u,
\end{tikzcd}
\end{equation}
where $\mathfrak{o}_{ij}=\X(M_i)\cdot`\sigma_{ij}$ for $\X(M_i)$ the variety of unramified characters of and $\sigma_{ij}$ a discrete-series representation of the Levi subgroup $M_i$ of $G$, respectively, and $W_i=W_{M_i}$ 
is the group parameterizing intertwining operators, and $\rho$ is a simple constituent of the permutation 
representation of $\pi_0(\Zur)$ on $\C[Y_u]$. We write $\pi_{ij}$ are the parabolic inductions of representations in $\mathfrak{o}_{ij}$. The main inputs to the construction are the observation, appearing 
below as Lemma \ref{lem opp-square of transitive set is trivial} that central extensions only appear in 
the ``off-diagonal part" of $\Yext_u\times\Yext_u^{\opp}$, together with an extension in Section \ref{subsection Construction of modules} of Lusztig's classification 
of $K_{\Zz}(Y\times Y)$-modules in \cite[Section 10]{cellsIV} to the centrally-extended case. 
A more refined version of this extension was also obtained Bezrukavnikov, Karpov, and Krylov
\cite[Thm. C]{BKK}. Later, Bezrukavnikov and Losev showed in \cite{BL} that one may take the set $Y_u$ to 
be Lusztig's canonical basis of $K_0(\Bb_u^{\Gm})$, yielding a connection between this basis and harmonic 
analysis on $G$. 

Being $I$-spherical, all the smooth $G$-representations we consider correspond to a single connected component of the Bernstein variety, namely $G^\vee\git G^\vee$, with $\Oo(G^\vee\git G^\vee)=Z(H)$ . Nonetheless, the tempered representations we consider have many different discrete supports. In the context of $J$, the representations extending to $J$-modules come naturally organized according to the monodromy of the $L$-parameter of their discrete supports. This is reflected in the appearance of the varieties $\Zur\git\Zur$ below.

Concretely, each of $J_u$, $\Ee_{J,u}$ and $\mathcal{M}_u$ are modules over the complexification $R(\Zur)$ of the 
representation ring of $\Zur$. In Section \ref{subsection Lower modifications of vectors bundles}, we 
explain that the outer terms are a vector bundle and a maximal Cohen-Macaulay sheaf, respectively, and $J_u$ and $\Ee_{J,u}$ are lower modifications of $\mathcal{M}_u$. 
To this end, we make the apparently novel observation that if $u$ appears in the parameter of a 
discrete series representation $\sigma$ of a Levi subgroup $M$, then its orbit the quotient $\mathfrak{o}_\sigma\git W_M$ of its orbit $\mathfrak{o}_\sigma=\X(M)\cdot\sigma$ is a connected
component of $\Zur\git\Zur$.

Studying the scheme-theoretic support of $\mathcal{M}_u/J_u$, together with a simpler argument for the 
discrete series of $G$ yields
\theoremstyle{theorem}
\newtheorem*{thm:onto}{Theorem \ref{thm BK onto}}
\begin{thm:onto}[\cite{BK}, Theorem 2.4]
Let $\G$ be a connected reductive group over $F$ and $G^\vee$ its dual group over $\C$. 
Then $\eta_u$ is an isomorphism.
\end{thm:onto}
%
This same result 
was obtained by Bezrukavnikov, Karpov, and Krylov in \cite[Prop. 1.3.6]{BKK}.

When $\G=GL_n$, the construction of \eqref{eqn lower modification intro} does not require the results of
\cite{BO}, and we recover a weak version of the main result of \cite{XiAMemoir}:
\theoremstyle{theorem}
\newtheorem*{thm:GL}{Theorem \ref{thm Ju matrix algebra for GLn}}
\begin{thm:GL}
Let $G=\GL_n(F)$ or $\PGL_n(F)$. Then $\iota_u\circ\eta_u$ in \eqref{eqn lower modification intro} is an isomorphism for all unipotent
conjugacy classes $u\in\GL_n(\C)$. In particular, $\iota_u$ and $\eta_u$ are each isomorphisms.
If $G$ is semisimple of type $A$, then $\eta_u$ is an isomorphism. 
\end{thm:GL}
Finally, the fact that the codomain of $\iota_u$ in \eqref{eqn lower modification intro} consists of matrices whose entries are regular, as opposed to merely rational, in $\nu$ implies 
\begin{cor}
The action of $J_u$, \textit{a priori} defined on representations $\pi_\nu=i_{P}^G(\sigma\otimes\nu)$
for $\sigma\in\Ee_2(M_P)$ and $\nu\in\X(M_P)$ nonstrictly positive, extends to the complement of a
finite union of divisors in $\Zur\git\Zur$.
\end{cor}
This generalizes a similar statement for the functions $\nu\mapsto\trace{\pi_\nu}{t_w}$ proven
in \cite{Plancherel} and confirms an expectation of Braverman and Kazhdan.

\subsection{Appearance of central extensions}
In the wake of \cite{BO}, it was expected that the central extensions allowed for in \textit{op. cit.} did not actually appear, so long as $G^\vee$ was simply-connected. In \cite{BDD}, Bezrukavnikov, Dobrovolska, 
and the author showed that this is false, by producing an explicit 
counter-example in the case $\Gd=\Sp_6(\C)$ and $u$ corresponding to the 
partition $(2,2,2)$. The same counter-example was computed independently by 
Qiu and Xi in \cite{QX}. 

When $\Gd$ is not simply-connected, nontrivial central extensions appear immediately; the example 
$G^\vee=\PGL_2(\C)$ is treated in \cite[Section 8.3]{XiAMemoir}. One way to see the \textit{a priori} appearance of projective representations of $\Zur$ is via the construction of $J_u$-modules we recall in Section \ref{section Rank 1 idempotents}, where we recall the construction of modules over the equivariant $K$-theory of the square of a finite set. Absent central extensions, this construction would produce all the simple $J_u$-modules. These modules are labelled in particular by an irreducible representation of $\pi_0(Z_{G^\vee}(u,s))$; it can happen that one needs at least two non-isomorphic representations to describe all the modules associated to $(u,s)$. Here, the action is via $\pi_0(Z_{G^\vee}(u,s))\to\pi_0(\Zur)$. Note that this morphism can be far from injective.

With modifications, the construction of modules is indeed exhaustive, as we show in Section \ref{subsection Construction of modules}, where we explain that the appearance of nontrivial central extensions is
intimately linked to \eqref{eqn lower modification intro} being a genuine lower modification, \textit{i.e.}
to $\iota_u$ failing to be injective on fibres, and that the latter can be forced by reducibility of
certain tempered representations of the $p$-adic group $G$. From this perspective, simple connectedness of 
$G^\vee$ avoids only central extensions 
forced by reducibility of unitary principal series. Indeed, as we explain in Examples 
\ref{ex SL2 central ext} and \ref{ex BDD central ext}, the counter-example of \cite{BDD} and \cite{QX} 
arises for essentially the same reason as for the well-known case $\Gd=\PGL_2(\C)$---the absence of central 
extensions would contradict reducibility of certain tempered representations.
\subsection{The Ciubotaru-He rigid cocentre}
In Section \ref{section rigid det}, we given a second application of the elements $t_{d,\rho}$:
In \cite{CH}, Ciubotaru and He studied the cocentre of the affine Hecke algebra $\HH$
for unequal parameters and proved several results on its independence with respect to the parameters 
$\bq_\alpha$. 
For example, they proved that the cocentre is a free $\C[\bq_\alpha^{\pm 1}]_\alpha$-module with an 
explicit basis, and in addition defined and studied a certain submodule of it, the so called 
\emph{rigid cocentre} $\Hr$, which they explain is intended to serve as a partial replacement for the Tits 
deformation theorem in the affine case. The rigid cocentre is free of finite rank over $\C[\bq_\alpha^{\pm 1}]_\alpha$, 
with a canonical basis that is independent of the parameters. For all but finitely-many specializations, 
it is dual to a certain finite-dimensional quotient, the \emph{rigid quotient} $\RHr=\RH/\RHdi$ of the Grothendieck 
group of the affine Hecke algebra, in which classes $[i_P^G(\sigma)]$ and 
$[i_P^G(\sigma\otimes\nu)]$ of parabolic inductions are identified for all unramified characters $\nu$. 
Thus the rigid cocentre consists of forms which are rigid under deformations of the central character. 

In the equal-parameters case, independence phenomena with respect to the parameter $\bq$ are
well-known to often be related to Lusztig's asymptotic Hecke algebra $J$ via Lusztig's map $\phi\colon\HH\to J\otimes\mathcal{A}$.
In the case when $\bq$ is a formal variable, the connection was made in Lusztig's original series 
of papers \cite{cellsI}, \cite{cellsII}, \cite{cellsIII}, \cite{cellsIV}.
Additionally, in \cite{BDD} and 
the appendix thereof, the authors of the main paper together with the authors 
of the appendix proved  that $\phi$ induces 
an isomorphism $\bar{\phi}\colon\hhh(\HH)\to\hhh(J\otimes\mathcal{A})$ of cocentres away from roots of unity
in $\C^\times=\Spec\mathcal{A}$.

\subsubsection{The rigid pairing}
In each of three equal-parameter examples considered in \cite{CH}, it was observed that the rigid pairing
$\Hrq\otimes\RHr\to\C$ was perfect whenever $P_{W_f}(q)\neq 0$. The idempotents $t_{\omega d,\rho}$ used to define \eqref{eqn lower modification intro} naturally lie in the rigid cocentre, and we show
that in fact the functionals $[\pi]\mapsto\trace{\pi}{t_{\omega d,\rho}}$, $\omega\in\pi_1(\Gd)$, span the 
rigid cocentre when $q>1$. We then prove
\theoremstyle{Theorem}
\newtheorem*{thm:rigid}{Theorem \ref{cor main thm body}}
\begin{thm:rigid}
Let $\HH$ be the affine Hecke algebra with equal parameters associated to a semisimple root 
datum. 
If $P_{W_f}(q)\neq 0$, the rigid pairing is perfect.
\end{thm:rigid}
%
In particular, we obtain a basis of $\Hr$ 
indexed by data from the Galois side of the Kazhdan-Lusztig parametrization, whereas the basis
of \cite{CH} is related to the Kottwitz set of $G$. However, we do not extract this basis from the spanning set we construct in any canonical way, see Remark \ref{rem extracting basis}.
\subsubsection{Application: formal degrees of unipotent discrete series representations}
Let $\G$ be connected reductive but not necessarily split over $F$. Let $\omega$ be a unipotent discrete series 
representation of $\G(F)$ and let $d(\omega)$ be its formal degree. 
By \cite{Solleveldfdegsequel}, there is a unique rational function in $\bq$ that
specializes to $d(\omega)$ for all $q>1$ (when $\G$ splits over an unramified extension of $F$, this was proved in \cite{FOS}). Of course $\G$ is not semisimple, then 
$d(\omega)$ is a measure on a compact torus, absolutely continuous with respect to a Haar measure, and our
statements are about the Radon-Nikodym derivative.
It is also shown in \textit{ops. cit.} that the possible denominators
are tightly constrained, as is expected to hold in wider generality, c.f. \cite[Prop. 4.1]{GrossReeder}.
As the formal degrees of unipotent discrete series are, up to a constant independent of $q$, constant in 
$L$-packets, they are mostly governed by the principal block, where the formulas of Reeder \cite{Reeder} suggest 
even stricter constraints. Using our results  on the rigid pairing and our results from \cite{Plancherel} about 
values, in particular at $1\in\tilde{W}$, of the functions $t_w$, we prove in Section \ref{subsection application formal degrees} 
\theoremstyle{Theorem}
\newtheorem*{thm:fdeg}{Theorem \ref{thm fdeg body}}
\begin{thm:fdeg}
Let $F$ be a non-archimedean local field and let $\G$ be a connected reductive group defined over $F$. 
Suppose that the quasi-split inner form $\G^*$ of $\G$ is actually split. Let $\omega$ be a unipotent discrete 
series representation of $\G(F)$ and let $d(\omega)$ be its formal degree, thought of as a rational function of 
$\bq$. Then the denominator of $d(\omega)$ divides a power of $P_W(\bq)$, where $W$ is the Weyl group of 
$G^\vee$.
\end{thm:fdeg}
%
Note that Reeder computed all unipotent formal degrees for exceptional groups in 
\cite{ReederLeubeck}, verifying Theorem \ref{thm fdeg body} for exceptional groups directly.

A major input to the proof is the \textit{a priori} rationality of $d(\omega)$ proved in \cite{FOS}--- it is a 
major input to the results of \cite{Plancherel} we use. It is the other assumptions, such as splitness of 
$\G^*$, that prevent us from proving a substantially stronger version of Theorem \ref{thm fdeg body}. 
Nonetheless, our techniques work in some circumstances where $\G^*$ is only quasi-split, see Remark 
\ref{rem extension to quasi-split}.
\subsection{Acknowledgements}
The author thanks Gurbir Dhillon, Hiroki Kato, Ivan Losev, Alexander Petrov, Oron Propp,  David Schwein, and Alex Youcis for helpful conversations. The author thanks Dennis Gaitsgory for helpful conversations and for suggesting the use of an ancestor of Lemma \ref{lem Dennis}. 

The author thanks Roman Bezrukavnikov 
and Ivan Karpov for several helpful and patient discussions relating to Section \ref{section surjectivity},
especially the implications of Example \ref{ex SL2 central ext}. 

The author thanks Dan Ciubotaru for introducing him to the rigid quotient and rigid determinant, and 
for helpful conversations at the conferences
``From $E_6$ to 
$\tilde{\tilde{E}}_{60}$'' at the  Korteweg-de Vries Institute in honour of Eric Opdam, and ``Workshop on interactions between representation theory, combinatorics, and geometry" at 
the National University of Singapore. The author thanks the organizers of both events for their 
hospitality.

This research was partially supported by NSERC, and was carried out in part while the author was
a visitor of the Max-Planck-Institut f\"{u}r Mathematik, which he 
thanks for hospitality and financial support.
\section{Rank $1$ idempotents in the asymptotic Hecke algebra}
\label{section Rank 1 idempotents}
\subsection{The affine and asymptotic Hecke algebras and the matrix Paley-Wiener theorem}
\label{subsubsection and affine and asymptotic Hecke algebras}
Let $\G$ be a connected reductive group defined and split over a fixed non-archimedean local field $F$
with $G=\G(F)$. Let $(X^*, R, X_*, R)$ be its root datum with Weyl group $W_f$, and let $\Gd$ be the connected reductive group over
$\C$ with dual root datum, and $W=W_f\ltimes X_*$ be the extended affine Weyl group of $\G$.

Let $\HH$ be the affine Hecke algebra over $\mathcal{A}=\C[\bq^{1/2},\bq^{-1/2}]$ with standard
basis $\sett{T_w}_{w\in W}$ and relations $(T_s+1)(T_s-\bq)=0$ and $T_wT_{w'}=T_{ww'}$
when $\ell(ww')=\ell(w)+\ell(w')$. Let $\{C_w\}_{w\in W}$ and $\{C'_w\}_{w\in W}$ be the two Kazhdan-Lusztig bases. Write$w=w_f\lambda\in\Waff$, let $T_w\mapsto (-1)^{\ell(w_f)}q^{\ell(w)}T_{w^{-1}}^{-1}$ as in \cite[Def. 6(a)]{Plancherel} and recall that if $\omega(x)\in\pi_1(\G)$ labels the $W_f\ltimes \Z[\Phi^\vee]$-coset containing $x\in\Waff$, then
t one has 
\begin{equation}
\label{eqn Goldman on Cw}
{}^\dagger C_x=(-1)^{\ell(x)+\ell(\omega(x)_f)}C_x'
\end{equation}
for all $x\in \Waff$, by \cite[Lemma 2 (a)]{Plancherel}.

In \cite{cellsII}, Lusztig defined a based $\C$-algebra\footnote[1]{Actually $J$ is defined over $\Z$, and $H$ and $\phi$ are defined over 
$\Z\left[\bq^{\pm\frac{1}{2}}\right]$.}
$J$ with basis $\sett{t_w}_{w\in W}$. 
By construction $J=\bigoplus_\cc J_\cc$ is a direct sum of two-sided ideals $\spn\sets{t_w}{w\in\cc}$
for two-sided cells $\cc$ of $W$, and $1_{J_u}=\sum_{d}t_d$ is a sum of
orthogonal idempotents over the subset $\mathcal{D}\cap\cc$ of the finite
set $\mathcal{D}\subset W$ of distinguished involutions. In \textit{op. cit.}, an injective morphism of $\mathcal{A}$-algebras
\[
\phi^\circ\colon\HH\to J\otimes\mathcal{A}.
\]
is defined, with the property that 
if $a\colon W\to\N$ is Lusztig's $a$-function, then
\begin{equation}
\label{eqn phi upper-triangular}
\phi^\circ(C_w)\in\spn{\sets{t_z}{a(z)\geq a(w)}}.
\end{equation}
\begin{dfn}
We set $\phi:=\phi^\circ\circ{}^\dagger(-)$.
\end{dfn}
For any associative algebra $R$, we write $\hhh(R)=R/[R,R]$.

At several points we will want to specialize
$\bq$ to some $q>1$ so as to use the tools of harmonic 
analysis. If $q$ is the cardinality of the residue field of $F$, then $\HH|_{\bq=q}$ is the Iwahori-Hecke algebra of $G$,
and $J$ will be a subalgebra of the Harish-Chandra Schwartz algebra, as 
shown in \cite{BK} and \cite{Plancherel}.

Given a Kazhdan-Lusztig triple $(u,s,\rho)$, for $u\in G^\vee$  
unipotent, $s\in\Zur$ semisimple, and $\rho$ an irreducible representation of 
$\pi_0(Z_{G^\vee}(u,s)/Z(\Gd))$, we write $K(u,s,\rho)$ for the corresponding standard
$H$-module, and $E(u,s,\rho)$ for the corresponding simple $J$-module following
\cite{KLDeligneLanglands} and \cite{cellsIV}. We have 
$E(u,s,\rho)|_H=K(u,s,\rho)$ whenever $q>1$ \cite{BK}.
By \cite{cellsIV}, the two-sided cells
are in bijection with unipotent conjugacy classes in $u\in G^\vee$, and 
we write $\cc(u)$ for the corresponding cell. Moreover, we
have $a(\cc)=\dim\mathcal{B}^\vee_u$, where $\mathcal{B}^\vee_u$ is 
the Springer fibre, and $u'\subset\overline{u'}$ if and only if
$\cc(u')\leq_{LR}\cc(u)$ \cite{BezPerv}, which implies $a(\cc(u))\leq a(\cc(u'))$. In this way we get a partial order on standard $H$-modules
compatible with the $a$-function. By definition, we have $J_{\cc(u)}E(u',s,\rho)=0$ unless $u'=u$.
\subsubsection{Matrix Paley-Wiener theorems}
Finally, we recall the definitions of the rings $\Ee^I$, $\Ee_J=\Ee_J^I$ and $\Ee_t^I$ in \eqref{eqn BK summary diagram}. 
Let $\mathrm{Forg}$ and $\mathrm{Forg}_t$ be the forgetful functors from the categories of admissible
$H$-modules and tempered $H$-modules to $\C$-vector spaces, respectively. Let $\Ee$ (resp. $\Ee_t$) be the subring of 
$\End_\C(\mathrm{Forg})$ (resp. $\mathrm{Forg}_t$) consisting of endomorphisms invariant under a compact open
subgroup of $G$ and depending algebraically (resp.
smoothly) on the unramified character (resp. unitary unramified character) $\nu$ at 
$\pi=i_P^G(\sigma\otimes\nu)$, in the sense of Section \ref{subsection Regular functions of unramified characters}, where $\sigma\in\Ee_2^I(M_P)$ is an $I$-spherical representation of $M_P$. By \cite[Thm. 25]{Bernstein}  (resp. \cite[Section 9]{SchZnk}) the outer maps $f\mapsto\pi(f)$ in \eqref{eqn BK summary diagram} are isomorphisms.
Finally, let $\Ee_J\subset\Ee_t$ be the subring of endomorphisms 
extending to rational functions of $\nu$ without poles at $\nu$ such that $\nu^{-1}$ is non-strictly 
positive, in the sense of \cite[Section 1.7]{BK}, and write $\Ee_{J_u}= 1_{J_u}\Ee_{J} 1_{J_u}$.
\subsection{Equivariant $K$-theory of finite sets}
\label{subsection Equivariant K-theory of finite sets}
Let $\Zz$ be a complex group with reductive identity component $\Zz^\circ$, and $Y$ be a finite $\Zz$-set
with action through $\Gamma=\pi_0(\Zz/Z(\Zz))$. For now we assume that $\Gamma$ is
abelian. In practice we will have $\Zz=\Zur$, the reductive quotient of the centralizer of a unipotent 
$u\in G^\vee$. Then we define
\[
K_\Zz(Y)=K_0(\Coh_\Zz(Y))\otimes_\Z\C
\]
and
\[
K_\Zz(Y\times Y)=K_0(\Coh_\Zz(Y\times Y))\otimes_\Z\C,
\]
where $\Zz$ acts diagonally on the product. Recall that $K_\Zz(Y\times Y)$ is naturally a 
ring under the convolution operation $\star$, and $K_\Zz(Y)$ is naturally a module over it. Both
are modules over
\[
K_\Zz(\pt)=R(\Zz)=K_0(\Rep(\Zz))\otimes_\Z\C=\Oo(\Zz)^\Zz=\Oo(\Zz\git\Zz),
\]
where the GIT quotient is taken with respect to the conjugation action.

Lusztig conjectured in \cite{cellsIV} that $J_u$ was isomorphic to $K_\Zz(Y\times Y)$ for 
some $\Zz$ and $Y$, and classified in [\textit{op. cit.}, Section 10.3] the simple 
$K_\Zz(Y\times Y)$-modules $E_{s,\rho}$. They are parameterized by pairs $(s,\rho)$ where $s\in \Zz$ is 
semisimple, and $\rho$ is a simple representation of $\pi_0(Z_{\Zz}(s)/Z(\Zz))$ occurring in the 
permutation representation $\C[Y]$. 

In our setting, with $\Gamma$ abelian and usually acting transitively on $Y$, so that
$Y^s=Y$ if $Y^s\neq \emptyset$, the $E_{s,\rho}$ are readily described: For 
$s\in\Zz$ a semisimple element such that $Y^s\neq\emptyset$, the specialization of 
$K_\Zz(Y\times Y)\otimes_{R(\Zz)}\C_s$ acts on $\C[Y]$ in the obvious way. We denote this module $E_s$. 
The permutation action of $\pi_0(Z_{\Zz}(s)/Z(\Zz))$ on $E_s$ commutes with the $K_{\Zz}(Y\times Y)$-
action, and in the natural basis of $\C[Y]$ consisting of functions $\mu_\rho\colon Y\to \C$ transforming 
by some representation $\rho$ of $\Gamma$ occurring in $\C[Y]$, we have 
\[
E_s=\bigoplus_{\rho_1\subset\C[Y]}E_{s,\rho_1},
\]
where $\rho_1$ is a simple $\pi_0(Z_{\Zz}(s)/Z(\Zz))$-representation, and
\[
E_{s,\rho_1}=\spn{\sets{\mu_{\rho}}{\rho\Restriction_{\pi_0(Z_{\Zz}(s)/Z(\Zz))}=\rho_1}}.
\]
\subsubsection{Equivariant $K$-theory of the square of a finite set}
\label{subsubsection Equivariant K-theory of square of a finite set}
First assume that $\Zz=\Gamma$ is finite abelian and let $y_1\in Y$. 
Then we have the sheaf $\Oo_{\Gamma\cdot (y_1,\gamma y_1)}$ in
$\Coh_\Gamma(Y\times Y)$, the line bundle
supported on the orbit of $y_1$, with $\Gamma$-equivariant structure defined by putting the trivial 
representation of the stabilizer on each fibre.
\begin{lem}
\label{lem orthogonal idempotents for unextended set}
\begin{enumerate}
\item[(a)] 
Let $Y\owns y_1$ be a transitive $\Gamma$-set for a finite abelian group $\Gamma$, such that the permutation representation $\rho_{\mathrm{perm}}$ of $\Gamma$ decomposes into 
characters as $\C[Y]=\bigoplus_\rho\rho$. Then in $K_\Gamma(Y\times Y)$, the classes
\[
t_\rho:=\frac{1}{\# Y}\sum_{\gamma\in\Gamma}\rho(\gamma^{-1})[\Oo_{\
\Gamma\cdot (y_1,\gamma y_1)}]
\]
give a system of orthogonal idempotents in $K_\Gamma(Y\times Y)$
summing to the identity endomorphism $\Oo_{\Gamma\cdot(y_1,y_1)}$ of $K_\Gamma(Y)$.
\item[(b)]
For all $\rho$, $\rank(t_\rho, \C[Y])=1$.
\end{enumerate}
\end{lem}
\begin{proof}
Identify $\Gamma/S\simeq Y$
for $S=\stab{\Gamma}{y_1}$ and view $\rho_{\mathrm{perm}}$ as a faithful $\Gamma/S$-representation. The $\Gamma/S$-orbits in $Y\times Y$ are each of the form
\[
\Gamma\cdot(y_1,\gamma y_1)=\sets{(gy_1, g\gamma y_1)}{g\in\Gamma}
=\sets{(gy_1, \gamma g y_1)}{g\in\Gamma}
\]
for $\gamma\in\Gamma$. We have
\[
[\Oo_{\Gamma\cdot(y_1,\gamma y_1)}]\star[\Oo_{\Gamma\cdot(y_1,\gamma' y_1)}]
=[\Oo_{\Gamma\cdot(y_1,\gamma\gamma' y_1)}].
\]
Then the coefficient of
$[\Oo_{\Gamma\cdot(y_1,\gamma y_1)}]$ in 
$t_{\rho_1}\star t_{\rho_2}$ is
\[
\frac{1}{(\#Y) ^2}
\sum_{\sets{(\gamma_1,\gamma_2)}{\gamma_1\gamma_2=\gamma}}\rho_1(\gamma_1^{-1})\rho_2(\gamma_2^{-1})
=
\frac{1}{(\#Y) ^2}
\rho_1(\gamma^{-1})\sum_{\gamma_2}\rho_1\left(\gamma_2\right)\rho_2(\gamma_2^{-1})=
\begin{cases}
\frac{1}{\# Y}\cdot\rho_1(\gamma^{-1})&\text{if}~\rho_1\simeq\rho_2\\
0& \text{otherwise}
\end{cases}.
\]
Therefore the $t_\rho$ are orthogonal idempotents. The coefficient of 
$[\Oo_{\Gamma\cdot(y_1,\gamma y_1)}]$ in $\sum_{\rho}t_\rho$ is 
\[
\sum_{\rho}\rho(\gamma^{-1})=\rho_{\mathrm{perm}}(\gamma^{-1})=
\begin{cases}
1&\text{if}~\gamma=1\\
0& \text{otherwise}
\end{cases}.
\]
Finally, we decompose $\rho_{\mathrm{perm}}$. 
Following \cite[Section 10]{cellsIV}, given a function $\mu\colon Y\to \C$  
transforming according to a constituent $\rho'$ of $\C[Y]$, we have
\[
t_{\rho}\mu(x)=
\frac{1}{\# Y}\sum_{\gamma}\rho(\gamma^{-1})\left([\Oo_{\Gamma\cdot(y_1,\gamma y_1)}]\mu\right)(x)
=
\frac{1}{\# Y}\sum_{\gamma}\rho(\gamma^{-1})\mu(\gamma x)=
\frac{1}{\# Y}\mu(x)\sum_{\gamma}\rho(\gamma^{-1})\rho'(\gamma).
\]
In particular, for $\rho'\neq\rho$, we have $t_{\rho}\mu=0$, and for any function $\mu$, we clearly have
\begin{equation}
\label{eqn trho mu transforms according to rho}
(t_{\rho}\mu)(\gamma x)=\rho(\gamma)(t_{\rho}\mu)(x).
\end{equation}
Thus $t_\rho$ annihilates all $\rho'$-isotypic parts of $\C[Y]$
for $\rho\neq\rho'$, and has image contained in the $\rho$-isotypic part of 
$\C[Y]$. As $\Gamma$ is finite abelian and the action is transitive, the permutation representation 
decomposes with multiplicity one. As $t_\rho$ is a projector onto the $\rho$-isotypic part, (b) follows.
\end{proof}

Now we consider the case $\Gamma=\Sn_3$. There are three options for transitive $\Gamma$-sets: $Y=\Gamma$, $\#Y=3$ with the permutation action, or $\#Y=2$ or $\#Y=1$, which are dealt with already, as $\Gamma$ acts through an abelian quotient.

First suppose $Y=\Gamma$, so that $Y^\gamma\neq\emptyset$ if and only if $\gamma=1$, in which case $\C[Y]$ is the regular representation of $\Gamma$. Let $\std$ be the irreducible two-dimensional representation. In this case, the classes that will yield our idempotents are the Young symmetrizers, namely, choosing $y_1\in Y$,
\begin{multline*}
t_\triv=\sum_{\sigma\in\Gamma}[\Oo_{\Gamma\cdot(y_1,\sigma)}], ~ t_{\sgn}=\sum_{\sigma\in\Gamma}\sgn(\sigma)[\Oo_{\Gamma\cdot(y_1,\sigma)}],\\
t_{\std}=\frac{1}{3}\left([\Oo_{\Gamma\cdot(y_1,y_1)}]+[\Oo_{\Gamma\cdot(y_1,(12)y_1)}]-[\Oo_{\Gamma\cdot(y_1,(132)y_1)}]-[\Oo_{\Gamma\cdot(y_1,(13)y_1)}]\right)
\end{multline*}
and $t_{\std'}=[\Oo_{\Delta}]-t_\triv-t_\sgn-t_2$. Each of these classes yields an idempotent in $K_{\Gamma}(Y\times Y)$ such that $t_\rho$ acts on $\HomOver{\Gamma}{\C[\Gamma]}{\rho'}$ as a rank 1 idempotent if $\rho=\rho'$, and by zero otherwise, by the proof of Lemma \ref{lem orthogonal idempotents for unextended set}.

Finally, if $\#Y=3$, then either $Y^\gamma=Y$ and $\C[Y]=\std\oplus\triv$, or $Y^\gamma$ is a point. Let $O$ denote the off-diagonal orbit, and $\Delta$ the diagonal orbit in $Y\times Y$. In the first case, write $E_{s,\triv}=\C\mu_\triv$ for the constant function $\mu_\triv\colon Y^\gamma=Y\to\C$, and likewise $E_{s,\std}=\C\mu_{\std}$. Then $\mathrm{Tr_\gamma}([\Oo_O])\mu_{\std}=-\mu_{\std}$ and $\mathrm{Tr_\gamma}([\Oo_O])\mu_\triv=2\mu_\triv$, where $\mathrm{Tr_\gamma}([\Ff])\in\C[Y^\gamma]$ denotes the image of $\Ff$ in the specialization at $\gamma$. As $\mathrm{Tr}_\gamma([\Oo_\Delta])=\id$, we have that the required rank 1 idempotents are 
\[
\frac{1}{3}\left([\Oo_\Delta]+[\Oo_{O}]\right),~~\frac{2}{3}[\Oo_\Delta]-\frac{1}{3}[\Oo_O].
\]
As $O\cap Y^\gamma=\emptyset$ when $\gamma$ is a transposition, we also obtain the needed idempotent for this case, when $Y^\gamma$ is a point.

Now we remove the finiteness condition on $\Zz$.
\begin{lem}
\label{lem trho act with rank 1 or 0}
Suppose that $Y$ is a transitive $\Zz$-set, where $\Gamma=\pi_0(\Zz)$ is finite abelian or is $\Sn_3$. Let 
$E_{s,\rho_1}$ be a simple $K_\Zz(Y\times Y)$-module such that $t_{\rho}E_{s,\rho_1}\neq 0$. Then $t_\rho$ acts 
as a rank 1 idempotent on $E_{s,\rho_1}$ if $\Gamma$ is abelian, and also if $\Gamma=\Sn_3$ and $\#Y\neq 3, 6$. If $\# Y=3,6$, we can find still find a family of orthogonal rank $1$ idempotents summing to the identity.
\end{lem}
\begin{proof}
First assume that $\Gamma$ is abelian.

As $s$ acts via $\Gamma$, either $Y^s=\emptyset$ or $Y^s=Y$. By assumption, 
we are not in the first case.

Now, the action of the centralizer of $s$ factors through 
\[
\pi_0(Z_{\Zz}(s)/Z(\Zz))\to\Gamma;
\]
let $\Gamma_1\subset\Gamma$ denote the image.

For any $\mu\colon Y\to\C$ transforming under $\Gamma_1$ by $\rho_1$, we have
\begin{align*}
(t_{\rho}\mu)(x)
&=
\frac{1}{\# Y}
\sum_{\gamma\in\Gamma}\rho(\gamma^{-1})\left([\Oo_{\Gamma(y_1,\gamma y_1)}]\mu\right)(x)
\\
&=
\frac{1}{\# Y}
\sum_{\dot{\gamma}\in\Gamma_1\rquotient\Gamma}\sum_{\gamma_1\in\Gamma_1}
\rho(\dot{\gamma}^{-1}\gamma_1^{-1})\mu(\gamma_1\dot{\gamma} x)
\\
&=
\frac{1}{\# Y}
\sum_{\dot{\gamma}\in\Gamma_1\rquotient\Gamma}\sum_{\gamma_1\in\Gamma_1}
\rho(\dot{\gamma}^{-1})\left(\rho(\gamma_1^{-1})\rho_1(\gamma_1)\right)\mu(\dot{\gamma} x)
\\
&=
\begin{cases}
\frac{\#\Gamma'}{\# Y}\sum_{\dot{\gamma}\in\Gamma_1\rquotient\Gamma}\rho(\dot{\gamma}^{-1})\mu(\dot{\gamma}x)& \text{if}~\rho~\text{restricts to}~\rho_1
\\
0&\text{otherwise}
\end{cases}.
\end{align*}
%
(Note that in the first case, the function $\dot{\gamma}\mapsto\rho(\dot{\gamma}^{-1})\mu(\dot{\gamma}x)$ is well-defined.)
Moreover, if $\rho$ does restrict to $\rho_1$, then as in \eqref{eqn trho mu transforms according to rho}, we have that 
$t_\rho\mu$ transforms under all of $\Gamma$ via $\rho$.
Thus
\[
E_{s,\rho_1}=\spn{\sets{\mu_{\rho}}{\rho\Restriction_{\Gamma_1}=\rho_1}},
\]
and
\[
\sets{t_\rho}{t_\rho E_{s,\rho_1}\neq 0}=\sets{t_\rho}{\rho\Restriction_{\Gamma_1}=\rho_1}.
\]
In particular, $\mathrm{rank}(t_\rho, E_{s,\rho_1})=1$ if it is nonzero, by 
Lemma \ref{lem orthogonal idempotents for unextended set} (b).

Now assume that $\Gamma=\Sn_3$.
First suppose $Y=3$. Then $Y^s$ is a point or $Y^s=Y$. 
If $Y^s$ is a point then $t_d$ itself has rank 1. If $Y^s=Y$, then $\C[Y]$ can decompose under $\pi_0(Z_{G^\vee}(u,s))\to\Sn_3$ as either $\triv\oplus\std$, the regular representation of $\An_3$, $\triv\oplus\triv\oplus\sgn$, or $\triv^{\oplus 3}$. If $\rho_1=\std$ or in the $\An_3$ case, then $t_d$ itself is again rank $1$, $\rho_1=\sgn$. If $\mathrm{rank}\pi(t_d)=2$, then $Y^s=Y$, $\pi_0(Z(u,s))=\Z/2\Z$, and $Y$ is a union of a point and a size two orbit. Taking the basis of functions $Y\to\triv$ given by the constant function and the indicator function of the size two orbit,  the action of the idempotents constructed above is given by the matrices
\begin{equation}
\label{eqn idempotent S3 two triv case}
t_{d,\triv}=\begin{pmatrix}
1&\frac{1}{2}\\
0&0
\end{pmatrix},~~~
t_{d,\std}=\begin{pmatrix}
0&\frac{-2}{3}\\
0&1
\end{pmatrix}.
\end{equation}
Hence these orthogonal idempotents each have rank 1 and give the required decomposition.

If $\mathrm{rank}(t_d, E_{s,\rho_1})=3$, then $\C[Y]=\triv^{\oplus 3}$. In this case, adding the indicator function of a point to the basis used in \eqref{eqn idempotent S3 two triv case} gives a basis in which
\[
\pi(t_{d,\std}-\frac{9}{2}\pi(t_{d,\std}t_{d,\triv})\colon \chi_Y\mapsto 0, \chi_{y_1,y_2}\mapsto \chi_{1,2}-\frac{2}{3}\chi_1, \chi_1\mapsto 0
\]
and
\[
\frac{9}{2}t_{\rho}t_{\triv}\colon \chi_1\mapsto \chi_1, \chi_Y,\chi_{1,2}\mapsto 0
\]
and $t_{d,\triv}$, $t_{d,\std}-\frac{9}{2}\pi(t_{d,\std}t_{d,\triv}$, and $\frac{9}{2}t_{\rho}t_{\triv}$ give three rank 1 orthogonal idempotents at $\chi_0$, hence at all $\chi$. As each has rank 1, and they sum to the identity, they give a decomposition into lines for all $\chi$.

Now suppose that $\# Y_d=6$. In this case $Y_d^s=\emptyset$ or $Y_d^s=Y$. Then we can have $\mathrm{rank}(t_d, E_{s,\rho_1})>2$ in the following cases: 

If the image of $\pi_0(Z_{G^\vee}(u,s))=\Sn_3$, and $\rho_1=\std$, then the two Young symmetrizers given above give the decomposition into lines.

If the image is $\An_3$, then $Y^s$ decomposes as two copies of the regular representation of $\An_3$ and $t_d$ has rank $2$.

If $\rho_1=\triv_{\An_3}$, then the orthogonal idempotents $t_{\triv_{\Sn_3}}$ and $t_{\sgn}$ both have rank $1$, as required. If $\rho_1$ is nontrivial, then as $\std|_{\An_3}$ decomposes with multiplicity one, the orthogonal idempotents $t_{\std}$, $t_{\std'}$ each have rank $1$.

If the image of $\pi_0(Z(u,s))$ is $\Z/2\Z$, then $\C[Y]=\triv^{\oplus 3}\oplus\sgn^{\oplus 3}$ and $\mathrm{rank}(
t_d, E_{s,\rho_1})=3$.
In either case each of $t_{\triv}$, $t_{\std}$, $t_{\std'}$, respectively $t_{\sgn}$, $t_{\std}$, and $t_{\std}$ have rank $1$, as $\std$ restricts to $\triv\oplus\sgn$.

If the image of $\pi_0(Z(u,s,))$ is trivial, then $\mathrm{rank}\pi(t_d)=6$ and elements of $t_dJ_ut_d$ act as $6\times 6$-matrices. In this case $t_{\std}$ and $t_{\std'}$ can be diagonalized by integer matrices, and we obtain again a full family of idempotents.
\end{proof}

\subsubsection{Equivariant $K$-theory of the square of a centrally-extended finite set and $J$}
\label{subsubsection Equivariant K-theory of square of centrally extended}
Bezrukavnikov and Ostrik proved a weak version of Lusztig's Conjecture 10.5 of
\cite{cellsIV} on the structure of ring $J$. In their results \cite{BO}, the finite set $Y$ is relaxed to a 
\emph{centrally-extended set} $\Y$ in the following sense.
\begin{dfn}
\begin{enumerate}
\item[(a)]
Given a finite $\Zz$-set $Y$ with $\Zz^\circ$ reductive, the structure of a \emph{centrally extended} 
$\Zz$-set $\Yext$ on $Y$ is the data of a given central extension
\begin{center}
 \begin{tikzcd}
 1\arrow[r]&\Gm\arrow[r]&\widehat{\stab{\Zz}{y}}\arrow[r]&\stab{\Zz}{y}\arrow[r]&1
 \end{tikzcd}
 \end{center} 
for each $y\in Y$, equivariant under the action of $\Zz$ in the sense that for all $g\in \Zz$ we are provided an isomorphism $i_y^g$ such that we have a commutative diagram
\begin{center}
 \begin{tikzcd}
 1\arrow[r]&\Gm\arrow[r]\arrow[d, "\id"]&\widehat{\stab{\Zz}{y}}\arrow[d, "i_y^g"]\arrow[r]&\stab{\Zz}{y}\arrow[r]\arrow[d, "C_g"]&1\\
  1\arrow[r]&\Gm\arrow[r]&\widehat{\stab{\Zz}{gy}}\arrow[r]&\stab{\Zz}{gy}\arrow[r]&1,
 \end{tikzcd}
 \end{center} 
where $C_g$ is conjugation by $g$. We further require that $i^{g'g}_y=i^{g'}_{gy}\circ i^g_{y}$, 
and that $i_y^g=C_g$ if $g\in\stab{\Zz}{y}$.
\end{enumerate}
\item[(b)]
Twisting the inclusion of $\Gm$ by $z\mapsto z^{-1}$ yields
the \emph{opposite centrally-extended set} $\Yext^{\opp}$.
\item[(c)] 
A $\Zz$-\emph{equivariant sheaf on} $\Yext$ is the data of 
\begin{enumerate}
\item[(i)] 
A sheaf $\Ff$ of finite-dimensional $\C$-vector spaces on $Y$ with a projective $\Zz$-equivariant structure;
\item[(ii)]
For all $y\in Y$, an action of the central extension $\widehat{\stab{\Zz}{y}}$ on $\Ff_y$ such that
$\Gm$ acts by the identity character.
\end{enumerate}
\end{dfn}
We denote
$\Rep^1(\widehat{\stab{\Zz}{y}})$ the category of $\widehat{\stab{\Zz}{y}}$-representations
satisfying (ii) and $R^1(\widehat{\stab{\Zz}{y}})$ its complexified Grothendieck group.
If the central extensions are all split, then the data of an equivariant sheaf 
on $\Y$ is the just the data of a $\Zz$-equivariant sheaf on the usual set $Y$. 
\begin{dfn}
\label{dfn product with opp}
If $1\to\Gm\to\hat{\mathcal{H}}_i\overset{\pi_i}{\to}\mathcal{H}_i\to 1$ are two central extensions, then the \emph{product of} $\hat{\mathcal{H}}_i$ \emph{with \textbf{the opposite extension}} 
$\hat{\mathcal{H}}_2^{\opp}$ is 
\[
\begin{tikzcd}
1\arrow[r]&\Gm\arrow[r, "\iota"]&\hat{\mathcal{H}}_1\times_\mathcal{H}\hat{\mathcal{H}}_2/\Delta\Gm\arrow[r, "\pi"]& H\arrow[r]& 1,
\end{tikzcd}
\]
where $\iota(z)=[(z,1)]$.
\end{dfn}
This defines a centrally-extended structure on $\Yext\times\Yext^{\opp}$ with 
$\mathcal{H}=\stab{\Zz}{(y_1,y_2)}$ and $\mathcal{H}_i=\stab{\Zz}{y_i}$.
Note that using the opposite extension means that we take the quotient by the diagonal as opposed
to antidiagonal copy of $\Gm$. These notions in hand, we  can recall
\begin{theorem}[\cite{BO}, \cite{BL}]
\begin{enumerate}
\item[(a)] 
There is a finite centrally-extended set $\Yext_u$ and an isomorphism of based 
rings
\[
J_u\to K_{\Zur}(\Yext_u\times\Yext_u^{\opp}),
\]
where $\Yext_u/\Zur$ is in bijection with $\mathcal{D}\cap\cc$ \cite{BO}. 
\item[(b)]
On may choose each $\Yext_u$ to have trivial $Z(\Gd)$-action \cite{BL}.
\end{enumerate}
\end{theorem}
We will always work with $Y=\coprod_uY_u$ as in (b).
The next lemma implies that central extensions appear only away from the diagonal in 
$Y_u\times Y_u$. In particular, they do not interfere with computing traces.
\begin{lem}
\label{lem opp-square of transitive set is trivial}
If $\Y$ is a transitive centrally-extended $\Zz$-set, then $\Y\times \Y^{\opp}$ with the product centrally 
extended structure has no nontrivial central extensions.
\end{lem}
\begin{proof}
For $(y_1,y_2)\in Y\times Y$, let $\Zz_i$ be the stabilizer of $y_i$ and let $\hat{\Zz}_i$ be the 
given 
central extension of $\Zz_i$. Let $y_1=gy_2$ for some $g\in \Zz$, so that 
$i_g\colon\hat{\Zz_1}\to\hat{\Zz_2}$ is an isomorphism restricting to the 
identity on the central copies of $\Gm$. According to Definition \ref{dfn product with opp} and the 
short five lemma, it suffices to show that the diagram 
\[
\begin{tikzcd}
1\arrow[r]&\Gm\arrow[d, "\id"]\arrow[r, "\iota"]&\hat{\Zz}_1\times_{\stab{\Zz}{(y_1,y_2)}}\hat{\Zz}_2/\Delta\Gm\arrow[r,"\pi"]\arrow[d, "\varphi"]&\stab{\Zz}{(y_1,y_2)}\arrow[d, "\id"]\arrow[r]&1\\
1\arrow[r]&\Gm\arrow[r]&\stab{\Zz}{(y_1,y_2)}\times\Gm\arrow[r]&\stab{\Zz}{(y_1,y_2)}\arrow[r]&1
\end{tikzcd}
\]
commutes, where $\varphi$ is induced by the map 
\[
(h_1,h_2)\mapsto (\pi_1(h_1),h_1i_g^{-1}(h_2^{-1})).
\]
Note that 
\[
\pi_1(h_1i_g^{-1}(h_2^{-1}))=\pi_1(h_1)\pi_1(i_g^{-1}(h_2^{-1}))=\pi_1(h_1)\pi_2(h_2)^{-1}=1
\]
by definition of the fibre product. Hence $h_1i_g^{-1}(h_2^{-1})\in\Gm\subset\hat{\Zz}_1$. 
Using this, one can check that $\varphi$ is a group homomorphism such that the diagram commutes.
\end{proof}
\subsection{Rank 1 idempotents}
\label{subsection rank 1 idempotents}
Now we return to the study of $J_u$, with $\Zz=\Zur$ such that $\Gamma=\pi_0(\Zur/Z(\Gd))$ is abelian or is $\Sn_3$.
\begin{lem}
\label{cor classical idempotents are integral}
Applying the constructions of Section \ref{subsubsection Equivariant K-theory 
of square of centrally extended} yields elements 
$t_{\omega d,\rho}\in J_u$, $\rho\subset\C[Y_d]$, $\omega\in Z(G^\vee)$ defined over 
$\Z\left[\frac{1}{\# Y_u}\right]$. 
\end{lem}
\begin{proof}
We have 
\[
J_u=K_{Z_{G^\vee}(u)^{red}}\left(\Y_u\times \Y^{\opp}_u\right)
\]
for some centrally-extended set $Y_u$. By Lemma \ref{lem opp-square of transitive set is trivial}, 
for each orbit $Y_d\subset Y_u$, the set $Y_d\times Y_d^{\opp}$ has no nontrivial central extensions. 
Hence in each subring $t_dJ_ut_d$ of $J_u$ the elements $t_{d,\rho}$ of Lemma \ref{lem orthogonal idempotents for unextended set} are defined, with $\Zz$-equivariant structure pulled back from their 
$\Gamma$-equivariant
structure. For $\omega\in Z(\Gd)$, put $t_{\omega d,\rho}=\phi(T_\omega)\star t_{d,\rho}$.

When $\Gamma$ is abelian, it $2$-torsion, as will be recalled in Sections 
\ref{subsubsection unipotent centralizers types BCD} and 
\ref{subection flatness of invariant functions on characters exceptional}. 
As $\phi(T_\omega)$ is defined over $\Z$, all coefficients appearing in $t_{\omega d,\rho}$ lie in 
$\Z\left[\frac{1}{\# Y_u}\right]$.
By inspection, this also holds for the case $\Gamma=\Sn_3$.
\end{proof}
We stress that whenever the action on $Y_u$ is nontrivial, $t_{\omega d,\rho}$ is not the class of 
a subobject of the bundle corresponding to $t_d$, which is in fact obviously simple.
\begin{lem}
\label{lem projection of simple is simple}
Let $E$ be a simple $J$-module. If $t_dE\neq 0$, then $t_dE$ is a simple $t_dJt_d$-module.
\end{lem}
\begin{proof}
If $E'$ is a simple proper $t_d Jt_d$-submodule of $t_dE$, then consider $JE'\neq 0$. As we have
$t_dJE'\subsetneq t_dE\neq 0$, we have $JE'\subsetneq E$.
\end{proof}
\begin{lem}
\label{lem pi(trho) has rank at most 1}
Fix $M_P\subset P$, $\sigma\in\mathcal{E}_2(M_P)$, and  $t_{d,\rho}$ in $t_dJt_d$ for some $d$.
If $i_P^G(\sigma\otimes\nu)(t_{d,\rho})$ is defined and nonzero at some
$\nu$, then the function 
$\trace{i_P^G(\sigma\otimes\nu)}{t_{d,\rho}}$ extends to a regular function of
$\nu$ with constant value 1. In particular, in this case the 
operator $i_P^G(\sigma\otimes\nu)(t_{d,\rho})$ is an idempotent of rank
$1$ wherever it is defined.
\end{lem}
\begin{proof}
Fix $d$, and, by injectivity of $\eta$, $(M_P,\sigma)$ such that if $\pi=i_P^G(\sigma\otimes\nu^{-1})$, 
then $\pi(t_d)\neq 0$ for all non-strictly positive $\nu$.
Now, for generic non-strictly positive $\nu$, $\pi$ is a simple representation of $G$, and hence a 
simple $J$-module. Fix such a character $\nu_0$ and the corresponding representation $\pi$.
By Lemma \ref{lem projection of simple is simple}, $\pi(t_d)\pi$ is a simple $t_dJt_d$-module.
Thus $\pi(t_d)\pi=E_{s,\rho'}$ as a $t_dJt_d$-module in the 
notation recalled in Section \ref{subsubsection Equivariant K-theory of square of a finite set}.
Therefore if $\pi(t_{d,\rho})\pi= t_{d,\rho} E_{s,\rho'}\neq 0$, we have $\rank(\pi(t_{d,\rho}))=1$
by Lemma \ref{lem trho act with rank 1 or 0}.

This is thus the situation for $i_P^G(\sigma\otimes\nu_0^{-1})$ for some character $\nu_0$: a subset
of the $t_{d,\rho}$ (in case $\Gamma=\Sn_3$, a collection of the idempotents from Section \ref{subsubsection Equivariant K-theory of square of a finite set} and the proof of Lemma \ref{lem trho act with rank 1 or 0}) act as orthogonal idempotents with rank one.
But $\trace{\pi}{t_{d,\rho}}$ is constant in $\nu$ by
\cite[Lemma 5]{Plancherel} and is equal to $\rank(\pi(t_{d,\rho}))$ whenever the latter is defined. Thus $\trace{\pi}{t_{d,\rho}}=1$ all $\nu$ and the claim follows.
\end{proof}
\begin{rem}
\label{rem s'=s observation and rank bound}
Considering central characters, we see that if $\pi=K(u,s'',\rho'')$ in the above proof, then $s''=s$. 
Additionally, we get that $\rank(\pi(t_d))\leq \# Y_d$.
\end{rem}
\subsection{Regularity of matrix coefficients}
\label{subsection Regularity of matrix coefficients}
In this section, we show that matrix coefficients with respect to a basis defined via the idempotents
$t_{d,\rho}$ depend in fact algebraically on the unramified character $\nu$. The proof is similar to 
proof of regularity of the function $\trace{\pi}{t_w}$ established in \cite{Plancherel}.
\subsubsection{Labels for intertwining operators}
\label{subsection labels for intertwining operators}
If $\mathbf{P}$ is a parabolic subgroup of $\G$ corresponding to a subset $\Delta_{\mathbf{P}}$ of 
$\Delta$, we will write $\mathbf{M}_\mathbf{P}$ for its Levi quotient. Given a Levi subgroup $M$ of $G$,
write $\X(M)$ for the group of unramified characters of $M$. By \cite[Cor. 3]{Howlett}, write
\[
W_{M_P}=W_M=\sets{w\in W_f}{w(\Delta_P)=\Delta_P}=N_G(A_{M})/M=N_G(M)/M=N_{W_f}(W_P)/W_P,
\]
where $W_P$ is the parabolic subgroup of $W_f$ corresponding to $\mathbf{P}$. 
The group $W_{M_P}$ acts on $\Ee_2(M_P)$ and every $w\in W_{M_P}$ labels a meromorphic family of 
intertwining operators of tempered parabolic inductions from $P$ to $G$. 
In general $W_M$ is not a Coxeter group, but has an internal semidirect product decomposition 
\cite[Cor. 7]{Howlett}
\[
W''_{M_P}\rtimes V_{M_P}=W_{M_P},
\]
where $W''_{M_P}$ is a Weyl group \cite[Thm. 6]{Howlett}. 
\subsubsection{Invariance of matrix coefficients}
Let $u$ be a unipotent conjugacy class and $\cc=\cc(u)$ be the corresponding two-sided cell. Then we have
\[
1_{J_u}=1_{\mathcal{E}_{J,u}}=\sum_{\substack{d\in\cc \\ \rho\subset\C[Y_d]}}t_{d,\rho}.
\]
Thus given an $\Ee_{J,u}$-module, say of the form $\pi=i_P^G(\sigma\otimes\nu^{-1})$,
we have a linear isomorphism
\begin{equation}
\label{eqn pi(trho) define rational basis}
\pi^I\to\bigoplus_{\substack{d\in\cc \\ \pi(t_d)\neq 0}}\pi(t_{d})\pi=
\bigoplus_{\substack{d\in\cc \\ \rho\subset\C[Y_d] \\ \pi(t_{d,\rho})\neq 0}}\pi(t_{d,\rho})\pi,
\end{equation}
and by Lemma \ref{lem projection of simple is simple}, each nonzero $\pi(t_d)\pi$ is a simple $t_dJt_d$-module. Hence by Lemma \ref{lem pi(trho) has rank at most 1}, under the isomorphism
$\pi(t_d)\pi\simeq E_{s,\rho'}$, we have $\pi(t_{d,\rho})\pi\simeq\C\mu_\rho$ for a function $\mu_\rho$
on $Y_d$ that is unique up to sign once we demand that $\mu_\rho$ take values in $\{\pm 1\}$. As the space of $t_dJt_d$-module isomorphisms $\pi(t_d)\pi\to E_{s,\rho'}$ is $\C^\times$, we obtain a direct sum decomposition of $\pi^I$ into
lines that depends rationally on $\nu$. For any fixed $\nu_0$ for which this
decomposition exists,  we obtain a basis 
of $\pi$ that is unique up to the obvious action of 
\[
\prod_{d\in\mathcal{D}\cap\cc}\C^\times\times\prod_{\rho\subset\C[Y_d]}\Z/2\Z,
\]
or a signed basis unique up to the action of the first factor. Choosing such 
a basis consisting of vectors $v_{d,\rho}(\nu_0)$ defines vectors
\[
v_{d,\rho}(\nu)=\pi(t_{d,\rho})(\nu)v_{d,\rho}(\nu_0)
\]
depending rationally on $\nu$, such that $\sett{v_{d,\rho}(\nu)}_{d,\rho}$ is 
a basis for all $\nu$ for which all its elements are defined. (In Example
\ref{ex sl2 rep that is not J-module}, we give an instance of this basis
ceasing to be defined.)

We can therefore consider matrix coefficients with respect to the 
$v_{d,\rho}(\nu)$ when these vectors are defined.
Indeed, for any $e\in\mathcal{E}_{J,u}$, and $\pi=i_P^G(\sigma\otimes\nu^{-1})$, we have that 
\[
\pi(e)=\sum_{\rho',\rho\subset\C[Y_d]}\pi(t_{d,\rho'})\pi(e)\pi(t_{d,\rho})
\]
and each $\pi(t_{d,\rho'})\pi(e)\pi(t_{d,\rho})$ has rank at most 1, and is 
determined by its action on a one-dimensional space, by Lemma \ref{lem pi(trho) has rank at most 1}. Hence for $v=v_{d,\rho}(\nu)$, $v'=v_{d',\rho'}(\nu)$, and $e\in t_{d',\rho'}\mathcal{E}t_{d,\rho}$, we have
\begin{equation}
\label{eqn definition of matrix coeffs}
\pi(e)(\nu)v=f(\nu)v'
\end{equation}
for a function $f(\nu)=f_{\pi,e,\rho',\rho}(\nu)$ depending rationally on $\nu$. Thus we obtain a map
\[
\mathcal{E}_{J,u}\to\Mat_{\dim\pi^I}(\C(\X(M))).
\]

We will now show that the $f_{\pi,e,\rho',\rho}$ extend to regular functions of $\nu$. For the proof
it will be helpful to fix realizations of induced representations: we view
$\pi=i_P^G(\sigma\otimes\nu^{-1})$ via the compact picture, in which the vector space 
$V_\pi$ consists of functions $\Phi\colon K\to V_\sigma$, where $V_\sigma$ is the underlying vector space of 
$\sigma$ and $K$ is a maximal compact subgroup of $G$. For each $w\in W_M$, there exists an $M_P$-
intertwining operator 
\[
T(w)\colon\sigma\overset{\sim}{\to}w\cdot\sigma,
\]
corresponding to the automorphism $T_{w'}\mapsto T_{ww'w^{-1}}$, $\theta_\alpha\mapsto \theta_{w(\alpha)}$
of $H_{M_P}$. Consider the intertwining operator 
\[
I_w\colon i_P^G(\sigma\otimes\nu^{-1})\to i_P^G(w\cdot\sigma\otimes w(\nu^{-1}))=\pi_1.
\]
From $I_w$ we obtain the endomorphism $T_wI_w$ of $ i_P^G(\sigma\otimes\nu^{-1})$, where we view $T(w)$ 
as acting by $\Phi\mapsto T(w)\circ \Phi$ for $\Phi\in V_{\pi_1}$, so that 
\[
T_wI_w(\Phi)(k)=T_w\left(I_w(\Phi)(k)\right).
\]
Now we prove regularity.
\begin{lem}
\label{lem regularity of matrix coefficients}
Fix $\pi, e,\rho',\rho$ as above. Then $f_{\pi,e,\rho',\rho}(\nu)$ is a regular function on $\X(M)\git W_M$.
\end{lem}
\begin{proof}
Let $e\in t_{d,\rho}\Ee_{J,u}t_{d,\rho'}$. Write $v(\nu)=v_{d,\rho}(\nu)=\Phi$, 
$v_{d,\rho}(\nu_0)=\Phi_0$, $v'(\nu)=v_{d',\rho'}(\nu)=\Phi'$, $v_{d',\rho'}(\nu_0)=\Phi'_0$, and $f(\nu)=f_{\pi,e,\rho',\rho}(\nu)$.

We first consider only unitary $\nu$ such that $\pi(\nu)$ is 
irreducible. Let $w\in W_M$ and let $I=I_w$ be the intertwining isomorphism 
labelled by $w$. By triviality of the $R$-group, we have $T(w)I_w=\alpha$
for $\alpha\in\C^\times$, so that 
\[
I_w(\Phi)(k)=\alpha \left(T(w)^{-1}\circ \Phi\right)(k),~~\Phi\in V_\pi,~k\in K.
\]
Then applying $I_w$ to the LHS of \eqref{eqn definition of matrix coeffs} gives
\begin{align*}
I_w\left(\pi(e)(\nu)\Phi\right)
&=
\pi(e)(w\nu)\left(I_w\pi(t_{d,\rho}(\nu)\Phi_0\right)
\\
&=
\pi(e)(w\nu)\left(\pi(t_{d,\rho}(w\nu)\left(\alpha T(w)^{-1}\circ \Phi_0\right)\right)
\\
&=
\alpha T(w)^{-1}\circ\left(\pi(e)(w\nu)\left(\pi(t_{d,\rho}(w\nu)\left(\Phi_0\right)\right)\right)\\
&=
\alpha T(w)^{-1}\circ\left(f(w\nu)\Phi'\right).
\end{align*}
Here we used the fact that $M_P$-linearity of $T(w)^{-1}$ implies that post-composition with $T(w)^{-1}$ is 
$G$-linear. Applying $I_w$ to the RHS of \eqref{eqn definition of matrix coeffs} gives
\begin{align*}
f(\nu)I_w\left(\pi(t_{d',\rho'}(\nu)\Phi'_0\right)
&=f(\nu)\pi(t_{d',\rho'}(w\nu)\left(\alpha T(w)^{-1}\circ \Phi'_0\right)
\\
&=
\alpha f(\nu) T(w)^{-1}\circ \Phi'.
\end{align*}
Therefore $f(w\nu)=f(\nu)$ for unitary $\nu$ such that $\pi$ is irreducible, \textit{i.e.}, for an open 
subset of the unitary characters, $f(\nu)=f(w\nu)$. Now the 
same argument as in \cite[Lemma 4]{Plancherel} shows that $f(\nu)=f(w\nu)$ for 
all $\nu\in \X(M)$. Thus $f(\nu)$ is a regular $W_M$-invariant function.
\end{proof}
Thus we have in fact
\begin{equation}
\label{eqn EJu to single Mat definition}
\iota_u\colon\mathcal{E}_{J,u}\to\Mat_{\dim\pi^I}\left(\Oo(\X(M))^{W_M}\right).
\end{equation}
We will abusively refer to the extensions of the matrix coefficients also as matrix coefficients.

By definition of the isomorphism \eqref{eqn pi(trho) define rational basis}, if $\pi(t_{d,\rho})\neq 0$,
then
\begin{equation}
\label{eqn trho maps to monomial matrix}
\iota_u\colon t_{d,\rho}\mapsto\pi(t_{d,\rho})=\diag(0,\dots, 0,1,0,\dots, 0)\in\Mat(\Oo(\X(M))^{W_M}).
\end{equation}
\section{Surjectivity of the Braverman-Kazhdan map}
\label{section surjectivity}
In this section, we prove
\begin{theorem}[\cite{BK}, Theorem 2.4]
\label{thm BK onto}
Let $\G$ be as in Section \ref{subsubsection and affine and asymptotic Hecke algebras}. Then $\eta_u$ is an isomorphism.
\end{theorem}
\subsection{Regular functions of unramified characters}
\label{subsection Regular functions of unramified characters}
Let $\M$ be a Levi subgroup of $\G$ with connected centre $\mathbf{A}_\mathbf{M}$. Write $X_*(\M)$ and $X^*(\M)$ for the $F$-rational cocharacters and characters, respectively, of $\M$, and
\[
M^1=\bigcap_{\chi\in X^*(\M)}\ker |\chi|_F,
\]
where $|\chi|_F(g)=|\chi(g)|_F$. Set $\Lambda(M)=M/M^1$, and recall that there is a short exact sequence
\[
\begin{tikzcd}
1\arrow[r]&X_*(A_M)\arrow[r]& \Lambda(M)\arrow[r]& K\arrow[r]&1
\end{tikzcd}
\]
for some finite group $K$. Dualizing, we obtain the group of unramified characters $\X(M)$ of $M$ is given by
\[
\X(M):=\HomOver{\Grp}{\Lambda(M)}{\C^\times}=\HomOver{\Z}{\Lambda(M)}{\Z}\otimes_\Z\C^\times=\HomOver{\Z}{X_*(\mathbf{A}_{\M})}{\Z}\otimes_\Z\C^\times=X_*(A_M^\vee)\otimes_\Z\C^\times=A_M^\vee(\C).
\]
This upgrades the abstract group $\X(M)$ to a complex algebraic group such that $\X(M)=A_M^\vee$.
%
By definition, the 
group $W_M=N_G(A_M)/M$ defined in Section \ref{subsubsection and affine and asymptotic Hecke algebras}
acts on $A_M$ and $M$, hence on $\X(M)$.

Finally, let $\omega\in\Ee^2(M)$ and recall that $\stab{\X(M)}{\sigma}\subset\X(M)$ is a finite group, making the orbit $\mathfrak{o}$ of $\omega$ a torus with \'{e}tale covering
\begin{equation}
\label{eqn orbit under etale twist is covered by unram characters}
\begin{tikzcd}
1\arrow[r]&\stab{\X(M)}{\sigma}\arrow[r]&\X(M)\arrow[r]&\mathfrak{o}\arrow[r]&1.
\end{tikzcd}
\end{equation}

%
%
\subsection{Construction of modules}
\label{subsection Construction of modules}
By Section \ref{subsection Regularity of matrix coefficients}, matrix coefficients of $\Ee_{J_u}$ are regular functions of $\nu$, but it will
be necessary in the sequel to exhibit certain specific regular functions of $\nu$ as matrix coefficients 
specifically of $J_u$. To this end, we can adapt the construction of $J_u$-modules given in \cite{cellsIV} 
to the centrally-extended case as follows.

In \cite[Section 2.4]{Propp}, Propp constructs a certain cover
$\widetilde{\Zur}\onto\Zur$ such that any projective 
representation of $\stab{\Zur}{(y,y')}$ is a genuine
representation of $\stab{\widetilde{\Zur}}{(y,y')}$.

Let $s\in\Zur$.
Let $\tilde{s}\in\widetilde{\Zur}$
be any lift of $s$, where the action of $\Zur$ on $Y_u$ is pulled back to an action of 
$\widetilde{\Zur}$, and consider the specialization map 
\[
\tilde{J}_u:= K_{\widetilde{\Zur}}(Y_u\times Y_u)\to\C[Y_u^s\times Y_u^s]
\]
to functions on $Y_u^s\times Y_u^s$, sending
\[
\Ff\mapsto\left((y,y')\mapsto\trace{\Ff_{y,y'}}{\tilde{s}}\right).
\]
The construction of the modules $E_{u,\tilde{s},\rho}$ given in \cite[Section 10.3]{cellsIV} now goes through
to exhaust the irreducible $K_{\widetilde{\Zur}}(Y_u\times Y_u)$-modules. Note that $J_u$ is a naturally
a subring of $\tilde{J}_u$.

%
Now put
\[
E(u,s):=\bigoplus_{\rho\in\Irr\left(\pi_0\left(Z_{G^\vee}(u,s)\right)\right)}E(u,s,\rho)
\]
for the direct sum of all $J_u$-modules with parameter $s$ (of course, some $E(u,s,\rho)$ are zero).
By Lemma \ref{lem projection of simple is simple}, either $E(u,s)|_{t_dJt_d}$ 
is zero, which happens by Lemma
\ref{lem opp-square of transitive set is trivial} if and only if $Y_d^s=\emptyset$, or is a 
direct sum of simple $t_dJt_d=K_{\Zur}(Y_d\times Y_d)$-modules. Indeed, each $t_dE(u,s,\rho)$ is a simple $t_dJt_d$-module.
From this perspective, each $E(u,s)|_{t_dJt_d}$ extends to 
a $K_{\widetilde{\Zur}}(Y_d\times Y_d)$-module, hence
$\bigoplus_d E(u,s)|_{t_dJt_d}$ extends to 
a $K_{\widetilde{\Zur}}(Y_u\times Y_u)$-module upon choosing a lift
$\tilde{s}$ of $s$. Restricting this 
extension to $J_u$ equips the vector space $E(u,s)$ with a second, 
\textit{a priori} distinct, $J_u$-action,
which also depends \textit{a priori} on the choice of lift $\tilde{s}$.
This action is now as constructed via the variation of 
\cite[Section 10]{cellsIV} described above. Denote this module 
$E(u,\tilde{s})'$. 
\begin{lem}
\label{lem mc are as implied by Lusztig}
We have $E(u,s)\simeq E(u,\tilde{s})'$ as $J$-modules.
In particular, the latter depends only on $s$, and
the matrix coefficients of any simple $J$-module are of the form implied by 
the above-described variant of the construction in 
\cite[Section 10]{cellsIV}.
\end{lem}
\begin{proof}
By construction and the fact that the diagonal part of $J_u$ is never 
centrally-extended, we have
$E(u,s)|_{t_dJt_d}\simeq E(u,\tilde{s})'|_{t_dJt_d}$
for all $d$, whence the equality $\Theta_{E(u,s)}=\Theta_{E(u,\tilde{s})'}$ 
of Harish-Chandra characters.
Thus $E(u,s)|_H$ and $E(u,\tilde{s})'|_H$ have the same Jordan-Holder factors as $H$-modules. 
In particular, their factors $L(u',s',\rho')$ such that $u'=u$ are the same and all other factors have
$a(u')<a(u)$. 
Therefore it follows from from Theorem 3.2 (a) of \cite{XiJAMS} that the Jordan-Holder factors
of $E(u,s)$ and $E(u,\tilde{s})'$ as $J$-modules agree. As 
$E(u,\tilde{s})'$ is a semisimple $J$-module by construction, the claim follows.
\end{proof}
The presence of central extensions enforces vanishing properties of the matrix coefficients;
see Example \ref{ex SL2 central ext}.
\subsection{Lower modifications of vector bundles}
\label{subsection Lower modifications of vectors bundles}
In this section we explain the proof of Theorem \ref{thm BK onto}.

Fix $u$ and let $M_1,\dots, M_k$ be all the Levi subgroups of $G$, up to association, 
such that $u$ appears in the Kazhdan-Lusztig parameter of an Iwahori-spherical discrete series representation 
of each $M_i$. Let the corresponding families of parabolic inductions be
\[
\pi_{i,j}=i_{P_i}^G(\sigma_{i,j}\otimes\nu),~~~\sigma_{ij}\in\mathcal{E}_2(M_i)/\X(M_i),~~\nu_i\in\X(M_i),~~P_i\supset M_i.
\]
We write $W_i=W_{M_i}$ as in Section \ref{subsection Regularity of matrix coefficients}, in which it was
established that the endomorphisms
\[
\pi_{ij}(t_{d,\rho}),~d\in\mathcal{D}\cap\cc(u),~\rho\subset\C[Y_d]
\]
form a system of rank 1 orthogonal idempotents, and hence define a basis of $\pi_{ij}^I$ depending 
rationally on $\nu$, with the resultant matrix coefficients extending, by Lemma \ref{lem regularity of matrix coefficients}, to define maps
\begin{equation}
\label{eqn nested modification}
\begin{tikzcd}
J_u\arrow[r, "\eta_u", hook]&\mathcal{E}_{J,u}\arrow[r, "\iota_u", hook]&\bigoplus_{i,j}\Mat_{\dim\pi_{ij}^I}\left(\Oo\left(\X(M_i)^{W_i}\right)\right)=:\mMu
\end{tikzcd}
\end{equation}
in $\Coh\left(\Zur\git\Zur\right)$. 
The map $\iota_u$ is injective by definition of $\Ee_{J_u}$. 

First, if $\Zur$ is finite, then we do not need the idempotents $t_{d,\rho}$, and the strategy of 
\cite{BK} goes through verbatim:
\begin{lem}
\label{lem surjectivity for finite centralizers}
Let $u\in \Gd$ be such that $\Zur$ is finite modulo $Z(G^\vee)$. Then $\eta_u$ is surjective.
\end{lem}
\begin{proof}
It suffices to show surjectivity on fibres at unitary unramified characters.
The hypothesis implies the class $u$ is distinguished in all of $[\Gd,\Gd]$, so any Iwahori-
spherical representation $\pi$ of $G$ with $u$ in its parameter is a discrete series 
representation of the group $G$. 

Fix a maximal ideal
$\mM$ of $R(\Zur)$ such that all the representations $\pi_i=\sigma_i\otimes\nu_i$ with 
$\sigma_i\in\Ee_2(G)$ 
and $\nu_i\in\X(G)$ unitary on which $R(\Zur)$ acts via the quotient by $\mM$ are tempered, and 
$J_u\pi_i\neq 0$. We obtain surjections $(J_u)_\mM\onto\bigoplus_i\End(\pi_i)$ and 
$\eta_\mM\colon(\Ee_{J,u})_\mM\onto\bigoplus_i\End(\pi_i)$.
It suffices to show that $\eta_{\mM}$ is injective.
For $f\in\Ee_{J,u}$ a Schwartz function, we have $f\nu\in\Ee_{J,u}$ for any $\nu\in\X(G)$ unitary.
Indeed, $(\sigma\otimes\nu)(f)=\sigma(f\nu)$ for all $\sigma\in\Ee_2(G)$. Now, generically $\eta_\mM$ is 
an isomorphism. Therefore let $\nu'$ and $\mM'$ be such that $\eta_{\mM'}$ is an isomorphism onto
$\bigoplus_i\End(\pi_i\otimes\nu')$. As $(\pi_i\otimes\nu')(f(\nu')^{-1})=\pi_i(f)$, we see that
$f\mapsto (\pi_i(f))_i$ is injective, as required.
\end{proof}

Now we deal with the representations outside the discrete series:
In Section \ref{subsection Regular functions of unramified characters}
we reduce to the case $G$ semisimple, and in Lemma \ref{lem suffices to show surjectivity for one 
member of isogeny class} we show that it suffices to prove surjectivity for a single member of the isogeny
class of $G$. In Sections \ref{subsubsection flatness of O(X(M)/W)} and Proposition \ref{Ju is a vector bundle}, 
we show that the outer terms of \eqref{eqn nested modification} are actually a vector bundle and a maximal Cohen-Macaulay sheaf on the Cohen-Macaulay scheme $\Zur\git\Zur$ if $G^\vee$ is 
classical or adjoint exceptional. By the identification in Section \ref{subsection Regular functions of unramified characters},
this means relating $A_{M_P}^\vee/{W_{M_P}}$ and $\Zur\git\Zur$ 
for $u$ appearing in the parameter of an element of $\Ee_2(M_P)^I$. 
Recall that this happens if and only if $P$ is minimal up to association such that 
that there is a semisimple element $s\in M_P^\vee$ such that $Z_{M_P^\vee}(s)$ is semisimple and 
$u\in Z_{M_P^\vee}(s)$ is distinguished. In Section \ref{subsubsection flatness of O(X(M)/W)}, we prove
\begin{prop}
\label{prop quotient of unramified characters is flat}
Suppose that $u$ appears in the parameter of an element $\sigma$ of $\Ee_2^I(M)$ of $G$. Then $\mathfrak{o}_\sigma\git W_M$ is a 
connected component of $\Zur\git\Zur$.
In particular, $\Oo(\mathfrak{o}_\sigma)^{W_M}$ is a flat $R(Z_{G^\vee}(u)^{red})$-module.
\end{prop}
%

As will be recalled below, the groups $\Zur$ are usually not connected, in which the Cohen-Macaulay schemes $\Zur\git\Zur$ are also not connected; their structure is recalled in general terms in the appendix.

Recall also that even for connected semisimple groups $\Gg$ with universal cover $\tilde{\Gg}$, the natural map 
\begin{equation}
\label{eqn git quotient of universal cover}
\tilde{\Gg}\git\tilde{\Gg}\onto\Gg\git\Gg
\end{equation}
is usually not flat. In particular, there is some Schur multiplier for which the corresponding module $K_0(\Rep^1\tilde{\Gg})$ is not a flat module.
It can happen that for classical groups, $\Zur\git\Zur$ is a disjoint union of affine spaces, so \eqref{eqn git quotient of universal cover} is an open morphism of smooth complex varieties and is flat. Thus in this case, $J_u$ is also a vector bundle on $\Zur\git\Zur$. However, in general we have only
\begin{prop}
\label{Ju is a vector bundle}
If  $G^\vee$ is classical or adjoint exceptional and $u$ is as above, then $J_u$ is a maximal Cohen-Macaulay $R(\Zur)$-module, \textit{i.e.} a maximal Cohen-Macaulay sheaf on each connected component of $\Zur\git\Zur$.
\end{prop}
\begin{proof}
It suffices to show that for every isomorphism class of central extension $\widetilde{\stab{Zur}{y}}$ of the stabilizer in $\Zur$ of some $y\in Y_u$, that $K_0(\Rep^1\widetilde{\stab{Zur}{y}})$ is a maximal Cohen-Macaulay 
sheaf on $\Zur\git\Zur$. For $\stab{\Zur}{y}\git\stab{\Zur}{y}$ this follows from Lemma \ref{lem1} of appendix by Rumynin. The case of centrally-extended stabilizers follows from the existence of Schur covers of every $\Zur$ \cite{Propp} and Lemma \ref{lem3} of the appendix by D. Rumynin.
\end{proof}
The morphism
$\iota_u\circ\eta_u$ is an isomorphism away the locus in $\X(M_1)\sqcup\cdots\sqcup\X(M_k)$ over which the $\pi_{ij}$ become reducible. Indeed,
let $\{\nu_{ij}\}$ be a collection such that all the  $\pi_{ij}$ have the same 
$Z(H)$-character corresponding to a closed point $s$ of $\Zur\git\Zur$. 
By Burnside's theorem, $\iota_u\circ\eta_u$ induces a surjective morphism on 
fibres at $s$. Over the reducible locus, $\im(\iota_u\circ\eta_u|_s)$ is a proper subspace of $\mMu|_{s}$.
For some $u$ and $\Zur$-orbits in $\Y_u$, this locus has codimension at least $2$, giving immediately that $\codim\supp(\Ee_u/J_u)\geq 2$. For other obits, this locus is a union of divisors. In other words,
\begin{prop}
\label{prop Ju saturated lower modification}
Let $G^\vee$ be a classical group or an adjoint exceptional group. 
Then either
\begin{enumerate}
\item[(a)] 
The locus of reducibility of the parabolic inductions $\pi_{ij}$, has codimension at least $2$;
\item[(b)]
The coherent sheaf $J_u$ is a lower modification 
of the vector bundle $\mMu$ on $\Zur\git\Zur$ as in \eqref{eqn nested modification}.
\end{enumerate}
\end{prop}
Next we prove
\begin{prop}
\label{prop Ju and EJu fibre images agree at generic points}
Let $G^\vee$ be a classical group or an adjoint exceptional group with $u$ appearing in the parameter of
an element of $\Ee_2(M)$ for some Levi subgroup $M$ of $G$ with $J_u$-
module $\pi=\pi_{ij}$ as in \eqref{eqn nested modification}.
In the setting of the second case of Proposition \ref{prop Ju saturated lower modification},
let $D_1,\dots, D_k$ 
in $\Zur\git\Zur$ be the irreducible divisors over which
the map 
$J_u\to\Mat_{\dim\pi^I}(\Oo(\X(M))^{W_M})$ induced by $\pi$ is not an isomorphism, \textit{i.e.} over which it 
is a genuine lower modification. Let $\eta_i$ be the generic point of $D_i$ and let $k(\eta_i)$ be the residue 
field. Then
\begin{enumerate}
\item[(a)] 
\[
\im\left(J_{u}|_{\eta_i}\to\mMu|_{\eta_i}\right)=\im\left(\mathcal{E}_{J,u}|_{\eta_i}\to\mMu|_{\eta_i}\right).
\]
That is, the two images of the fibres of $J_u$ and $\mathcal{E}_{J,u}$ at the generic points 
$\eta_i$ are equal in 
\[
\mMu|_{\eta_i}=\bigoplus_{ij}\Mat_{\dim\pi^I_{ij}}(\Oo(\X(M)\git W_M)|_{\eta_i}=\bigoplus_{ij}
\Mat_{\dim\pi^I_{ij}}(k(\eta_i))
\]
\item[(b)]
The scheme-theoretic support of $\mathcal{M}_u/J_u$ is equal to $D_1\cup\cdots\cup D_k$.
\end{enumerate}
\end{prop}
Unlike Proposition \ref{Ju is a vector bundle}, the proofs of Propositions \ref{prop Ju saturated lower modification} and \ref{prop Ju and EJu fibre images agree at generic points} require some casework, as we must understand the vanishing loci of various characters of (projective) representations of stabilizers in $\Zur$.
With the propositions in hand, we are ready for the 

\begin{proof}[Proof of Theorem \ref{thm BK onto}]

Feed the reduction of Section \ref{subsubsection reduction to semisimple groups} to semisimple groups and Propositions \ref{prop quotient of unramified characters is flat}--\ref{prop Ju and EJu fibre images agree at generic points} into
\begin{lem}
\label{lem Dennis}
Let $S$ be a Cohen-Macaulay scheme,
and let $\Ee_2$ be a vector bundle, $\Ff$ be coherent
sheaf, and $\Ee_1$ be a maximal Cohen-Macaulay sheaf on $S$. Suppose that we have
\[
\Ee_1\into\Ff\into\Ee_2
\]
such that either
\begin{enumerate}
\item 
$\codim\supp\left(\Ff/\Ee_1\right)\geq 2$; or
\item
$\Ee_1\to\Ee_2$ is an isomorphism away from a union $D$ of irreducible divisors $D_i$, and for 
every $D_i$ over which $\Ee_1\to\Ee_2$ is not an isomorphism, we have
\begin{equation}
\im(\Ee_1|_{\eta_i}\to\Ee_2|_{\eta_i})=\im(\Ff|_{\eta_i}\to\Ee_2|_{\eta_i}),
\end{equation}
where $\eta_i$ is the generic point of $D_i$.  Suppose also that the scheme-theoretic support 
of $\Ee_2/\Ee_1$ is $D$.
\end{enumerate}
Then $\Ee_1\to\Ff$ is an isomorphism.
\end{lem}
\end{proof}
\begin{proof}[Proof of Lemma \ref{lem Dennis}]
Consider the short exact sequence
\begin{equation}
\label{eqn Dennis lemma ses}
0\to\Ee_1\to\Ff\to\Ff/\Ee_1\to 0.
\end{equation}
We claim that in the second case, we also have $\codim\supp\Ff/\Ee_1\geq 2$. Indeed, consider the localization
at the generic point $\eta_i$ of one of the divisors $D_i$ from the statement of the Lemma.
This localization is a discrete valuation ring, and so the torsion-free coherent sheaves $\Ee_1, \Ff$ restrict
to vector bundles over it. It is clear that a lower modification of a vector bundle by a vector 
bundle such that the scheme-theoretic support of the quotient is equal to 
$\eta_i$ is determined
by the image on fibres. Thus we have $\codim\supp\Ff/\Ee_1\geq 2$.

Therefore we have
\[
\mathrm{depth}(\Ee_1)-\dim\supp(\Ff/\Ee_1)=\dim(S)-\dim\supp(\Ff/\Ee_1)=\codim\supp(\Ff/\Ee_1)\geq 2.
\]
By Ischebeck's Theorem, we have $\Ext^1(\Ff/\Ee_1,\Ee_1)=0$.
Thus \eqref{eqn Dennis lemma ses} is split, and in particular
$\Ff/\Ee_1\into\Ff$. Thus $\Ff/\Ee_1$ is both torsion and torsion-free, and so is zero.
\end{proof}
\subsubsection{Examples of lower modifications}
The simplest examples of $\iota_u\circ\eta_u$ failing to be injective on fibres arise as follows: Suppose that $u$ appears in the parameter of a unique $\sigma\in\mathcal{E}_2(M)$ for a unique Levi subgroup $M$, and also that 
\[
\rank\left(i_P^G(\sigma\otimes\nu)(t_d)\right)=1
\]
for all $d\in\cc(u)$. Then injectivity on fibres of $\iota_u\circ\eta_u$ would imply
that $J_u$ is a matrix ring. But if ever 
\[
i_P^GK(u,s,\rho)=i_P^G(\sigma\otimes\nu)=\pi_1\oplus\pi_2
\]
is a reducible tempered representation (so that $s$ is compact and $\nu$ is unitary) with irreducible
summands $\pi_1,\pi_2$,
the containment $\mathcal{E}_{J,u}\into\mathcal{C}(G)$ gives that $i_P^GK(u,s,\rho)$ is a 
reducible $\mathcal{E}_{J,u}$-module, and the image of $\mathcal{E}_{J,u}|_s$ in 
$\Mat_{\dim\pi_1^I+\dim\pi_2^I}(\C)$ consists of block matrices. In particular $\pi_i(t_d)\neq 0$ for exactly
one of $i=1,2$, so if $e\in t_d\mathcal{E}_{J,u}t_{d'}$ is an ``off-diagonal" element with 
$\pi_1(t_{d'})\neq 0$ and $\pi_1(t_d)=0$, then $e|_s\in\ker(\iota_u|_s)$ and $\iota_u\circ\eta_u$ is not 
injective on fibres. This is precisely what happens in both of the following examples, forcing central 
extensions to appear. 
\begin{ex}
\label{ex SL2 central ext}
Let $G=\SL_2(F)$, and $s_0, s_1$ be the simple reflections,
with $s_1$ the finite reflection. Recall that 
the Schwartz functions corresponding to $t_{s_1}$ and $t_{s_0}$ are
$G(\Oo)\times G(\Oo)$-invariant and $K'\times K'$-invariant, respectively, 
where $K'$ is the image of $G(\Oo)$ under the outer automorphism of $\SL_2(F)$ \cite{DSL2}
\footnote[1]{In \textit{op. cit.}, unusual conventions are used: $s_0$ is the finite simple reflection, and $s_1$ is the affine simple reflection. In the present paper we have translated the results to the usual conventions.}.

We have $J=\End(\St^I)\oplus J_0=\Z t_1\oplus J_0$ \cite{BK}, where the appearance of central
extensions means that, instead of being a matrix algebra, we have maps
\begin{equation}
\label{eqn J0 SL2 matrix}
J_0=\sett{\begin{pmatrix}
R(\PGL_2)& R(\SL_2)_{\mathrm{odd}}\\
R(\SL_2)_{\mathrm{odd}}&R(\PGL_2)
\end{pmatrix}}\into\Ee_{J_0}\into\Mat_{2\times 2}(R(\PGL_2))
\end{equation}
of $R(\PGL_2)=Z(\HH)$-modules, where $R(\SL_2)_{\mathrm{odd}}=V(1)R(\PGL_2)$ is the 
$(R(\PGL_2),R(\PGL_2))$-bimodule of spanned by simple $\SL_2$-modules of odd highest weight 
\cite[Section 8.3]{XiAMemoir} and the maps are induced by the action of $J_0$ on unitary principal series.

At the nontrivial order two element of $\PGL_2$, the map on fibres in \eqref{eqn J0 SL2 matrix}
is not injective and is not an isomorphism\footnote[1]{We especially thank R. Bezrukavnikov and I. Karpov for explaining the implications of this point.}. Indeed, if $z^2=-1$, then $z+z^{-1}=0$. This is the only point 
of $\PGL_2\git\PGL_2$ over which the above map is not an isomorphism
on fibres. Note that $\Oo(\X(A)\git W_T)=\Oo(\PGL_2\git\PGL_2)$.

Correspondingly, for any unitary principal series representation $\pi$, we have
\begin{equation}
\label{eqn SL2 principal series decomposition}
\pi^I=\pi^{G(\Oo)}\oplus\pi^{K'}
\end{equation}
as vector spaces. 
When $\nu$ is the quadratic character, 
\eqref{eqn SL2 principal series decomposition} holds as $H$-modules, and hence, by temperedness, as 
modules over the Schwartz algebra, and in particular as $J$- and $\Ee_J$-modules. Hence elements of 
$J_0$ not preserving \eqref{eqn SL2 principal series decomposition} must act by zero: these are the 
off-diagonal elements of \eqref{eqn J0 SL2 matrix}. Thus the appearance of central extensions is 
forced by reducibility of the principal series.
\end{ex}
\begin{ex}
\label{ex BDD central ext}
Let $G=\SO_7(F)$. In this case $G^\vee=\Sp_6(\C)$ is simply-connected. Consider the two-sided 
cell corresponding to the unipotent element $u=(2,2,2)$ of $G^\vee$ with $\Zur=\SO_3\times\Z/2\Z$.
If the central extensions were trivial, \cite{BO} would imply that $J_u$ 
was a matrix algebra over $R(\Zur)$. In \cite[Theorem 5]{BDD},
Bezrukavnikov, the author, and Dobrovolska showed that central extensions appear in this case, and that
\[
J_u
=
\sets{
\begin{pmatrix}
A & B\\
C &D
\end{pmatrix}
}{A,D\in\Mat_{\overset{3\times 3}{9\times 9}}\left(R(\PGL_2)\right),~B,C\in\Mat_{\overset{3\times 9}{9\times 3}}\left(R(\SL_2)_{\mathrm{odd}}\right)}.
\]
Thus central extensions appear despite the hypothesis on $G^\vee$. The same counterexample
was also computed in \cite{QX}.

We may also see that the central extensions must appear by noting 
that $u$ appears in the parameter of a discrete series representation
of a Levi subgroup of $G$ only for
$\St_{\GL_2}\otimes\pm\St_{\SO_3}\in\Ee_2(\GL_2(F)\times\SO_3(F))$.
The corresponding tempered $J_u$-modules are
\[
\pi=i_P^G(\nu\St_{\GL_2}\otimes\pm\St_{\SO_3})^I,
\]
with reducibility exactly when $\nu$ is the quadratic character. Indeed, G. Dobrovolska has
computed the components of the fixed points of the Springer fibre
$\Bb_u^s$, where $s\in\SO_3(\C)=\left(\Zur\right)^\circ$ corresponds to the quadratic character of 
$\GL_1(F)$, showing that $\Bb_u^s$ is the disjoint union of three copies of $\Pp^1$ and six points.
The points are partitioned into three pairs, each with nontrivial action of
$\pi_0(Z_{\Sp_6}(u,s))=\Z/2\Z$. The $\Pp^1$'s are acted upon trivially. Thus we decompose 
Borel-Moore homology as
\begin{equation}
\label{eqn Galya example}
H_*^{\mathrm{BM}}(\Bb_e^s)=H_*^{\mathrm{BM}}(\pt)^{\oplus 6}\oplus H_*^{\mathrm{BM}}(\Pp^1)^{\oplus 3}
=\triv^{\oplus 9}\oplus\mathrm{sgn}^{\oplus 3}
\end{equation}
as a $\pi_0(Z_{\Sp_6}(u,s))$-representation.

In other words, when $\nu$ is the quadratic character, 
$i_P^G(\nu\otimes(\St_{\GL_2}\otimes\pm\St_{\SO_3}))^I$
is the direct sum of one nine-dimensional and one three-dimensional tempered representation of $H$, by 
\cite{KLDeligneLanglands}. This decomposition holds as modules over the Schwartz algebra, and the map 
$J_u\to\Mat_{12\times 12}(R(\Zur))$ on fibres at $s$ must also land in 
$((3\times 3), (9\times 9))$-block matrices is not injective on fibres. In particular, as concluded in \cite{BDD}, $J_u$ is not isomorphic to a matrix algebra even as an abstract 
ring: its simple representations can have the wrong dimensions.
This proof be made independent of \cite{BDD} or \cite{QX} 
by replacing the computation \eqref{eqn Galya example} with the reducibility computation
in \cite[Theorem 5.2(v)]{Zoric}.
\end{ex}
\subsubsection{Poles of the idempotents $\pi(t_{d,\rho})$}
We emphasize that \eqref{eqn nested modification} does not imply that the operators $\pi(e)$, $e\in\Ee_J$ are 
regular in $\nu$; the basis we used to define \eqref{eqn EJu to single Mat definition} itself depends on 
$\nu$ via the $\pi(t_{d,\rho})$. We now give example singularities in $\nu$ of these idempotents.
\begin{ex}
\label{ex sl2 rep that is not J-module}
Let $G=\SL_2(F)$ and consider the reducible indecomposable principal series $H$-modules, in the notation 
of \cite[Example 2]{Plancherel},
\[
\begin{tikzcd}
0\arrow[r]&\triv\arrow[r]&\pi_1\arrow[r]&\St\arrow[r]&0,
\end{tikzcd}
~~~z^2=q^{-1}
\]
and
\[
\begin{tikzcd}
0\arrow[r]&\St\arrow[r]&\pi_2\arrow[r]&\triv\arrow[r]&0,
\end{tikzcd}
~~~z^2=q.
\]
Note that $\pi_1\neq\pi_2$; in fact the two modules differ by the involution ${}^\dagger(-)$.

The representation $\pi_2$ lies outside the unit circle, and extends to a simple $J$-module, 
as it is a standard $H$-module. In realization \eqref{eqn J0 SL2 matrix} of $J_0$, 
via which we identify $\pi_2^I=\C^2$
we view $z^2$ as the coordinate on $R(\PGL_2)$ where $z$ is the coordinate on the maximal torus of 
$\SL_2$, so that the character of $V(1)\in R(\SL_2)_{\mathrm{odd}}$ is $z+z^{-1}$. Then it is 
easy to see that 
\[
\pi_2(\phi({}^\dagger T_{s_0}))=
\begin{pmatrix}
q&-(q+1)\\
0&-1
\end{pmatrix}
,~~~~
\pi_2(\phi({}^\dagger T_{s_1}))=
\begin{pmatrix}
-1&0\\
-(q+1)&q
\end{pmatrix}
\]
with respect to the standard basis of $\C^2$ \cite{DSL2}. Hence the diagonal is isomorphic to $\St^I$
as an $H$-module, but is not stable under the action of $J$. (Else the quotient $\triv$ would be a 
$J$-module.)
The Steinberg subspace may constructed by taking $v_0$ to be any nonzero $\pi_2(t_{s_0})$ eigenvector,
and defining $v_1=(z+z^{-1})^{-1}\pi_2(t_{s_1s_0})v_0$, so that $\St^I=\spn\{v_0+v_1\}$.


As $\pi_1$ lies inside the unit circle, \cite[Theorem 2.4]{BK}
does not guarantee that $\pi_1$ extends to a $J$-module. And indeed, it does not:
we have $a(\pi_1)=1$ but $a(\St)=0$, contradicting \cite[Theorem 3.2]{XiJAMS}.
\end{ex}

\subsubsection{Reduction to semisimple groups}
\label{subsubsection reduction to semisimple groups}
Let $A^\vee_{\mathrm{der}}$ be a maximal torus in the derived subgroup $G^\vee_{\mathrm{der}}$. Then we have
\[
R(G^\vee)=R(A^\vee_{\mathrm{der}})^W\otimes_\C X^*(Z(G^\vee)^\circ)_\C,
\]
and similarly for unramified characters of $G$.
Moreover, 
the torus $Z(G^\vee)^\circ$ acts trivially on $Y_u$ for all $u$,
and does not admit projective representations.
Therefore there is no harm in assuming that the connected centre of 
$G^\vee$ is trivial, \textit{i.e.}, that $G^\vee$ and $\G$ are 
semisimple.

Next, consider an isogeny $\G_1\to\G_2$ of semisimple groups over $F$ inducing
\begin{equation*}
\begin{tikzcd}
H(G_1,I_1)\arrow[d, hook]\arrow[r, hook]&J_1\arrow[r, hook, "\eta_1"]\arrow[d, hook, "\iota_J"]&\Ee_{J_1}\arrow[r,hook]\arrow[d, hook, "\iota_\Ee"]&\mathcal{C}(G_1,I_1)\arrow[d,hook, "\iota_\mathcal{C}"]\\
H(G_2,I_2)\arrow[r, hook]&J_2\arrow[r, hook, "\eta_2"]&\Ee_{J_2}\arrow[r,hook]&\mathcal{C}(G_2,I_2)
\end{tikzcd}
\end{equation*}
\begin{lem}
\label{lem suffices to show surjectivity for one member of isogeny class}
The map $\eta_1$ is surjective if and only if the map $\eta_2$ is surjective.
\end{lem}
\begin{proof}
We may assume that $G_1$ is simply-connected. Then $H(G_2,I_2)$ is generated over the image of $H(G_1,I_1)$
by $\{T_{\omega}\}_{\omega\in\pi_1(G_2)}$, and likewise for $\mathcal{C}(G_2,I_2)$. Indeed,
if 
\[
f=\sum_{x\in\widetilde{W}(G_2)}a_xT_x
\]
is a Schwartz function on $G_2$, then 
\[
f=\sum_{\omega\in\pi_1(G_2)}T_\omega f_\omega,
\]
where each 
\[
f_\omega=\sum_{x\in\widetilde{W}(G_1)}a_{\omega x}T_x
\]
is a Schwartz function in the image of $\mathcal{C}(G_1,I_1)$.

Suppose that $\eta_2$ is an isomorphism and let $e\in \Ee_{J_1}$. Then there is $j_2\in J_2$ such that
$\eta_2(j_2)=\iota_\Ee(e)$ in $J_2$. By \cite{Plancherel}, the formula for $\phi_{G_2}^{-1}(j_2)$ in the 
completion of the affine Hecke algebra of $G_2$ is the formula for $\iota_\mathcal{C}(e)$ as a Schwartz function on $G_2$, or equivalently, for $e\in\mathcal{C}(G_1,I_1)$. As Lusztig's map $\phi$ for $G_1$ is just the restriction of $\phi_{G_2}$, this means that $j_2=\iota_J(j_1)$ for some $j_1\in J_1$, and $\eta_1(j_1)=e$.

Conversely, suppose that $\eta_1$ is an isomorphism. Then we have injections of $\C$-vector spaces
\[
J_2/\iota_J(J_1)\into\mathcal{E}_{J_2}/\iota_\Ee(\mathcal{E}_{J_1})\into\mathcal{C}(G,I_2)/\iota_\mathcal{C}(\mathcal{C}(G_1,I_1)).
\]
The outer terms are both of dimension $\#\pi_1(G_2)$, hence so is the inner term. Thus, for each
unipotent $u\in\Gd$,
\[
\mathcal{E}_{J_2,u}=\bangles{\iota_{\Ee}(\Ee_{J_1,u}), T_{\omega}}_{\omega\in\pi_1(G_2)}=
\bangles{\eta_2(\iota_J(J_{1,u})),T_{\omega}}_{\omega\in\pi_1(G_2)}=\eta_2(J_{2,u}).
\]
\end{proof}
Therefore is suffices to prove surjectivity of $\eta$ only for one member of a given isogeny class
of semisimple $F$-groups. We will elect to work with classical groups and adjoint exceptional groups.
\subsection{Unipotent centralizers}
\label{subsection Unipotent centralizers}

\subsubsection{Character rings of disconnected groups}
\label{subsubsection rep rings of disconnected and Clifford theory}
We refer to the Appendix for information about the varieties $\Zur\git\Zur$. As noted there, when $\Zur=(\Zur)^\circ\rtimes\pi_0(\Zur)$, the information we need is in \cite{diss}, and for classical groups, we need only
\begin{ex}[\cite{Minami}, \cite{Tak}, \cite{Husemoller}]
\label{example rep ring of orthogonal groups}
There are two conjugacy classes of Cartan subgroups (see Section \ref{subsection appendix Cartan subgroups}) in $\OO_{2n}$, with representatives given by 
the diagonal maximal torus $C_1=C_{1,\SO_{2n}}\subset\SO_{2n}$, and 
$C_2:=\bangles{C_{1,\SO_{2n-2}}, \begin{pmatrix}1& 2n\end{pmatrix}}$. The action of the generator
$\gamma=\begin{pmatrix}1& 2n\end{pmatrix}$ of $\Gamma$ on $C_1$ changes a single sign, and the normalizer 
of $C_2^\circ$ may be calculated in $\SO_{2n}$. Therefore \cite{Minami}, \cite{Tak} give
\begin{equation}
\label{eqn RO(2n) direct factor presentation}
R(\OO_{2n})=R(C_1)^{W(B_n)}\times R(C_2^\circ)^{W(B_{n-1})}=\C[V_1,\dots, V_n,\det]/({\det}^2-1, \det\otimes V_n-V_n),
\end{equation}
where $W(B_{m})$ is the Weyl group of type $B_m$, $V_i$ is the $i$-th exterior power of the defining representation on $\C^{2n}$, and $\det$ gives rise to the idempotent defining the direct factors.
For $i<n$, $V_i|_{\SO_n}$ is simple, and $\gamma(V_i|_{\SO_n})=V_i|_{\SO_n}$.
The two irreducible summands of $V_n|_{\SO_n}=V_n^+\oplus V_n^-$ are permuted by $\gamma$, and 
\[
R(\SO_{2n})=\C[V_1,\dots,V_{n-1}, V_n^+, V_n^-]/\left((V_n^++V_{n-2}+\cdots )(V_n^-+V_{n-2}+\cdots )-
(V_{n-1}+V_{n-3}+\cdots)^2\right).
\]
\end{ex}
We need an interpretation of the finite group $\Delta_\gamma$ in Proposition \ref{Prop2} (1). By the proof of Proposition \ref{Prop2}, $\Delta_\gamma$ has the same interpretation as in the case when $\Gg=\Gg^0\rtimes\pi_0(\Gg)$ considered in \cite{diss}. Namely, by \cite[Lemma 3.9]{diss}, there is a semidirect product decomposition $\Delta_\gamma=H\rtimes W_\gamma$, where $W_\gamma$ is the Weyl group with Dynkin diagram given by folding that of $\Zz^\circ$ according to $\gamma$.

\subsubsection{Unipotent centralizers in classical groups}
\label{subsubsection unipotent centralizers types BCD}
Suppose that $\Gd$ is $\Sp_{n}$ or $\SO_{n}$, where for symplectic groups, $n$ is even.
Then unipotent conjugacy classes in $G^\vee$ are parameterized by partitions 
\[
u=(\underbrace{a_1,\dots, a_1}_{m_1}, \underbrace{a_2,\dots, a_2}_{m_2},\dots,\underbrace{a_k,\dots, a_k}_{m_k})
\]
of $n$,
where we enforce $m_i$ even when $a_i$ is odd in type $C$, and that $m_i$ is even when $a_i$ is even in 
types $B$ and $D$. If $G^\vee=\SO_n$, $n$ is even, the partition is very even, in 
which case two unipotent conjugacy classes are labelled by $(a_i)_i$. 

By the Springer-Steinberg theorem \cite[Thm. 6.1.3]{CM} we have
\begin{equation}
\label{eqn classical centralizers form}
\Zur=
\begin{cases}
\prod_{\substack{a_i~\text{odd}}}\Sp_{m_i}\times\prod_{\substack{a_i~\text{even}}}\OO_{m_i}&\text{if}~G^\vee=\Sp_n\\
\prod_{a_i~\text{even}}\Sp_{m_i}\times \prod_{\substack{a_i~\text{odd} \\ m_i~\text{odd}}}\SO_{m_i}
\times S\left(\prod_{\substack{a_i~\text{odd} \\ m_i~\text{odd}}}\Z/2\Z\times\prod_{\substack{a_i~\text{odd}\\ m_i~\text{even}}}\OO_{m_i}\right)&\text{if}~G^\vee=\SO_n
\end{cases},
\end{equation}
where $S(-)$ imposes condition that the product of the determinants is unity.

We compute $R(Z_{\Sp_n}(u)^{\mathrm{red}})$ by applying Example \ref{example rep ring of orthogonal groups}
factor-wise. For special orthogonal groups, we use
\begin{lem}
\label{lem direct factors of special orthogonal centralizers from orthogonal}
Let $G^\vee=\SO_{2m}$ or $\SO_{2m+1}$ and 
$\Zz=S(\prod_{m_i~\emph{even}}\OO_{m_i})$.
\begin{enumerate}
\item[(a)]
Every conjugacy class of Cartan subgroups of $\Zz$ has a representative 
of the form $S(C)$ for $C$ a Cartan subgroup of $\prod_{m_i~\emph{even}}\OO_{m_i}$ such that
$\tilde{R}(S(C))=R(C^\circ)$.
\item[(b)] 
$R(\Zz)$ is a direct product indexed by decompositions $\sett{m_j}_j=M_1\sqcup M_2$, with factors
\begin{equation}
\label{eqn classical centralizer component general form}
\left(\bigotimes_{i\in M_1}R(C_1^i)\right)^{\Gamma}\otimes\bigotimes_{j\in M_2}R((C_2^{j})^\circ)^{W(B_{m_j-2})},
\end{equation}
where $C_i^j$ is a Cartan subgroup of $\OO_{m_j}$ as in Example \ref{example rep ring of orthogonal groups}.
\item[(c)]
Each direct factor \eqref{eqn classical centralizer component general form} with 
$M_2\neq \emptyset$ is equal to 
\begin{equation}
\label{eqn classical centralizer component total breakdown}
\bigotimes_{i\in M_1}R(C_1^i)^{W(B_{m_i})}\otimes\bigotimes_{j\in M_2}R((C_2^{j})^\circ)^{W(B_{m_j-2})}
\end{equation}
\item[(d)]
If a factor of $\Z/2\Z$ is present in the second case of \eqref{eqn classical centralizers form},
then $R(\Zur)$ is a direct product of tensor products with tensor factors given by $R(\Sp_m)$,
$R(\SO_{2m+1})$, and factors of the form \eqref{eqn classical centralizer component total breakdown}.
\item[(e)]
Given a $\Zur$-action on a finite set $Y$, for any $y\in Y$, we have that $R(\stab{\Zur}{y})$ is as in
(c), except in the factor coming from (b), we replace $\Gamma$ by $\stab{\Gamma}{y}$.
\end{enumerate}
\end{lem}
\begin{proof}
Part (a) is clear.
Part (b) follows directly from \ref{Prop2} and Example \ref{example rep ring of orthogonal groups}. For (c), we have $\pi_0(S(\prod_{m_i~\emph{even}}\OO_{m_i})=\genrel{\begin{pmatrix}1&m_i\end{pmatrix}\begin{pmatrix}1&m_j\end{pmatrix}}{i\neq j}$, and $\begin{pmatrix}1&m_j\end{pmatrix}$ acts trivially on all the $(C_2^j)^\circ$. 
Hence for $t\in M_1$ and $r\in M_2$, $\begin{pmatrix}1&2m_t\end{pmatrix}\begin{pmatrix}1&2m_r\end{pmatrix}$
acts by a single sign change on the first tensor factor of \eqref{eqn classical centralizer component general form}, giving (c). If $\gamma\neq 1$ in $\Z/2\Z$, the same holds for 
$\left(\gamma, \begin{pmatrix}1&m_t\end{pmatrix}\right)\in\pi_0(\Zur)$, whence (d). As 
$\stab{\Zur}{y}=(\Zur)^\circ\rtimes\stab{\Gamma}{y}$, (e) follows.
\end{proof}
\subsubsection{$\Spin$ and $\Pin$ groups}
\label{subsubsection Spin and Pin groups}
We specify the group $\Pin_n$ of $\OO_n$ by declaring that its compact real form
is as defined in \cite[Section 3]{Tak} using Clifford algebras. When $n$ is odd, this definition is 
such that $\Pin_n=\Spin_n\times\Z/2\Z$, and in general, 
$\Pin_n^\circ=\Spin_n$. It plays the role of the universal covering
group for $\OO_n$ in that projective representations of the later
are honest representations of the former with nontrivial central character.
If $V_i$ is the $i$-th exterior 
product of the defining representation of either $\SO_n$ or $\OO_n$, then we will also denote by $V_i$ 
its pullback to $\Spin_n$ or $\Pin_n$, respectively.
\begin{ex}[\cite{Tak}, \cite{Husemoller}]
\label{example rep ring of Pin}
Representatives of the two conjugacy classes of Cartan subgroup in $\Pin_{2n}$ can be taken to be
$C_1=T_{\Spin_{2n}}$ a maximal torus of $\Spin_{2n}$, and $C_2=\bangles{T_{\Spin_{2n-2}},\gamma}$, where
$\bangles{\gamma}=\Z/2\Z=\pi_0(\Pin_{2n})$. As in Example \ref{example rep ring of orthogonal groups},
we can compute $\tilde{R}(C_2)^{W_{C_2}}=R(C_2^\circ)^{W_{C_2}}$ in $\Spin_{2n}$. Proposition \ref{Prop2}
and \cite{Tak} give the first and second equalities in
\[
R(\Pin_{2n})=R(C_1)^{W(B_n)}\times R(C_2)^{W(B_{n-2})}
=\C[V_1,\dots, V_{n-1},\Pi, \nu]/(\Pi\otimes\det-\Pi, {\det}^2-1)
\]
respectively, where $\det\colon\Pin_{2n}\to\Z/2\Z$ is pulled back from $\OO_{2n}$.
The two irreducible summands of $\Pi|_{\Spin_{2n}}=\Delta_{2n}^++\Delta_{2n}^-$ are 
are the two half-spinorial representations and are permuted by the component 
group, whereas $V_i|_{\Spin_{2n}}$ is simple, with
\[
R(\Spin_{2n})=\C[V_1,\dots, V_{n-2}, \Delta_{2n}^+,\Delta_{2n}^-].
\]
In the natural coordinates on a maximal torus of $\Spin_{2n}$,
\[
\Delta^+=\sum_{\epsilon(j)=\pm 1}z_1^{\epsilon(1)/2}\cdots z_r^{\epsilon(r)/2}
\]
and likewise for $\Delta^-$.
\end{ex}
We define $S(\prod_{i}\Pin_{n_i})=\ker\boxtimes_i\det_i$. Every
connected component of $S(\prod_{i}\Pin_{n_i})$ is simply-connected, 
and projective representations of  $\prod_{i}\Sp_{m_i}\times S\left(\prod_{j}\OO_{m_j}\right)$ 
are honest representations of 
$\prod_{i}\Sp_{m_i}\times S\left(\prod_{j}\Pin_{m_j}\right)$
with nontrivial central character.
\begin{lem}
\label{lem shape of characters of finite index Zur subgroups}
\begin{enumerate}
\item[(a)]
Statements (a)--(e) of Lemma \ref{lem direct factors of special orthogonal 
centralizers from orthogonal} are true for the covers of $\Zur$ and $\Zz$
obtained via the (s)pin groups.
\item[(b)]
The character of any irreducible representation of any finite index subgroup of $S(\prod_{i}\Pin_{m_i})$ is an 
element of $\bigotimes_i R^i$, where
$R^i$ is $R(\Spin_{m_i})$ if $m_i$ is odd, and is one of $R(\Spin_{m_i})$, 
$\Pi_{m_i}R(\Pin_{m_i})$ or $\Delta_{m_i}^\pm R(\Spin_{m_i})$ if $m_i$ is even. In particular,
any projective representation of $\stab{\Zur}{y}$ as in Lemma \ref{lem direct factors of special orthogonal centralizers from orthogonal} (e) has character of this form, with (s)pinorial representations being
allowed in fixed coordinate slots.
\end{enumerate}
\end{lem}
\begin{proof}
The same argument as for Lemma \ref{lem direct factors of special orthogonal centralizers from orthogonal} 
gives (a), which implies (b).
\end{proof}
\subsection{Proof of Proposition \ref{prop quotient of unramified characters is flat}: Flatness of $\Oo(\X(M))^{W_M}$}
\label{subsubsection flatness of O(X(M)/W)}
This section proves Proposition \ref{prop quotient of unramified characters is flat} case-by-case.
We explain the strategy. For each unipotent $u\in G^\vee$, we will 
compute all Levi subgroups $M$ such that $u$ is distinguished in the centralizer in $[M^\vee, M^\vee]$ of a 
semisimple element which itself is semisimple. The Dynkin diagrams of such centralizers are given according
to Kac's classification \cite[Ch.4, Section 4.8, Problem 61]{Kac} in terms of full subdiagrams of the affine Dynkin diagram of $[M^\vee, M^\vee]$,
and when $G^\vee$ is of classical type, so are the semisimple centralizers as groups.
Thus we obtain a constraint on $M$ with $\sigma\in\Ee_2(M)$ such that $u$ appears in the 
parameter of $\sigma$, in terms re-orderings of the partition corresponding to $u$, and the fact that, by
\eqref{eqn classical centralizers form}, a unipotent is distinguished in a classical group
if and only if $u=(n)$ for type $A_{n-1}$, and otherwise all $a_i$ are distinct.

Let $T$ be a maximal torus of $(\Zur)^\circ$ and $\gamma\in\Gamma$. 
As mentioned in Section \ref{subsubsection rep rings of disconnected and Clifford theory}, by the proof of Lemma 3.11 of \cite{diss}, the coinvariant torus Proposition \ref{Prop2} is presented as the quotient
\begin{equation}
\label{eqn fixed torus is etale cover of coinvariant}
\begin{tikzcd}
1\arrow[r]&H\arrow[r]&(T^\gamma)^\circ\arrow[r]&T_{\gamma}\arrow[r]&1.
\end{tikzcd}
\end{equation}

We have $A^\vee_M\subset\Zur$, and claim that in fact $A^\vee_M$ is the identity 
component of a Cartan subgroup of $\Zur$ as in Example \ref{example rep ring of orthogonal groups} and
Lemma \ref{lem direct factors of special orthogonal centralizers from orthogonal}. 
Then we check that the 
group $W_M$ of Section \ref{subsection labels for intertwining operators} is one of the groups appearing in a 
direct factor in Lemma \ref{lem direct factors of special orthogonal centralizers from orthogonal}. 
Finally, we check in Lemma \ref{lem quotients of unram characters and fixed torus agree} that the kernels in \eqref{eqn orbit under etale twist is covered by unram characters} and \eqref{eqn fixed torus is etale cover of coinvariant} match. Together this shows that $\mathfrak{o}_\sigma\git W_M$ is a connected component of $\Zur\git\Zur$. 

%
%
\subsubsection{Type $B$}
Let $\G=\SO_{2n+1}$, so that a unipotent conjugacy class in $G^\vee=\Sp_{2n}(\C)$ is
specified by a partition as in Section \ref{subsubsection unipotent centralizers types BCD}. Let $a_1,\dots, a_\ell$ be all the odd entries, with $a_1=1$. If $u$ can be rewritten as
\[
u=(\underbrace{a_1,\dots, a_1}_{m_1'}, \underbrace{a_2,\dots, a_2}_{m_2'},\dots,\underbrace{a_k,\dots, a_k}_{m_k'},\underbrace{a_1,\dots, a_1}_{m_1''}, \underbrace{a_2,\dots, a_2}_{m_2''},\dots,\underbrace{a_k,\dots, a_k}_{m_k''})
\]
such that $m_i=m_i'+m_i''$ with $m_i'$ even, $\sum_{i}m_i''a_i=2r$, and 
\[
u''=(\underbrace{a_1,\dots, a_1}_{m_1''}, \underbrace{a_2,\dots, a_2}_{m_2''},\dots,\underbrace{a_k,\dots, a_k}_{m_k''})\in\Sp_{2r}(\C),
\]
a unipotent distinguished in a semisimple centralizer inside the type $C_r$ component of $M^\vee$,
then $u$ occurs in the parameter of $i_P^G(\sigma)$ 
for some $\sigma\in\mathcal{E}_2(M)$, where $\M$ is the standard
Levi subgroup
\begin{equation}
\label{eqn type B standard Levi}
\M=\underbrace{GL_{a_1}\times\cdots\times\GL_{a_1}}_{\frac{m_1'}{2}}\times\cdots\times
\underbrace{\GL_{a_k}\times\cdots\times\GL_{a_k}}_{\frac{m_k'}{2}}\times\G_1
\end{equation}
of $\G$ contained in a parabolic $\mathbf{P}$, with $\G_1=\SO_{2r+1}$. All such Levis arise this way. 

In this case, we claim that 
\[
m_1''=\cdots = m_{\ell}''=0
\]
and
\[
m_{\ell+1}'',\cdots, m_{k}''\in\sett{0,1,2}.
\]
Indeed, by Kac's classification, we have 
\begin{equation}
\label{eqn centralizer is smaller symplectic}
Z_{\Sp_{2r}}(s)=\Sp_{2(r-j)}\times\Sp_{2j}
\end{equation}
if $s$ is the semisimple part of a discrete-series parameter. By 
\eqref{eqn classical centralizers form}, $u''=(u''_1,u''_2)$ is distinguished in \eqref{eqn 
centralizer is smaller symplectic} if and only if the partitions $u''_i$ each have distinct 
parts. Both claims follow.

Of course, 
\[
A_M^\vee=\Gm^{\frac{m'_1}{2}}\times\cdots\times
\Gm^{\frac{m'_k}{2}}
\]
is a subgroup of 
\[
\Zur=\Sp_{m_1}\times\cdots\times\Sp_{m_\ell}\times\OO_{m_{\ell+1}}\times
\cdots\times\OO_{m_k}.
\]
It remains only to match the above with the description of $W_M$
from \cite{Howlett}; $W_M$ is of type
\begin{equation}
\label{eqn type B Howlett match}
B_t\times B_{\frac{m'_2}{2}}\times\cdots\times B_{\frac{m'_k}{2}},
\end{equation}
where 
\[
t=r+\sum_{i=1}^{k}(a_i+1)\frac{m_i'}{2}
\]
is the maximal Dynkin diagram label belonging to $\mathbf{P}$, in the conventions of \cite{Howlett}. This factor corresponds to $a_1=1$, \textit{i.e.} to the factors $\GL_1^{\times t}$ in $M^\vee$ not appearing in the Dynkin diagram of $M^\vee_{\mathrm{der}}$.

Hence for $m_2=m'_2,\dots, m_l=m'_l$, we obtain
the equality $R(\Sp_{m_i})=\Oo(\Gm^{m_i/2})^{W_{B_{m_i}}}$ as in Example \ref{example rep ring of orthogonal groups}.
%
For $m_{l+1},\dots, m_k$, if $m_i$ is even, then $m_i'=m_i$ or $m_i'=m_i-2$. In each case, 
we see that $R(\Gm^{\frac{m'_i}{2}})^{W(B_{m_i'/2})}$ matches with one of the direct 
factors of $R(\OO_{m_i})$, by Example \ref{example rep ring of orthogonal groups}.
If $m_i=m_i'-1$ is odd, we again conclude for each of the isomorphic direct factors 
of $R(\OO_{m_i})$.
%
\subsubsection{Type $C$}
Next Let $\G=\Sp_{2n}$ and $G^\vee=\SO_{2n+1}$.
In this case $m_i$ is even when $a_i$ is even, and semisimple centralizers in $G^\vee$ are of type
$B_{j}\times D_{r-j}$. As in the previous paragraph, supposing $a_1,\dots, a_l$ to be all the even entries, we obtain $m''_1=\cdots=m''_l=0$, $m''_i\in\{0,1,2\}$ with $m_i''=1$ if and only if $m_i$ is odd, and 
the relevant $\M$ are as in \eqref{eqn type B standard Levi},
where now $\G_1=\Sp_{2r}$ and $2r+1=\sum_ia_im_i''$. Again
\[
\mathbf{A}_\M^\vee=\Gm^{\frac{m'_1}{2}}\times\cdots\times\Gm^{\frac{m'_k}{2}}
\]
is a subgroup of $\Zur$. 

By \cite{Howlett}, $W_M$ is again abstractly of type 
\eqref{eqn type B Howlett match}. On the other hand, we must have $m_i$ odd for some $i$, and so 
we can carry out the matching as in the previous paragraph, by Lemma \ref{lem direct factors of special orthogonal centralizers from orthogonal}, (c).
\subsubsection{Type $D$}
Finally, let $G^\vee=\SO_{2n}$. In this 
case we again have that $m_i$ is even for $a_i$ even, and semisimple centralizers in $G^\vee$ are of type  $D_{j}\times D_{(r-j)}$. As in the previous paragraph, supposing $a_1,\dots, a_l$ to be all the even 
entries, we again obtain $m''_1=\cdots=m''_l=0$, $m''_i\in\{0,1,2\}$, and the relevant $\M$ are as in \eqref{eqn type B standard Levi}, where now $\G_1=\SO_{2r}$ and $2j=\sum_ia_im_i''$. Thus
\[
\mathbf{A}_\M^\vee=\Gm^{\frac{m'_1}{2}}\times\cdots\times\Gm^{\frac{m'_k}{2}}.
\]
\paragraph{Case $j>0$.}
In this case $u''$ is distinguished in $\SO_{2r}\times\SO_{2j-2r}$ for some $r$. 
Hence for some $i$, $m_i''=1$, implying $m_i$ odd, unless $j$ is even and $r=j/2$, in which 
case $A_M^\vee=S(C)$ is a Cartan subgroup of $\Zur$ with $C$ containing a factor of the form $C_2^j$
for some $j$. Either way, we are  in case (c) of Lemma \ref{lem direct factors of special orthogonal centralizers from orthogonal}, as required.
%
\paragraph{Case $j=0$.}
In this case $m_i''=0$ for all $i$ and all $m_i$ are even. Without loss of generality, 
$\alpha_1\not\in\Delta_{P^\vee}$. First suppose that $a_i>1$ for all $i$. Then by
\cite[p.73, Case 2(iii)]{Howlett}, we have 
\[
W_M=\left(\prod_{a_i~\text{even}}W(C_{m_i/2})\times\prod_{a_j~\text{odd}}W(D_{m_j/2})\right)\rtimes V
\]
where if $v_j\in\Aut(D_{m_j/2})$ acts by diagram symmetry, the group $V$ is generated by elements 
$v_iv_j$, $i\neq j$. Thus by Lemma \ref{lem direct factors of special orthogonal centralizers from 
orthogonal} (b) (case $M_2=\emptyset$), $\Oo(A_{M^\vee})^{W_M}$ is a direct factor of $R(\Zur)$.

Now suppose that the multiplicity of $1$ in $u$ is $m_{l+1}>0$. In this case, $V$ is generated by all the $v_j$, and all that remains to check 
is that each $v_j\in V$ also acts on the component of type $D_{m_{l+1}/2}$ of $W_M''$. Indeed, though, 
this is implied by the fact that $v_j\not\in W_{M}''$: if the image of $v_j$ in $V/W_M''$
is nontrivial, then $v_j$ does not act as a reflection under the isomorphism of Theorem 6 of \cite{Howlett}, 
and hence $v_j$ acts nontrivially on the component $D_{m_{l+1}/2}$, and must moreover be the unique
nontrivial element of $\Aut(D_{m_{l+1}/2})/W(D_{m_{l+1}/2})$. Thus $v_j$ acts by the automorphism
of the $D_{m_{l+1}/2}$ diagram. We therefore match according to Lemma 
\ref{lem direct factors of special orthogonal centralizers from orthogonal} (b).
\begin{rem}
For classical groups, $\Zur$ is connected if and only if all the $m_i$ have the same certain parity,
and therefore if and only if $u$ appears in a parameter in 
$\Ee^2(M_P)$ for a unique parabolic $P$, for which $M_P$ is a product of general linear groups. 
By \cite{Ale79}, it is also true for exceptional groups that when $\Zur$ is connected, the 
Levi subgroup is unique.
\end{rem}
\subsubsection{Flatness of $\Oo(\X(M))^{W_M}$ for exceptional groups: $\Gamma=\Z/2\Z$}
\label{subection flatness of invariant functions on characters exceptional}
Now let $\G$ be adjoint of exceptional type.

We first detail more of the results of Alexeevski \cite{Ale79}.
Recall that a subgroup $R$ of $G^\vee$ is \emph{regular} if it is normalized by a maximal
torus of $G^\vee$. Given $u$, let $S$ be a 
three-dimensional subgroup of $G^\vee$ containing $u$. For $R_i$ a minimal regular $S$-containing subgroup
of $G^\vee$, let $\hat{R}_i$ be a connected subgroup of maximal rank such that the semisimple part 
of $\hat{R}_i$ is $R_i$. Denote $D_i=Z(\hat{R}_i)$. Then
$D_i$ is maximal among diagonalizable regular subgroups of $G^\vee$, and conversely, for any such group $D$,
if $\hat{R}:=Z_{G^\vee}(D)$ and $R=\hat{R}_s$, then $R$ is a minimal $S$-containing regular subgroup
\cite[Prop. 2.1]{Ale79}. 
In \cite{Ale79}, for every $u$, 
all of the groups $R_i$ and $D_i$
are computed up to conjugacy. In particular, centralizers of tori of $G^\vee$ are regular, 
and the above discussion shows that 
each group $[Z_{M_P^\vee}(s)^{\mathrm{red}},Z_{M_P^\vee}(s)^{\mathrm{red}}]$ 
is of the form $R_i$ for some $i$ for which $A_P^\vee=D_i$. In particular, the group
$D_i$ is naturally a subgroup of $\Zur$.

Therefore, for fixed $u$, and each subgroup $R_i$ attached to $u$ in
\cite{Ale79}, 
we must find all $P$ such that $R_i$ is contained in $M_P^\vee$,
and check that $\C[D_i]^{W_{M_P}}$ is a direct factor
of $R(\Zur)$. When ${\Zur}^\circ$ is semisimple, we will do this by using Proposition \ref{Prop2} and the folding interpretation recalled in Section \ref{subsection Unipotent centralizers}. The few other cases we will indicated how to deal with by hand.

Not every group $R_i$ appears as a semisimple centralizer in a Levi subgroup of $G^\vee$. 
When this happens, it is either the case that is semisimple $\Zz^\circ$ and its Dynkin diagram has no symmetries,
or it is the case that all $M^\vee$ of correct rank to contain $R_i$ in fact have $W_{M}$ as required.
When $D_i$ is finite, $\hat{R_i}$ itself is semisimple and 
cannot be contained in any proper parabolic; this situation arises if 
and only if  
$u$ appears in the parameter of a discrete series 
representation of $G$. Therefore we need only deal with the infinite 
$D_i$ below. 

To determine $P$ starting from $R_i$, we again reason via Kac's classification;
the groups $W_{M_P}$ are tabulated in \cite{Howlett}. For most unipotent conjugacy classes $u$, there is only a single subgroup $R_i$, $\Zur$ is connected with $D_i$ as a maximal torus
by rank considerations, and $W_{M_P}$ is isomorphic to the Weyl group of $\Zur$. 

Whenever $\Zur$ has nontrivial abelian component group, the component group has order two.
For example, $\Zur$ can be of the form we have dealt with already above for classical groups, namely
a direct product of $\Z/2\Z$ and a connected group, or be $\Pin_4$ or $\Pin_6$.
We give the details of the matching of \cite{Ale79} and \cite{Howlett} in some illustrative cases
beyond these examples, keeping the notation of \textit{op. cit.}

Most of the time $\Zur={\Zur}^\circ\rtimes\Z/2\Z$ with ${\Zur}^\circ$ semisimple. Then
we match the direct factors calculated by Proposition \ref{Prop2}, under the folding interpretation recalled in Section \ref{subsection Unipotent centralizers} (and this case is contained already in \cite{diss})
to check that each occurring $\mathfrak{o}_\sigma\git W_M$ matches a connected
component of $\Zur\git\Zur$. 

Some individual reasoning is required. In some 
cases one must use the isomorphisms $W(A_3)\times\Z/2\Z\simeq W(G_2)$, as well as the exceptional isomorphism
$\Sn_4\simeq\Sn_3\ltimes(\Z/2\Z\times\Z/2\Z)$. 
If ${\Zur}^\circ=Z\times Z$ is a direct product of two simple groups and 
$C_\epsilon$ permutes the factors, then Proposition \ref{Prop2} gives 
$R(\Zur)=R({\Zur}^\circ)^{\Z/2\Z}\times R(Z)$; a representation being $C_\gamma$-fixed translates to restricting 
irreducibly to the diagonal. These cases again match \cite{Howlett}, with the $\Z/2\Z$ factor
matching that $W_{M_P}$ can fail in these cases to be a Weyl group. 

The remaining exceptional cases for exceptional groups are as follows: In all cases $\Gamma=\Z/2\Z$, and
in the first class of exceptions, $\Zz^\circ=\GL_n$ or $\mathrm{GSpin}_5$ is not semisimple, and $\Zz$ may fail 
to be a semidirect product. In these cases it is easy to determine the identity components of the centralizers in a maximal torus of $(\Zur)^0$ of $\gamma\neq 1$ in $\Gamma$.
(Note that in the case of $(\Zur)^\circ=\mathrm{GSpin}_5$, the outer automorphism
acts by $t\mapsto t^{-1}$ on the central torus and so $((\Zur)^\circ)^\gamma=\Spin_5$.) The remaining cases
are extensions of $\Gamma$ by $T=\Gm^{\times i}$, $i=1,2$, such $T^\gamma=\mu_2$. In 
this case there are two Cartan subgroups (see Section \ref{subsection appendix Cartan subgroups}), $T$ and $\Gamma$, and $\Zz\git\Zz\simeq T/(\Z/2\Z)\sqcup\Spec\C$.

For $G^\vee=E_7$ and
$J=\{1,5,7\}$, we have that $u=[3A_1]''$ is contained in $M_{P_J}$ and $W_M$ is of type $F_4$, matching 
$\Zur$. (This case is omitted from \cite{Howlett}.)

For $G^\vee=E_7$ and $u=A_1^{21}$, $M_P$ is of type $A_1\times A_4$, and we must see that 
$W_{M_P}=\Z/2\Z$ also acts by $t\mapsto t^{-1}$ on $T_2$. As $W''=1$ in the notation of \cite{Howlett}, $W_M$ 
does not act by reflections, and the only nontrivial involution in $\GL_2(\Z)$ that $W_M$ can act by is $-\id$, 
as required.

\subsubsection{Flatness of $\Oo(\X(M))^{W_M}$ for exceptional groups: $\Gamma=\Sn_3$}
\label{subsubsection flatness of X(M)/WM exceptional S3}
Now we consider the case $\Gamma=\Sn_3$.
In each case, either 
$\Zur$ is the normalizer of a maximal torus in $\SL_3$, or $\Zz^0$ is semisimple \cite{Ale79}. In the former case, the 
results of \cite{MatsTaka} give the factorization 
\[
R(\Zur)=R(T_2)^{\Sn_3}\times R(T_1\times\Z/2\Z),
\]
again matching \cite{Howlett}. In the latter case, by \ref{Prop2}, and additional \textit{ad-hoc} considerations, $R(\Zur)$ is a direct product of three factors, two of which are given by Dynkin diagram automorphisms and one checks matching as in other cases. The \textit{ad-hoc} considerations are needed because \ref{Prop2} and \cite{diss} go via cyclic subgroups $\Gg^\circ\rtimes\An_3$ or $\Gg^\circ\rtimes\Z/2\Z$ for $\Gg=\Zur$, and are arise for $G^\vee=E_7$, $u=[A^{12}_1]'$ or $G^\vee=E_8$, $u=A_1^{13}$ or $[A_1^{12}]''$ (where we encounter the triality folding
as well as the well-known identity $W(F_4)\simeq\Sn_3\ltimes W(D_4)$). In all three cases, we need only note that the $\gamma=1$ direct factor appears with additional $\Sn_3$-invariants, and that $T^{\gamma}=T^{\Sn_3}$ for $\gamma$ a 3-cycle, and that for $\gamma$ a $2$-cycle, no $3$-cycle belongs to the normalizer of the corresponding Cartan subgroup with identity component $T^\gamma$.

Now we finish the proof of Proposition \ref{prop quotient of unramified characters is flat}.
It remains only to show
\begin{lem}
\label{lem quotients of unram characters and fixed torus agree}
Under the identification of $(T^\gamma)^\circ=A_M^\vee=\X(M)$ from $\omega$, we have 
\[
T_\gamma=T/H\simeq \X(M)/\stab{X(M)}{\omega}=\mathfrak{o}_\omega.
\]
\end{lem}
\begin{proof}
According to the proof of Lemma 3.11 in \cite{diss} and the paragraph preceding it, we have $T_\gamma=T/H$ (quotient of groups) and the group $H$ in \eqref{eqn fixed torus is etale cover of coinvariant} may equivalently be described as the image of the homomorphism
\[
\nu\colon T/(T^\gamma)^\circ\to (T^\gamma)^\circ
\]
induced by $t\mapsto t\gamma t^{-1}\gamma^{-1}$. The inclusion $Z_{M^\vee}(u)\into Z_{G^\vee}(u)$  is compatible with the resulting action by multiplication on $(T^\gamma)^\circ=A^\vee_M$. Now, if $(s, u)$ is the parameter of $\omega$ and $\chi\in H$, we have $(\chi s, u)=(\gamma s\gamma^{-1}, \gamma u\gamma^{-1})$. Conversely, if 
$z sz^{-1}=\chi s$ for some $\chi\in \X(M)$, then we see that we may take $z\in H$. 
%
%
%
%
\end{proof}
As we have already checked that the further groups quotiented by in Proposition \ref{Prop2} and $W_M$ coincide, we conclude that $\Mat_n (\Oo(\mathfrak{o}_i)^{W_i})$ is a vector bundle on $\Zur\git\Zur$.
\qed

\subsection{Proof of Proposition \ref{prop Ju and EJu fibre images agree at generic points}}
\label{subsection Proof of prop Ju and EJu imagess agree at generic points}
We outline the proof, the main technical point of which is to control the scheme-theoretic support of
$\mathcal{M}_u/J_u$ in \eqref{eqn nested modification}. 
Recall that irreducibility of a tempered representation $\pi$
is equivalent over $\Ee_{J,u}$ or over $J_u$. Indeed, if such $\pi$ is irreducible
over $\Ee_{J,u}$, it is irreducible over the Schwartz, and thus Hecke algebras, and thus over $J_u$.
Therefore a matrix coefficient of $\pi$ as an $\Ee_{J,u}$-module vanishes if and only if it vanishes as a 
$J_u$-module. Below, we will examine when cases (a) and (b) of Proposition \ref{prop Ju saturated lower modification} occur for each $u$. Case (a) is easy to recognize, and happens for instance when $\Zur$ is connected and several fundamental representations of its universal cover have the same Schur multiplier as projective $\Zur$-representations.

In case (b), let $D_i$ be an irreducible divisor over which \eqref{eqn nested modification} is not
injective on fibres, \textit{i.e.} some matrix coefficient $f$ vanishes, where we think of $f$ as a  class 
function of $s\in\Zur$. Thus $D_i$ is defined as a topological space by the vanishing of some virtual 
character of $\Zur$. Our strategy is to use the observation from Clifford theory at the end of Section 
\ref{subsubsection rep rings of disconnected and Clifford theory} to control which $\Zur$-representations
can appear in the formula for $f$ as a $J_u$-module, and then in each case we observe that the vanishing 
set of the character is reduced. This is precisely the statement of Proposition
\ref{prop Ju and EJu fibre images agree at generic points} (b).
\begin{proof}[Proof of proposition \ref{prop Ju and EJu fibre images agree at generic points}]

First, if $\Zur$ is the direct product of a finite group and a connected group 
not admitting projective
representations, then Lusztig's conjecture is true, and it is easy to see that 
reducibility of the families $\pi_{ij}$ is locally constant in the semisimple 
part $s\in\Zur$ of the parameter of $\pi_{ij}$. Therefore
$\iota_u$ is injective on fibres, so $\eta_u$ is surjective by Nakayama's 
lemma as in Lemma \ref{lem surjectivity for finite centralizers}.

We use Section \ref{subsubsection reduction to semisimple groups} to reduce to semisimple
groups, and then Lemma \ref{lem suffices to show surjectivity for one member of isogeny class} to 
reduce to classical groups and adjoint forms of exceptional groups. For $G^\vee$ of type $A$,
we prove surjectivity of $\eta_u$ separately in Section \ref{subsubsection Surj for GLn}.
\subsubsection{Proof of Proposition \ref{prop Ju and EJu fibre images agree at generic points} (a)}
We note that if $J_u$ is a genuine modification of 
the matrix bundle at a semisimple element $s\in\Zur$, then one of the corresponding $J_u$-modules 
$i_P^G(\sigma\otimes\nu^{-1})$ with $\nu$ non-strictly positive must be reducible.

Recall from \cite[2.2 (3) and Cor. 2.6]{BK}, that the $H$-module
$i_P^G(\sigma\otimes\nu^{-1})$ is a semisimple $J$-module 
whenever $\nu$ is non-strictly positive; this means that the results of \cite[Section 4]{SolSurvey} apply. 

Indeed, if $\nu$ is strictly positive, then $i_P^G(\sigma\otimes\nu^{-1})$
is equal to a standard module, and is a simple $J$-module; this is the preface to the Langlands
classification. Therefore, given the datum $\xi=(P\supset M_P,\sigma\in\Ee_2(M_P),\nu\in\X(M_P))$
(consider the intermediate standard parabolic $P(\xi)$ defined by 
\[
P(\xi)^\vee=\sets{\alpha^\vee\in\Delta^\vee}{|\alpha^\vee(\nu)|=1},
\]
under $\X(M_P)\onto A_P^\vee$). We have $P(\xi)\supset P$, and $P(\xi)$ is defined so that $i_{P}^{P(\xi)}(\sigma\otimes\nu^{-1})$ 
is unitary and hence completely reducible 
with irreducible direct summands $\pi_i$ \cite[Prop. 3.20 (a)]{SolSurvey}. By Prop. 3.20 (b), (c) of \textit{loc. cit.}, 
each $i_{P(\xi)}^G(\pi_i)$ is a standard $H$-module and a simple $J$-module, for the same reason 
as at the start of this paragraph. Therefore all reducibility of 
$i_P^G(\sigma\otimes\nu^{-1})$ as a $J$-module comes from the first induction stage, which is the province 
of (projective representations of) the Knapp-Stein $R$-group. 


Therefore the divisors over which $J$ is a genuine modification must contain unramified characters
$\nu$ for which  $i_{P}^{P(\xi)}(\sigma\otimes\nu)$ is reducible; in particular
such that $P(\xi)\neq P$. Thus $\nu$ must at least have a nontrivial stabilizer where it is unitary,
and the product formula for the Harish-Chandra $c$-functions $c_\alpha^P$ of 
\cite[Section 4.1]{SolSurvey} then implies that the set of 
divisors on which the fibres of $J_u$ can have proper image in $\mMu$ is a subset of the 
divisors
\begin{equation}
\label{eqn poles of c-function general form}
\{\mathcal{V}(\alpha-1),\mathcal{V}(\alpha+1)\}_{\alpha}
\end{equation}
for $\alpha$ in the positive roots of $(P, A_P)$, where we write $\mathcal{V}$ for the vanishing set.
(Note that we are using the commuting, as opposed to $q$-commuting conventions for Kazhdan-Lusztig parameters, and so must perform a scalar change of coordinate, accounting for the residual coset of $\sigma$, for the poles of the $c$-function to take the form
\eqref{eqn poles of c-function general form}.)

In the usual coordinates for classical groups, these are divisors are of the form
\[
\{z_i=z_j^{-1}\}, \{z_i=-z_j^{-1}\}, \{z_i=-z_j\}, \{z_i=z_j\}, \{z_i=1\}, ~\{z_i=-1\}.
\]
The $c$-function controlling reducibility of $i_{P}^{P(\xi)}(\sigma\otimes\nu)$  always has poles along 
divisors of the form $V=V(\alpha-1)$, as 
$c_\alpha$ is always singular there \cite[Eq. (3.4)]{SolSurvey}. 
Let $\eta$ be the generic point of $V$.
No other factor $c_\beta$ of the $c$-function of Section 4.1 of \textit{loc. cit.} is singular at 
$\eta$, and so at $\eta$, we have $R_{\xi}=\{\pm \alpha\}$ in the notation of \textit{loc. cit.}
Hence we have $R_\xi=\stab{W_P}{\alpha}$, and $i_{P}^G(\sigma\otimes\nu)$ is generically a 
direct sum of simple $J$-modules indexed by $\Irr(R_\eta, \natural_\eta)$ for some $2$-cocycle $\natural_\eta$, by \cite[Theorem 4.2]{SolSurvey}.
That is, the image of $J_u|_\eta$ is as large as possible, and hence is equal to the image
of $\Ee_{J,u}|_\eta$.

\subsubsection{Proof of Proposition \ref{prop Ju and EJu fibre images agree at generic points} (b)}
\label{subsubsection proof of (b)}
Now we determine the scheme-theoretic support of $\mathcal{M}_u/J_u$. Let $D=\bigcup D_i$.
As the scheme-theoretic and set-theoretic supports are equal as topological 
spaces, it suffices to show that $\codim D\geq 2$ or that the $D_i$ are divisors and the scheme-theoretic support is reduced \textit{i.e.} to show 
that any matrix in
$\Mat_{\dim\pi_{ij}^I}\left(\Oo\left(\X(M_i)\right)^{W_i}\right)$ vanishing on $D$ lies in the image of $J_u$. 

Indeed, each $D_i$ can be taken to be of the following form: there is a Levi subgroup $M$ of $G$
and $\sigma,\nu$ as usual such that the image of $\Ee_{J,u}$ in
$\Mat_{\dim\pi^I}(\Oo(\X(M))^{W_M})$ under the map induced by $\pi$ is zero on fibres at $\nu$ at some entry 
at position $(d,\rho),(d',\rho')$. That is, $i_P^G(\sigma\otimes\nu)$ is reducible at $\nu$ and $D_i$
is a divisor in the connected component $A_M^\vee\git W_M$ of $\Zur\git\Zur$.
By the second sentence of Section \ref{subsection Proof of prop Ju and EJu imagess agree at generic points}, $D_i$ is equivalently characterized by just the vanishing on fibres of the image 
$(\iota_u\circ\eta_u)(J_u)_{(d,\rho), (d',\rho')}$ of $J_u$ in 
$\Mat_{\dim\pi^I}(\Oo(\X(M))^{W_M})$ under the map induced by $\pi$  at 
position $(d,\rho),(d',\rho')$. 

By Lemma \ref{lem mc are as implied by Lusztig}, the image of $J_u$ in 
$\Mat_{\dim\pi^I}(\Oo(\X(M))^{W_M})$ consists in each coordinate
of linear combinations of characters of $\Zur$-stabilizers of $Y_u\times Y_u$. Therefore
$D_i$ must be cut out as a topological space by the vanishing of these linear combinations, \textit{i.e.}, we have
\[
D_i=\mathcal{V}(\iota_u(J)_{(d,\rho), (d',\rho')})_{\mathrm{red}},
\]
constraining the equation of $D_i$. 

We now show that $\iota_u(J)_{(d,\rho), (d',\rho')}$ is already radical or has codimension at least $2$. The former case is more complicated; for example classical groups the description of characters of $\Zur$-stabilizers 
afforded by Lemma \ref{lem shape of characters of finite index Zur subgroups} implies that
$D_i$ is cut out by characters of the (s)pinorial fundamental representations of (s)pin groups, or the 
representations $V_{2m}^{\pm}|_{\SO_{2m}}$, which are prime elements of the representation ring. With this in hand it will be obvious that for any function $f$ vanishing on $D_i$, the virtual representation corresponding to $f$ defines
a class in
\[
K_{\Zur}(Y_{d}\times Y_{d'})
\]
whose matrix coefficient at $((d,\rho),(d',\rho'))$ is exactly $f$. 

\paragraph{Trivial orbits for classical groups.}
First, note that if $Y_d$ and $Y_{d'}$ are both singletons
and $D_i$ is given by vanishing on fibres at position $((d,\rho),(d',\rho'))=((d,\triv),(d',\triv))$
in a summand of $\mMu$, the claim is obvious. Indeed, the image of $J_u$ is nonzero but must vanish on $D_i$.
In particular, the trivial representation of $\Zur$ cannot be a matrix coefficient, and so 
$\Zur$ is centrally-extended at $(d,d')$. Thus the image of $J_u$ consists of all functions
in $R^1(\widetilde{\Zur})$ for some central extension. That is, 
for $\Gd$ classical, some 
factors of the $R\left(\Zur\right)$-module $\pi(t_{d,\rho})\pi(J_u)\pi(t_{d',\rho'})$
are generated, by Lemma \ref{lem shape of characters of finite index Zur subgroups}, over the corresponding factors $R(\OO_{m_i})$ or $R(\SO_{m_i})$ by the (half) (s)pinorial representations as in Examples \ref{example rep ring of orthogonal groups} and \ref{example rep ring of Pin}
$\Pi_i$, $\Delta_i$ or $\Delta_{i,1},\Delta_{i,2}$. Hence for a matrix in $\Ee_{2,(d,\triv),(d',\triv)}$ to vanish on $D_i$
is equivalent to that element lying in  
an union of vanishing sets $\mathcal{V}(\Pi_iR(\OO_{2m})$, $\mathcal{V}(\Delta_i R(\SO_{2m+1}))$, 
$\mathcal{V}(\Delta_{i,1}R(\SO_{2m}))$, or $\mathcal{V}(\Delta_{i,2}R(\SO_{2m}))$. But now consulting the 
formulas in Examples \ref{example rep ring of orthogonal groups} and \ref{example rep ring of Pin}, we see that any such union is reduced, and hence that for
any function vanishing on $D_i$, we may define a class in
\[
K_{\Zur}(Y_{d}\times Y_{d'})
\]
with precisely that function as its matrix coefficient. This proves (b) in this case.

\paragraph{Trivial orbits for exceptional groups.}
Now we suppose $\Gd$ to be exceptional and deal with $\Zur$ not occurring for classical 
groups. If, up to direct products with connected simply-connected groups,
$\Zur$ is the quotient of a connected simply-connected group ${\Zur}'$ by a central subgroup $\Z/2\Z$, then
$R^1(\widetilde{\Zur})$ is generated over $R(\Zur)$ by fundamental ${\Zur}'$ representations
$\Delta$ with specified Schur multiplier. Similar considerations apply to $\Zur=\SL_6/(\Z/2\Z)\rtimes(\Z/2\Z)$ acting trivially on some point $y\in Y_u$

If $\Zur=\PGL_3$ or $\Zur=\SL_6/(\Z/3\Z)$, then in fact the vanishing locus of the matrix coefficients 
of $J_u$---equivalently the reducibility locus of the corresponding $J_u$-representations---has 
codimension at least 2. For example, if $\Zur=\SL_6/(\Z/3\Z)$, then $\Zur$ has two nontrivial Schur multipliers 
identified with $\Hom{\Z/3\Z}{\Gm}=\Z/3\Z$, and the image of $J_u$ in $\mMu$ consists of 
matrices with entries in $\left(V(\varpi_1))+(V(\varpi_2)^2)\right)R(\Zur)$
at $(d,d')$ and $\left(V(\varpi_1)^2)+(V(\varpi_2))\right)R(\Zur)$
at $(d',d)$, or vice-versa. Thus the only point of $\Zur\git\Zur$ annihilating the image of $J_u$
at $(d,d')$ is the image of the origin in $\SL_6\git\SL_6$, and hence is of codimension 5 in
$\Zur\git\Zur$. Likewise for $\PGL_3$, the reducibility locus has codimension $2$. 
Similar considerations apply to 
$\Zur=\left(\SL_3\times\SL_3/(\Z/3\Z)\right)\rtimes\Z/2\Z$.

If $\Zur=(\Spin_{2r+1}\times\SL_2)/(\Z/2\Z)$, then 
\[
R^1(\widetilde{\Zur})=(\Delta_r\boxtimes\triv)R(\Zur)\oplus(\triv\boxtimes V(1))R(\Zur).
\]
If $\Zur=(\SL_2\times\SL_2\times\SL_2)/(\Z/2\Z\times\Z/2\Z)$, then there are three possible nontrivial
Schur multipliers, each yielding a module $R^1(\widetilde{\Zur})=\Delta R(\Zur)$, where $\Delta$
is the external tensor product of two trivial representations and one representation $V(1)$.

In all cases, we see that $D$ must be cut out precisely by the characters of the various generators $\Delta$. 
On the other hand, we obviously have $\Delta\in\Delta R(\Zur)$, and in each case,
the character of $\Delta$ is a square-free element of the UFD $R(\widetilde{\Zur})$ and so vanishes on
$D$ with order one, as required. 

The only cases with $\Gamma=\Sn_3$ in which central extensions can occur are $\Zur=N_{\SL_3}(T)$ and $\Zur=\left(\SL_2^{\times 3}/\Z/2\Z\right)\rtimes\Sn_3$. In this last case, a projective representation of any possible stabilizer is a representation of the stabilizer with odd highest $\SL_2$-weight in precisely one coordinate. For example, in the presence of central extensions, \eqref{eqn S3 matrix coefficients} would become, for a possible class $\mathcal{F}$
\[
\left(\mathrm{Tr}(s, V(1)\boxtimes 1\boxtimes 1)\pm\mathrm{Tr}(s, 1\boxtimes 1\boxtimes V(1)\right)=0,
\]
As $\mathcal{V}(V(1))$ is already reduced, we again conclude (b) in this case. For $\Zur=N_{\SL_3}(T)$,
central extensions are, by the proof of \cite[Lem. 2.1.22]{Propp}, groups
$\widetilde{\Zur}=(\Gm\times T_{\SL_3})\rtimes\Sn_3$ where $\Sn_3$ acts on $\Gm$ via a class $H^1(\Sn_3, X^*(T))=\Z/3\Z$ as in \textit{loc. cit.} and recalled below. By \cite[Prop. 4.1]{MatsTaka}, the characters of irreducible
projective representations with specified Schur multiplier are as follows. Let $\sigma$ be an irreducible representation of $\stab{\Sn_3}{\id\boxtimes\lambda_2}$. Then the character of the irreducible representation $\id\boxtimes\lambda_2\rtimes\sigma$ is 
\[
\frac{1}{\# \stab{\Sn_3}{(\id\boxtimes\lambda_2}}\sum_{\substack{\gamma_0\in \stab{\Sn_3}{\id\boxtimes\lambda_2}\\ \gamma_0~\text{conj. to}~\gamma}}\chi_\sigma(\gamma_0)\sum_{\substack{\gamma'\in\Sn_3 \\ \gamma'\gamma\gamma'^{-1}=\gamma_0}}(\id\boxtimes\lambda_2)(\gamma'\cdot(z,t)),
\]
for $z\in\Gm$, $t\in T_{\SL_3}$, and  $c\colon\Sn_3\to X^*(T_{\SL_3})$ the corresponding cohomology class, where then $\gamma'\cdot(z,t)=(zc(\gamma')(t), \gamma'\cdot t)$. 

As that the Schur multiplier depends only on $c$, we see that upon varying $\lambda_2$, the vanishing locus has codimension at least $2$ in $\Zur\git\Zur$ (note that vanishing is independent of $z$).

\paragraph{Nontrivial orbits for classical groups.}
If $G^\vee$ is a classical group,
then each factor of $\Sp_m$ in $\Zur$ acts trivially,
as does each factor of $\SO_{m_i}$ in $\OO_{m_i}=\SO_{m_i}\times\Z/2\Z$ for $m_i$ odd. Moreover, the Schur
multiplier for each factor of $\SO_{m_i}$ is constant on $Y_d\times Y_{d'}$. Therefore we have
\begin{equation}
\label{eqn classical groups off-diagonal}
K_{\Zur}(Y_d\times Y_{d'}^{\opp})=
\bigotimes_{a_i~\text{odd/even}}R(\Sp_{m_i})\otimes\bigotimes_{\substack{a_i~\text{even/odd} \\ m_i~\text{odd}}}R^i(\SO_{m_i})\otimes
K_{\mathcal{S}}(\mathbf{Y}_d\times \mathbf{Y}_{d'}^{\opp}),
\end{equation}
where 
\[\mathcal{S}=S\left(\prod_{_{\substack{a_i~\text{even/odd} \\ m_i~\text{odd}}}}\Z/2\Z\times\prod_{_{\substack{a_i~\text{even/odd} \\ m_i~\text{even}}}}\OO_{m_i}\right)
\]
and either $R^i(\SO_{m_i})=R(\SO_{m_i})$ or is $\Delta_{m_i}R(\SO_{m_i})$, depending on the Schur multiplier.
Given an orbit $\oO\subset Y_d\times Y_{d'}$, the stabilizer $\mathcal{S}'$ of any point in $\oO$ is a 
finite-index subgroup of $\mathcal{S}$, so characters of irreducible representations of $\mathcal{S}'$ are 
given by Lemma \ref{lem shape of characters of finite index Zur subgroups}.

An element of the $K$-theory factor of \eqref{eqn classical groups off-diagonal} can be specified by
providing a representation of $\mathcal{S}$ on a vector space $\bigoplus V_{(y_1,y_2)}$
commuting with bundle projection. When $\stab{\mathcal{S}}{(y_1,y_2)}$ is a proper subgroup of 
$\mathcal{S}$, if $V_{(y_1,y_2)}$ is simple, it is as in Lemma 
\ref{lem shape of characters of finite index Zur subgroups}. By the reasoning of the sentence preceding \eqref{eqn matrix coeff a priori no extension exceptional}, we obtain non-vacuous equations when $V_{(y_1,y_2)}$ does not 
extend to a representation of $\mathcal{S}$.
That is, assuming no central extensions, some tensor factors of 
the character of $V_{(y_1,y_2)}$ must lie in one of the ideals $V_{m_j}^{\pm}R(\SO_{m_j})$, where the 
$j$ depend only on the action of $\Gamma$ on $Y_d\times Y_{d'}$. 

In the presence of central extensions, we must instead chose classes in  
$\Delta_{-}R(\SO_{2r})$ or $\Delta_+R(\SO_{2r})$. For factors $\OO_{2r}$ in the stabilizer, 
we may chose either any representation $V$, or, in the presence of central extensions, a class in 
$R^1(\mathrm{Pin}_{2r})=\Pi_{2r} R(\OO_{2r})$. 

Therefore, by the discussion in Section \ref{subsection Construction of modules}, the image of $J_u$ in 
$\Ee_{2,(d,\rho),(d',\rho')}$ consists of functions belonging to
\begin{equation}
\label{eqn matrix coeff a priori no extension}
\bigotimes_{a_i~\text{odd/even}}R(\Sp_{m_i})\otimes\bigotimes_{\substack{a_i~\text{even/odd} \\ m_i~\text{odd}}}R^i(\SO_{m_i})\otimes
\sum_{\gamma\in\mathcal{S}''}\epsilon_\gamma\eta^t C_{\gamma}\left(V_1\boxtimes\cdots\boxtimes V_r\right)\boxtimes V_{r+1}\boxtimes\cdots\boxtimes V_{k}
\end{equation}
where $\mathcal{S}''=\Z/2\Z\times\cdots\times\Z/2\Z$ is the quotient by which $\Gamma$ acts,
$V_i$ is a representation of $\SO_{m_i}$ or of $\OO_{m_i}$ as described in the previous paragraph, 
$\epsilon_\gamma\in\sett{\pm 1}$ depends only on the action of $\Gamma$, and $\eta^t$ is a simple 
representation of $\mathcal{S}''$. 

Therefore $\Ee_{2,(d,\rho),(d',\rho')}$, and hence some divisor of reducibility $D_i$ is the reduced locus of
vanishing of all functions of the form \eqref{eqn matrix coeff a priori no extension}. Consulting
the presentations in Examples \ref{example rep ring of orthogonal groups} and \ref{example rep ring of Pin}, 
we see by degree considerations that the vanishing locus of \eqref{eqn matrix coeff a priori no extension}
is already reduced.

Therefore Proposition \ref{prop Ju saturated lower modification} (a) and Proposition \ref{prop Ju and EJu fibre images agree at generic points} (b) are proved in these cases.

\paragraph{Nontrivial orbits without central extensions for exceptional groups.}
Now we consider nontrivial orbits without central extensions, with $\Gamma=\pi_0(\Zur)=\Z/2\Z$ acting nontrivially on $X\times Y$ for $1\leq \#X, \#Y\leq 2$ with any stabilizer being ${\Zur}^\circ$. 
In this case, for $\Ff\in K_{\Gamma}(X\times Y)$, write $f_{ij}$ 
for the function $s\mapsto\trace{\Ff_{(x_i,y_j)}}{s}$ for semisimple 
$s\in\Zur$ fixing $X$ and $Y$. In the basis $\{\delta_{y_1}+\delta_{y_2}, \delta_{y_1}-\delta_{y_2}\}$ of functions on $Y$, a sheaf $\Ff$ supported
on the orbit $\sett{(x_1,y_2),(x_2,y_1)}$ in $X\times Y$ acts by 
\[
\begin{pmatrix}
\frac{f_{12}+f_{21}}{2} & \frac{f_{12}-f_{21}}{2}\\
\frac{f_{12}-f_{21}}{2} & \frac{f_{12}+f_{21}}{2}
\end{pmatrix},
\]
whereas a sheaf supported on $\sett{(x_1,y_1),(x_2,y_2)}$ acts by 
\[
\begin{pmatrix}
\frac{f_{11}+f_{22}}{2} & \frac{f_{11}-f_{22}}{2}\\
\frac{f_{11}-f_{22}}{2} & \frac{f_{11}+f_{22}}{2}
\end{pmatrix}
\]
with $f_{ji}=C_{\gamma}(f_{ij})$ etc. We may assume that
$C_\gamma(f_{ij})\neq f_{ij}$ or $C_\gamma(f_{ij})\neq -f_{ij}$  as representations of the stabilizer; otherwise 
we obtain the zero function, contradicting the generic irreducibility of parabolic inductions. 

Therefore the image of $J_u$ in 
$\Ee_{2,(d,\rho),(d',\rho')}$ consists of functions of the form
\begin{equation}
\label{eqn matrix coeff a priori no extension exceptional}
R(\Zur)\otimes\left(V\pm C_\gamma(V)\right)
\end{equation}
for $V\neq \pm C_\gamma(V)$ as functions of $s$.


Therefore if one of these matrix coefficients is 
forced to vanish on some $D_i$ because of reducibility, then $\rho\neq\rho'$
and we must have that $D_i$ is cut
out by the character of $V-C_\sigma(V)\in R({\Zur}^\circ)$.

If ${\Zur}^\circ$ is simple and simply-connected, then it is easy to see that for every 
vertex $i$ of the Dynkin diagram of ${\Zur}^\circ$ such that $\gamma(i)\neq i$, $V(\varpi_i)-V(\varpi_{\gamma(i)})$
vanishes on $D_i$, and in fact we have
\begin{equation}
\label{eqn Di intersection exceptional}
D_i=\left(\bigcap_{\sets{i}{\gamma(i)\neq i}}\mathcal{V}\left(V(\varpi_i)- V(\varpi_{\gamma(i)})\right)
\right)_{\mathrm{red}}.
\end{equation}

If $\Zur=(G_2\times G_2)\rtimes\Z/2\Z$, then 
we have instead $\mathcal{V}\left(V_1(\omega_i)-V_2(\omega_i)\right)$ in 
\eqref{eqn Di intersection exceptional}, where subscripts refer to 
factors in the identity component. In all cases, it is clear that we can define
a class in $J_u$ giving any function in the right hand side of  
\eqref{eqn Di intersection exceptional} (if \eqref{eqn Di intersection exceptional} doesn't already force the reducibility locus to have dimension at least $2$), proving (b) in these cases.

If $\Zur$ is one of the groups in Examples \ref{example rep ring of orthogonal groups} and \ref{example rep ring of Pin}, it is clear from the presentations those examples that any $V\pm C_\gamma(V)$ vanishing on $D_i$ can be realized as a matrix coefficient of $J_u$. If $\Zur$ is 
is presented as a quotient
\[
1\to \Z/n\Z\to T\times[\Zur,\Zur]\to\Zur\to 1
\]
of a simply-connected semisimple group with no outer automorphisms by a central torus $T$, then 
\eqref{eqn matrix coeff a priori no extension exceptional} is generated
over $\Zur$ by the ${\Zur}^\circ$ character $t-t^{-1}$, where we write $t$ for the character
of $\Zur$ identified with the generator of the character group of $T/P$,
which must then cut 
out $D_i$ exactly, vanishing on it with order one. If $\Zur$ is a nonsplit extension of $\Z/2\Z$ by $T_2$, then we see that 
$D_i=\bigcap_{i=1}^2\mathcal{V}\left(t_i-t_i^{-1}\right)$. If
$\Zur$ is an extension of a pushout of by a central torus and $\SL_3$, then by the same logic as above, we have
\[
D_i=\left(\mathcal{V}\left(t-t^{-1}\right)\cap\mathcal{V}\left(V(\varpi_1)-V(\varpi_3)\right)\right)_{\mathrm{red}}
\]
and the expression inside right hand side is already reduced; thus any function vanishing on $D_i$ is a matrix coefficient of an element 
of $J_u$.

Now we consider $\Gamma=\Sn_3$. First consider sets $X,Y$ with $\# Y_d=\# Y_{d'}=3$. By Section \ref{subsubsection Equivariant K-theory of square of a finite set}, it is enough to consider the action of $K_{\Zur}(Y_d\times Y_{d'})$ on $\mu_2\colon Y_{d'}\to \std$ and $\mu_\triv\colon Y_{d'}\to\C$. For any $\Ff$, the target of $F\mu$ is that of $\mu$, so there are no cross terms; we only need to consider matrix coefficients at $\left(d,\triv), (d',\triv)\right)$ and $\left((d, \std,), (d', \std)\right)$.
In both cases, $\Oo_\Delta\mu=\mu$. Therefore no vanishing is possible in this case absent central extensions. If $\#Y_{d'}=2$ and $\#Y_d=3$, then $Y_d\times Y_{d'}$ is a transitive $\Sn_3$-set and for $\mu\colon Y_{d'}\to\C$ we have 
\begin{equation}
\label{eqn S3 matrix coefficients}
\Ff\mu(y_i)=\Ff_{i,1}\mu(y'_1)+\Ff_{i,2}\mu(y'_2)=\left(\Ff_{i,1}+\pm C_{(23)}(\Ff_{i,2})\right)\mu(y'_1),
\end{equation}
where the sign $\pm$ is taken depending on whether $\mu$ transforms according to $\sgn$ or $\triv$. If $\Zur=\Spin_8\rtimes\Sn_3$, we obtain the same condition \eqref{eqn Di intersection exceptional} and conclude similarly.
If $\Zur=(\SL_2^3)/\Z/2\Z\rtimes\Sn_3$, then in the absence of central extensions, we conclude as in the sentence following \eqref{eqn Di intersection exceptional}. If $\Zur=N_{\SL_3}(T)$, then we obtain $\mathcal{V}(\chi-(12)\chi)$, where the action on characters $\chi$ of $T=T^2$ is given by the Weyl group action of $\SL_3$ and also conclude that $D_i$ is reduced.
The cases $\# Y_d=1, \#Y_{d'}=3$ and vice-versa are similar; in the latter there can be no vanishing absent central extensions. The stabilizer in this case is $\stab{y_i}{\Zur}=\bangles{(\Zur)^0, \begin{pmatrix}
j&k
\end{pmatrix}}$.

Now consider $\#Y_d=6=\# Y_{d'}$. We are examining linear maps $t_{d'}\pi\to t_d\pi$.
 
First suppose that, for all $\chi$, we have $\dim t_d\pi=1$. By Section \ref{subsection Construction of modules}, $\dim_{t_{d'}}\pi=1$ also, and absent central extensions, the matrix coefficient in question cannot vanish.

Next suppose $\dim t_{d}\pi=2$. Again $\dim t_{d'}\pi=2$ also. The class $[\Oo_{\Gamma\cdot(y,(23)y)}]$ has constant matrix coefficient equal to $1$ at both positions $\left((d', \std'), (d, \std)\right)$ and $\left((d', \std), (d, \std')\right)$ 

Otherwise $\dim t_{d}\pi=6=\dim t_{d'}\pi$ and elements of $t_{d}J_ut_{d'}$ act as $6\times 6$ matrices on $\C[Y]$. We deal with the only possible reducibility, at some point where $\pi_0(Z(u,s))=\Z/2\Z$, as for classical groups.


%
\paragraph{Nontrivial orbits with central extensions for exceptional groups.}
Now suppose ${\Zur}^\circ$ is $\SL_6/(\Z/2\Z)$, $\PGL_2$, or $\PGL_3$. In this case, viewing projective
representations as honest representations of $\SL_n$, $n=6,2,3$, the nontrivial action of $\Gamma$
again implies via Clifford theory as in Section \ref{subsubsection rep rings of disconnected and Clifford theory} that the image of $J_u$ in 
$\Ee_{2,(d,\rho),(d',\rho')}$ consists of functions of the form 
\eqref{eqn matrix coeff a priori no extension exceptional}, now with the following additional constraints:
If for $n=6$ and the Schur multiplier is trivial, we may have only $V(\varpi_2)-V(\varpi_4)$; if the Schur multiplier is nontrivial, we may have only $V(\varpi_1)-V(\varpi_5)$. If ${\Zur}^\circ=\PGL_3$,
then the $\SL_3$ representations $V(\varpi_1)$ and $V(\varpi_2)$ have distinct 
Schur multipliers as $\PGL_3$-representations, and we see that projective
representations cannot in fact occur if $\Gamma$ acts non-trivially. Hence
we are back in the case of the last paragraph.

Finally, if $\Zur=\left((\SL_3\times\SL_3)/(\Z/3\Z)\right)\rtimes\Z/2\Z$, where the 
component group acts by exchanging factors, then matrix coefficients of $J_u$
in the case of nontrivial orbits are supported on the direct factor $R((\Zur)^\circ)^{\Z/2\Z}$ and
are without loss of generality of the form
$f_if_j+C_\gamma(f_if_j)$, where $f_i\neq f_j\in R(\SL_3)$ are such that $f_if_j\in R((\Zur)^\circ)$. 
Projective representations can occur, with the possibilities being much the 
same as dealt with above: matrix coefficients of $J_u$ in this case are generated over
$R((\Zur)^\circ)^{\Z/2\Z}$ by $\sett{V(\varpi_1)\boxtimes\triv\pm\triv\boxtimes V(\varpi_1), V(\varpi_2)^2\boxtimes\triv\pm\triv\boxtimes V(\varpi_2)^2}$, for one nontrivial Schur multiplier, and 
similarly for the other. Hence we are already in case (a) of Proposition \ref{prop Ju saturated lower modification}.  Note that $V(\varpi_2)^2\boxtimes\triv\pm\triv\boxtimes V(\varpi_2)^2$ is
reducible in $R((\Zur)^\circ)$ but not in $R((\Zur)^\circ)^{\Z/2\Z}$. 

Now we consider $\Gamma=\Sn_3$ with central extensions and nontrivial orbits. Reducedness for stabilizers in
$\Zur=\left(\SL_2^{\times 3}/\Z/2\Z\right)\rtimes\Sn_3$ and for $N_{\SL_3}(T_{\SL_3})$ is dealt with similarly to the case of classical groups. 


%
We have shown that for every parabolic induction $\pi$ as considered at the beginning of this section, either the reducibility locus has codimension at least two, or is a union of divisors which is the scheme-theoretic support of $\Ee_{J_u}/J_u$. Therefore we have part (b) of the Proposition.
\end{proof}
\begin{rem}
We have already remarked that when $\Zur$ is connected, reducibility of families of tempered representations
with parameter containing $u$ implies the appearance of central extensions. We see now from the proof
that in this case the corresponding poles of the $c$-function in the case of classical groups are all of the 
form $z_i=\pm 1$, because these are the only divisors cut out by the characters of (s)pinorial representations. 
\end{rem}
\subsection{Surjectivity for $\GL_n$, $\PGL_n$, and Xi's theorem}
\label{subsubsection Surj for GLn}
As remarked in the introduction, we have avoided using any of the results from \cite{XiAMemoir} or 
\cite{BO} for $G=\GL_n$ and $G=\PGL_n$. Recall that in this case, Levi subgroups of $G$ and unipotent conjugacy classes
in $\Gd$ are in bijection, and for every $u$ there is a unique family of parabolic inductions
$\pi=i_{P}^G(\St_{M_P}\otimes\nu)$ such that $\rank\left(\pi(t_d)\right)=1$ for all $d\in\cc(u)$
by Corollary 5 of \cite{Plancherel}. We proved the Corollary, thereby activating Proposition \ref{prop Ju saturated lower modification} for $\GL_n$, without any reference to $K$-theory.
Clearly the same is true for $G=\PGL_n(F)$.

Let $u=(a_1,\dots, a_m)$ be a unipotent conjugacy class in $\Gd$. Let $(m_1,\dots, m_k)$ denote the 
multiplicities in $u$. For $G^\vee=\GL_n$, we have
\[
\Oo(\X(M))^{W_M}= R\left(\GL_{m_1}\times\cdots\times\GL_{m_n}\right),
\]
so by Lemma \ref{lem regularity of matrix coefficients}, \eqref{eqn nested modification} specializes to 
\[
\begin{tikzcd}
J_u\arrow[r, "\eta_{u}", hook]&\mathcal{E}_{J,u}\arrow[r, "\iota_{u}", hook]&
\Mat_{N}(\Oo(\X(M)\git W_M))=
\Mat_{N}\left(R(\GL_{m_1}\times\cdots\times\GL_{m_k})\right),
\end{tikzcd}
\]
where $N=n!/(a_1!\dots a_k!)$. For $G^\vee=\SL_n$, we have $\Zur=S(Z_{\GL_n}(u)^{\mathrm{red}})$
\cite[Thm.6.1.3]{CM}.
\begin{theorem}
\label{thm Ju matrix algebra for GLn}
Let $G=\GL_n(F)$ or $\PGL_n(F)$. Then $\iota_u\circ\eta_u$ is an isomorphism for all unipotent
conjugacy classes $u\in\GL_n(\C)$. In particular, $\iota_u$ and $\eta_u$ are each isomorphisms.
If $G$ is semisimple of type $A$, then $\eta_u$ is an isomorphism.
\end{theorem}
\begin{proof}
It suffices to show that $\iota_u\circ\eta_u$ is an isomorphism.
Recall that for any parabolic subgroup $P$ of $\GL_n(F)$, 
$\Ee_2(M_P)^I/\X(M_P)=\{\St_{M_P}\}$, and for all $\mu\in\X(M_P)$ 
non-strictly positive, $i_P^G(\St\otimes\nu)^I=K(u,s,\triv)$ is a standard $H$-module, and hence 
a simple $J$-module. Therefore $\iota_u\circ\eta_u$ is surjective on fibres, hence is surjective by Nakayama's 
lemma. Hence $\iota_u\circ\eta_u$ is an isomorphism. 

As $R$-groups for $\PGL_n(F)$ are also trivial (correspondingly, as remarked in Section \ref{subsubsection 
proof of (b)}, the groups $\Zur$ are connected modulo the centre and do not admit projective representations),
the same proof works for $\PGL_n(F)$. By Section \ref{subsubsection reduction to semisimple groups},
we obtain that $\eta_u$ (but not, of course, $\iota_u$) is an isomorphism for all semisimple groups
of type $A$.
\end{proof}
\begin{rem}
\label{rem isomorphism for GLn is Xi a posteriori}
Theorem \ref{thm Ju matrix algebra for GLn} is weaker than the main result of \cite{XiAMemoir}, as we do not
prove that $\iota_u\circ\eta_u$ is an isomorphism of based rings. By construction,  $(\iota\circ\eta)(t_{d'}Jt_d)$ consists of matrices with unique nonzero entry at the same coordinates, but we fall short of proving that the unique nonzero entry of $(\iota_u\circ\eta_u)(t_w)$ is given by the class of a simple representation. However, it follows \textit{a posteriori}
from Section \ref{subsection Construction of modules} and \cite[Section 10]{cellsIV}  that 
$\iota_u\circ\eta_u$ is in fact equal in both cases to the isomorphism of \cite{XiAMemoir}.
\end{rem}
\section{The rigid determinant}
\label{section rigid det}
In this section we will introduce the rigid quotient, cocentre, and determinant of \cite{CH}.

\subsection{Cocentres and rigid cocentres}

\subsubsection{Cocentres and rigid cocentres for $q>1$}
\label{subsubsection cocentres and rigid cocentres for q>1}
The connection between the rigid cocentre and the ring $J$ is immediate from the existing literature
when $\bq=q>1$. Indeed, let $K_{0,\C}(\HH|_{\bq=q}-\Mod)^{*,\,\mathrm{good}}$ be the space of good forms in 
the sense of the trace Paley-Wiener theorem \cite{BDeligneKazhdanTrace}. These are the linear functionals 
$\varphi$ such that if $\pi=i_P^G(\sigma\otimes\nu)$ for $\sigma\in\Ee_2(M_P)$, then $\varphi([\pi])$ depends algebraically on the unramified character $\nu$. In particular, 
forms that are constant in $\nu$ are good.	

Lusztig's map $\phi$ induces the diagram
\begin{center}
\begin{tikzcd}
\Hrs\arrow[dd, "\sim", near start]\arrow[rr, hook]\arrow[dr, "\phir_q", "\sim" near end]&&\hhh(H)\arrow[dd, "\sim", near start]\arrow[dr, "\bar{\phi_q}", "\sim" near end]&\\
&\Jrs\arrow[dl, "\sim"]\arrow[rr, hook, crossing over]&&\hhh(J)\arrow[dl, "\sim"]\\
K_{0,\C}(\HH|_{\bq=q}-\Mod)^{*,\,\mathrm{const}}\arrow[rr, hook]&&
K_{0,\C}(\HH|_{\bq=q}-\Mod)^{*,\,\mathrm{good}},
\end{tikzcd}
\end{center}
where the right vertical morphism is an isomorphism by \cite{BDeligneKazhdanTrace}, the right triangle exists by \cite[Lemma 4]{Plancherel}, and 
the map $\bar{\phi}_q$ induced by $\phi_q$ is an isomorphism by \cite{BDD}. The rear rectangle is Cartesian by definition. The diagram gives
\begin{prop}
\label{prop phir isomorphism q>1 and density for J}
\begin{enumerate}
\item 
There is an isomorphism $\phir_q\colon\Hrs\to\Jrs$ when $q>1$.
\item
Density of characters hold for $J$ for $q>1$.
\end{enumerate}
\end{prop}
Equivalently, the spaces
\[
\Hrs:=\sets{h\in \hhh(H)}{\trace{-}{h}|_{\RHdi}=0}
\]
and 
\[
\Jrs:=\sets{j\in \hhh(J)}{\trace{-}{j}|_{\RHdi}=0}
\]
of forms descending to the rigid quotient $\RHr:=\RH/\RHdi$ are isomorphic, 
where we recall from \cite{CH} that, writing $i_\theta$ for the corresponding parabolic induction	 
\begin{multline*}
\RHdi=\spn{\sets{i_\theta(\sigma)-i_\theta(\sigma\otimes\chi)}{\theta\subsetneq\Pi,~\sigma\in\mathcal{R}(\HH_\theta),~\chi\in\HomOver{\Z}{X\cap\Q R/X\cap\Q \theta}{\C^\times}}}
\\
\subset
\RH:=K_{0,\C}(\HH|_{\bq=q}-\Mod).
\end{multline*}
This picture also makes transparent the upper-triangularity of the Gram matrix of the rigid pairing 
proven in \cite{CH}. Namely, it will follow from the relationship between $J$-modules and 
$\HH$-modules together with upper-triangularity of Lusztig's map $\phi$ with respect to the Kazhdan-Lusztig basis
(although in the sequel we will opt to work with lower-triangular matrices).
Modulo the isomorphism $\phir$ to be defined below, the rigid 
pairing will actually have a block diagonal Gram matrix, with blocks indexed by the two-sided cells of 
$W$. 

For every $q>1$, we may chose the same basis of tempered $H$-modules
of $\RHr$. We do so now.
\begin{dfn}
\label{dfn alpha ordered basis}
Let $\alpha$ be a fixed basis of $\mathcal{R}_{\mathrm{rigid}}$ consisting of 
tempered representations. We further chose an ordering on $\alpha$ compatible with Lusztig's
$a$-function under the Kazhdan-Lusztig parametrization recalled in Section \ref{subsubsection and affine and asymptotic Hecke algebras}.
\end{dfn}
%
%
\subsubsection{Cocentres of $\HH$ and $J\otimes\mathcal{A}$ as $\mathcal{A}$-modules}
\label{subsubsection cocentres of HH and J}
In principal, the spectrally-defined space $\Hrs$, and hence $\Jrs$, could depend on $q$, but as recalled in the introduction, a main point of \cite{HN} (in the unequal-parameters case, \cite{CH}) is that this dependence is minimal.
In this section we recall two facts about
$\hhh(\HH)$ and $\hhh(J\otimes\mathcal{A})$ needed to state this independence.
\begin{theorem}[\cite{HN}, Theorem 6.7]
\label{thm CH cocentre is free}
For any conjugacy class $\oO$ in $W$, the image $T_{\oO}$ in $\hhh(\HH)$ of $T_w$ for any $w\in \oO$ of minimal length is well-defined, and
$\hhh(\HH)$ is a free $\mathcal{A}$-module with basis 
$\{T_{\oO}\}_{\oO}$.
\end{theorem}
\begin{theorem}[\cite{BDD} and appendix by Bezrukavnikov-Braverman-Kazhdan]
\label{thm BDD cocentre isomorphism}
The map
\begin{equation}
\label{eq BDD cocentre isomorphism}
\bar{\phi}\colon \hhh(\HH)\left[\frac{1}{P_{W_f}}\right]\onto \hhh(J\otimes_\C\mathcal{A})\left[\frac{1}{P_{W_f}}\right]
\end{equation}
is surjective, where $P_{W_f}$ is the Poincar\'{e} polynomial of $W_f$. If $q$ is moreover an admissible 
parameter in the sense of \cite[Def. 6.6]{CH}, in particular if $q$ is not a root of unity, then 
$\bar{\phi}$ specializes to an isomorphism $\bar{\phi}_q$. 
\end{theorem}
Surjectivity is proven in the body of \cite{BDD}, injectivity is proven in the appendix. 

It is natural that $\bar{\phi}$ is not an isomorphism outside the localization
$\mathcal{A}[1/P_{Wf}]$, as in this case the connection between the representation theory of 
$\HH|_{\bq=q}$ and $J$ is lost \cite{XiJAMS}. However, it is equally natural to expect that an interesting
statement survives without the admissibility hypothesis. 
\subsubsection{The rigid cocentre as an $\mathcal{A}$-module}
Following \cite{CH}, we recall that the \emph{rigid cocentre} is the free $\mathcal{A}$-submodule
of $\hhh(\HH)$ defined as
\[
\Hr:=\spn{\sets{T_\oO}{\oO\in\mathrm{cl}(W)_0}},
\]
where $\mathrm{cl}(W)_0$ is the finite set of conjugacy classes in $W$ with zero Newton point, in the sense
of \textit{op. cit.} Roughly, these are the conjugacy classes of finite-order elements of $W$.

It is nontrivial to relate $\Hr$ to the space of forms $\Hrs$ on $\RH$ descending to $\RHr$. 
\begin{theorem}[\cite{CH}, Theorem 7.1, Theorem 8.2]
\label{thm CH separation}
Suppose that $q$ is admissible. Then $\Hrq=\Hrs$, and
the free $\mathcal{A}$-module $\Hr$ has rank equal to $\dim_\C\Hrq$ for any admissible $q$.
\end{theorem}
%
%
The fact that $J$ is defined 
over $\Z$ will mean that one gets the $J$-analogue $\Jr$ of $\Hr$ from the $J$-analogue
$\Jrs$ of $\Hrs$ more-or-less automatically. 

We emphasize that while both $\hhh(\HH)$ and $\hhh(J\otimes\mathcal{A})$ are free $\mathcal{A}$-modules,
and hence have descriptions independent of $q$, the definition of $\Jrs$
from Section \ref{subsubsection cocentres and rigid cocentres for q>1} makes sense \textit{a priori} 
only for $q>1$. However, $\Jrs$ does admit a description independent of $q$, corresponding to
that of $\Hr$. This amounts to giving a basis of $\Jrs$, which we now do.

Let $q>1$ and recall the elements $t_{\omega d,\rho}$ from Section \ref{section Rank 1 idempotents}  defined for $d\in\cc$ such that $\pi_0\left(\Zur(u(\cc))/Z(\Gd)\right)$ is abelian. Write $U=\sets{u}{\Zur~\text{is finite}}$. 
For such $u\in U$, rigidity is vacuous, and the corresponding part of the rigid cocentre is just the corresponding part of the entire cocentre. Let $S_2\subset\bigoplus_{u\in U} \hhh(J_u)$ be a basis of the space of forms annihilating $K(u',s,\rho)$ for $u'\not\in U$, which
exists by Theorem \ref{thm BDD cocentre isomorphism}. Recall the notation $t_{\omega d, \rho}$ from the proof of Corollary \ref{cor classical idempotents are integral}, and set
\[
S:=\sets{[t_{\omega d,\rho]}}{d\cap\mathcal{D}\setminus\cc(U),~\rho\subset\C[Y_d],~\omega\in\Irr(Z(G^\vee)) }\cup S_2\subset\hhh(J).
\]
When $\pi_0\left(\Zur(u(\cc))/Z(\Gd)\right)=\Sn_3$, put $S$ to be the union $S_2$ and the set of all $\phi(T_\omega)\star t$, where $t$ is any of the idempotents for $\Sn_3$ constructed in Section \ref{subsubsection Equivariant K-theory of square of a finite set}.

\begin{lem}
\label{lem Jrs basis and rank}
The set $S$ specializes to a spanning set of $\Jrs$ for  every $q>1$. We may select the same subset of $S$ as a 
basis of $\Jrs$ for every $q>1$.
\end{lem}
\begin{proof}
Clearly $\spn_\C S$ is contained in $\Jrs$. 
On the other hand, let $\pi=\Ind_P^G(\sigma\otimes\nu)$ be tempered and irreducible. If under the isomorphism
\eqref{eqn pi(trho) define rational basis}, $j\in t_dJt_d$ is such that $\pi(j)\mu_\rho=\delta_{\rho,\rho'}$
for some $\rho'$, then there is $j'\in S$ such that $\pi(j')=\pi(j)$ by construction of the idempotents in Section \ref{subsubsection Equivariant K-theory of square of a finite set} and the definition of $S$; the $\mu_{\rho'}$ in question give a basis of the underlying vector space of $\pi$ for all $\nu$. Now together
with \eqref{eqn EJu to single Mat definition}, \eqref{eqn trho maps to monomial matrix}, and the definition of $S_2$, it becomes
clear that $\Jrs\subset\spn_\C S$; the $t_{d,\rho}$ can separate irreducible 
direct summands at points of reducibility at $\pi$ and the $t_{\omega d,\rho}$ 
can distinguish between non-identified twists in $\RHr$.
\end{proof}
\begin{dfn}
\label{dfn beta and Jr}
The \emph{rigid cocentre} $\Jr$ \emph{of} $J$ is $\spn_\Z S\subset\hhh(J)$. Let $\beta\subset S$ be a subset 
specializing to a basis of $\Jrs$ for all $q>1$ as in Lemma \ref{lem Jrs basis and rank}, chosen once-and-for-all.
\end{dfn}
\begin{rem}
\label{rem extracting basis} 
It would be interesting to extract $\beta$ canonically in terms of our chosen realization of $Y_u$ in terms of 
Lusztig's canonical basis \cite{BL}.
\end{rem}
\begin{cor}
\label{cor Jr and Hr free of equal rank}
$\Jr\otimes\mathcal{A}$ and $\Hr$ are free finitely-generated $\mathcal{A}$-modules of equal rank.
\end{cor}
\subsubsection{The rigid cocentre over $\mathcal{A}[\frac{1}{P_{W_f}}]$}
Having extracted Corollary \ref{cor Jr and Hr free of equal rank} from studying the case $q>1$, we can 
again relax the hypothesis to $q\in D(P_{W_f})$.
\begin{prop}
\label{prop phir restricts to isomorphism}
The map $\overline{\phi}$ restricts to an isomorphism
\[
\phir\colon \Hr\left[\frac{1}{P_{W_f}}\right]\to\Jr\otimes_\C\mathcal{A}\left[\frac{1}{P_{W_f}}\right].
\]
\end{prop}
\begin{proof}
By Theorem \ref{thm CH cocentre is free}, $\hhh(\HH)$ is a free $\mathcal{A}$-module, as obviously is 
$\hhh(J\otimes\mathcal{A})$. By the surjectivity of $\bar{\phi}$ from Theorem 
\ref{thm BDD cocentre isomorphism}, for any $j\in\Jr$, we have 
$
\phi^{-1}(j)=\sum_{\oO}a_{\oO,j}T_{\oO}
$
for $a_{\oO,j}\in\mathcal{A}\left[\frac{1}{P_{W_f}}\right]$. By Proposition \ref{prop phir isomorphism q>1 and density for J}, 
for any conjugacy class $\oO$ with nonzero Newton point, we have that 
$a_{\oO,j}(q)=0$  for a set of $q\in 
D(P_{W_f})$ with an accumulation point. Therefore $a_{\oO,j}=0$ for all such 
$\oO$. This says that $\bar{\phi}$ restricts to a map 
as in the statement, that is moreover surjective. As both $\Hr$ and $\Jr$ are finitely-generated free 
$\mathcal{A}$-modules, we may conclude by noting that by Corollary 
\ref{cor Jr and Hr free of equal rank}, their localizations have equal ranks.
\end{proof}
\subsection{The rigid determinant}
\subsubsection{The rigid pairing and rigid determinant}
\begin{theorem}[\cite{CH}, Theorem 7.6 (a), Theorem 8.2]
\label{thm CH rigid pairing}
Suppose that $q$ is admissible. There is a perfect pairing
\[
\Hrq\otimes\RHr\to\C
\]
induced by 
\[
(h, [\pi])\mapsto\trace{[\pi]}{h}.
\]
\end{theorem}
The theorem applies in particular for $q>1$ and our chosen basis $\alpha$ of $\RHr$ consisting
of tempered representations. The resulting Gram matrix extends to all of $\Spec\mathcal{A}$, as the trace of any element of $\Hrq$ is a Laurent polynomial in $q$.
\begin{dfn}
Let $\gamma$ be a basis of $\Hr$. Then \emph{rigid matrix for the basis} $\gamma$ is the extension to 
$\Spec\mathcal{A}$ of the Gram matrix of the pairing from Theorem \ref{thm CH rigid pairing} for all admissible $q$, where we use the basis $\alpha$ for $\RHr$. The \emph{rigid determinant for basis} $\gamma$ is the determinant of the rigid matrix for basis $\gamma$.
\end{dfn}
\subsubsection{The rigid pairing for $\Jrs$ for $q>1$}
As $t_wE(u,s,\rho)=0$ unless $w$ is in $\cc(u)$, the following is immediate from the direct sum decomposition 
of $J$ and the definition of the basis $\alpha$ of the rigid quotient.
\begin{lem}
\label{lem J pairing is block diagonal}
Let $\alpha$ and $\beta$ be as in Definitions \ref{dfn alpha ordered basis} and \ref{dfn beta and Jr}, 
respectively, and let $q>1$. The Gram matrix $B_\beta$ for the pairing
\[
\Jrs\otimes\RHr\to\C
\]
induced by 
\[
(j, [\pi])\mapsto\trace{[\pi]}{j}
\]
is block-diagonal with integer entries.
\end{lem}
We view $B_\beta$ as giving the linear map 
\[
\RHr\to(\Hr)^*\simeq (\Jr)^*
\]
\[
[\pi]\mapsto\left([j]\mapsto\trace{[\pi]}{[j]}\right)
\]
with respect to the bases $\alpha$ and $\beta$. 

\subsubsection{Aside: The rigid cocentre of $K_\Zz(Y\times Y)$}
In the abelian case, one can cleanly define the rigid cocentre for any ring $K_\Zz(Y\times Y)$.
Assume that $Y$ is transitive and $\Zz^\circ$ is semisimple of adjoint type. We define the \emph{rigid cocentre} of $K_{\Zz}(Y\times Y)$ as the space of elements $j$ in 
$\hhh\left(K_{\Zz}(Y\times Y)\right)$ such that $\trace{E_{s,\rho}}{j}$ is locally constant in $s$, in the 
following sense: by Borel-de-Siebenthal theory, there are finitely-many isomorphism classes of 
centralizers of semisimple elements in $\Zz$. For a given isomorphism type $C$ of the latter, and 
$\rho\in\mathrm{Irr}(\pi_0(C))$, 
we ask that $\trace{E_{s,\rho}}{j}$ be defined for all $s$ such that $Z_{\Zz}(s)\simeq C$,
and also that the trace be locally constant in $s$. If this holds for each $\rho$, then 
$j$ belongs to the rigid cocentre. Note also int the abelian case, any $\rho\in\mathrm{Irr}(\pi_0(\Zur)/Z(G^\vee))$ will restrict irreducibly to $\pi_0(C)$.
\begin{lem}
\label{lem trho span rigid cocentre of K-theory}
The elements $t_\rho$ span the rigid cocentre of $K_{\Zz}(Y\times Y)$. In fact, if 
$j\in K_{\Zz}(Y\times Y)$ is such that $j\mu_\rho=\delta_{\rho,\rho'}\mu_{\rho'}$
for all $\rho,\rho'\subset\C[Y^s]$, then $j=t_{\rho'}$.
\end{lem}
\begin{proof}
Fix $j=\sum_{\gamma\in\Gamma}\alpha_\gamma V_\gamma$ for bundles $V_\gamma$ supported on 
$\Gamma\cdot (y_1,\gamma y_1)$, and $s\in \Gg$ such that $Y^s\neq\emptyset$. Then 
\[
j\star\mu_\rho =
\sum_{\gamma\in\Gamma}\alpha_\gamma V_\gamma\star\mu_\rho
=
\sum_{\gamma\in\Gamma}\alpha_\gamma V_\gamma\star\sum_{g\in\Gamma}\rho(g)\delta_{g y_1}
=
\sum_{\gamma\in\Gamma}\sum_{g}\alpha_\gamma\trace{V_{\gamma^{-1}gy_1,gy_1}}{s}\rho(g)\delta_{\gamma^{-1}g y_1}.
\]
If $j\star\mu_{\rho}=\mu_\rho$, the coefficient of any $\delta_{\gamma_0y_1}$ in $j\star\mu_\rho$ is 
then
\begin{equation}
\label{eqn trho unique monomial diagonal}
\sum_{\gamma}\alpha_\gamma\trace{V_{\gamma_0 y_1,\gamma_0\gamma y_1}}{s}\rho(\gamma_0\gamma)=\rho(\gamma_0).
\end{equation}
By orthogonality of characters, if $j\mu_\rho=\delta_{\rho,\rho'}\mu_{\rho'}$, then the unique solution
to \eqref{eqn trho unique monomial diagonal} is 
\[
\alpha_\gamma\trace{V_{\gamma_0 y_1,\gamma_0\gamma y_1}}{s}=
\alpha_\gamma\trace{V_{y_1,\gamma y_1}}{s}=
\rho(\gamma^{-1}).
\]
For this to hold for all $s$ with $Y^s=Y$, $\trace{V_{y_1,\gamma y_1}}{s}$ must be locally constant in 
$s$, so $V_{y_1,\gamma y_1}$ has trivial action and $j=t_\rho$.
\end{proof}
\subsubsection{Nonvanishing of the rigid determinant}
%
%
\begin{theorem}
\label{thm rigid pairing factors}
Let be $A$ the rigid matrix for the basis ${\phir}^{-1}(\beta)$. If $P_{W_f}(q)\neq 0$, then 
$\det A\neq 0$. 

Further, we have a factorization of matrices over $\mathcal{A}$
\begin{equation}
\label{eq body thm factor}
B_\beta={\phir_\beta}^{T}A.
\end{equation}
The matrix $B_\beta$ has integer entries, is block-diagonal, and $\det B_\beta\neq 0$. The matrix $\phir$ is upper-triangular
with entries in $\mathcal{A}$.
\end{theorem}
\begin{proof}
By Lemma \ref{lem J pairing is block diagonal}, $B_\beta$ has integer entries and is block 
diagonal. The matrix $\phir$ is upper-triangular by Definition \ref{dfn alpha ordered basis} together 
with properties  \eqref{eqn Goldman on Cw} and \eqref{eqn phi upper-triangular} of the map $\phi$
and the involution ${}^\dagger(-)$.

Equation \eqref{eq body thm factor} holds for all $q>1$ by definition of $A$, hence  holds over 
$\mathcal{A}$.

Now let $q>1$. By Proposition \ref{prop phir restricts to isomorphism}, the matrix $\phir$ is invertible.
As the matrix $A$ is invertible in this case by Theorem \ref{thm CH rigid pairing}, it follows 
that $\det B_\beta\neq 0$. But $B_\beta$ is independent of $q$ and so $B_\beta$ is invertible for all $q$.
Now suppose only that $P_{W_f}(q)\neq 0$. Then $\det{\phir}^T\neq 0$.
\end{proof}
The change of basis from either of the Kazhdan-Lusztig bases the standard basis is invertible 
for all $q$, whence
\begin{theorem}
\label{cor main thm body}
The rigid determinant for the basis $\sets{T_\oO}{\oO\in\cl(W)_0}$ is nonzero whenever 
$P_{W_f}(q)\neq 0$.
\end{theorem}
\subsection{Application: Formal degrees of unipotent discrete series representations}
\label{subsection application formal degrees}
In this section, we drop the splitness hypothesis and $\G$ is just connected reductive over $F$ with quasi-split inner form $\G^*$. Let $\omega$ be a 
unipotent discrete series representation of $G$ and let $d(\omega)$ be its formal degree. 
By \cite{Solleveldfdegsequel}, there is a unique rational function in $\bq$ that
specializes to $d(\omega)$ for all $q>1$. Therefore we think of $d(\omega)$ as a rational function of $\bq$. 
Recall Solleveld's unipotent Local Langlands Correspondence for rigid inner twists 
\cite[Thms. 1, 3]{Solleveldfdegsequel}, and that, by feature (l) of \textit{loc. cit.}, the HII conjecture 
holds for unipotent representations of $\G(F)$.

Now we can prove 
\begin{theorem}
\label{thm fdeg body}
Let $\G$ be as above with $\G^*$ split over $F$. Let $\omega$ be a unipotent discrete 
series representation of $\G(F)$ and let $d(\omega)$ be its formal degree, 
thought of as a rational function of $\bq$. 
Then the denominator of $d(\omega)$ divides a power of $P_W(\bq)$, where $W$ is 
the Weyl group of $G^\vee$.
\end{theorem}
\begin{proof}
Let the enhanced $L$-parameter of $\omega$ be $(\varphi, \rho)$. By the unipotent HII conjecture for $\G(F)$, 
there is a rational function $\Gamma_\varphi(\bq)$ of $\bq$ such that
\[
d(\omega)=\dim(\rho)\Gamma_\varphi(q)
\]
up to a multiplicative constant independent of $F$ or $\omega$. In particular, if $\omega^*$ is the 
discrete series representation of $\G^*(F)$ with parameter $(\varphi,\triv)$, then 
$d(\omega^*)=\Gamma_\varphi(q)$ up to a multiplicative constant of the same nature, and $\omega^*$ is 
Iwahori-spherical. Only the Radon-Nikodym derivative of $d(\omega)$ with respect to the 
Haar measure on the orbit of $\omega$ under unramified unitary characters of $\G(F)$ depends on $\bq$.
Therefore we may assume that $\G^*$ is semisimple.

Therefore it suffices to prove the theorem for Iwahori-spherical discrete series representations of semisimple 
split groups. In this case, by Propositions \ref{prop phir isomorphism q>1 and density for J} and
\ref{prop phir restricts to isomorphism}, there is $j$ in $J$ such that 
\[
\trace{[\pi]}{j}=\delta_{[\pi],[\omega^*]}
\]
and hence in fact
\[
\trace{\pi}{j}=\delta_{\pi,\omega^*}.
\]
On the other hand, we have $j=\sum_{w}\alpha_wt_w$ for $\alpha_w\in\mathcal{A}[\frac{1}{P_{W_f}(\bq)}]$. 
Hence for all $q>1$, the Plancherel formula gives
\[
\sum_{w}\alpha_w t_w(1)=j(1)=d(\omega^*),
\]
and by the main theorem of \cite{Plancherel}, multiplying by a power of $P_{W_f}(q)$ depending only on 
$\tilde{W}$ clears denominators of the left hand side.

\end{proof}
\begin{rem}
\label{rem extension to quasi-split}
There are some simpleminded additional hypotheses under which one can do slightly better.
Namely, suppose that $\G^*$ splits over an unramified extension $E$ of $F$ of degree $m$.
Then $\varphi_\omega|_{W_E}$ is unipotent, and if it remains discrete, then \cite{Solleveldfdegsequel} gives 
that
\[
\Gamma_{\varphi_\omega|_{W_{E}}}(\bq)=\Gamma_{\varphi_{\omega}}(\bq^n)
\]
for some $n$. By Theorem \ref{thm fdeg body}, multiplication by a power of
$P_W(\bq^n)$ clears denominators of the left hand side. 

For example, $\varphi_\omega|_{W_E}$ remains discrete (possibly after enlarging $E$) if either the order of 
$s=\varphi_{\omega}(\mathrm{Fr})$ is coprime to $m$, or of course if $\varphi_{\omega}$ is inertially discrete. 
\end{rem}

\subsection{Examples of the rigid pairing}
\label{section Examples}
As remarked above, we do not know how to canonically extract a basis from the spanning set $S$ that we have produced.
In this section we will compute $B_\beta$ (and in low ranks $\phir_\beta$) for specific choices 
of $\beta$ in the cases corresponding to Iwahori-Hecke algebras of $\SL_2(F)$, $\PGL_2(F)$, and $\SO_7(F)$.
In \cite[Example 11]{Plancherel}, we essentially computed the rigid pairing for $\Jrs$
when $G=\SO_5(F)$.
\subsubsection{$\SL_2(F)$}
\label{subsection example SL2}
We first consider the rigid pairing for 
\[
W=\genrel{s_0,s_1}{s_0^2=s_1^2=1},
\]
that is, for $G=\SL_2(F)$ and $\Gd=\PGL_2(\C)$. Recall from Example \ref{ex SL2 central ext}
and \eqref{eqn SL2 principal series decomposition} that we can chose the basis
$\alpha=\sett{\St, \pi^{G(\Oo)},\pi^{K'}}$ of the rigid quotient. 
It then follows from \cite{DSL2}  that
we can chose the ordered basis
and $\beta=\sett{t_{s_1}, t_{s_0}, t_{1}}$ 
of $\Jr$. With these bases $B_\beta$ is the identity matrix, and the $\phir_\beta$ is given by
\[
\begin{pmatrix}
q^{\frac{1}{2}}+q^{-\frac{1}{2}} &&1\\
&q^{\frac{1}{2}}+q^{-\frac{1}{2}}&1\\
&&1
\end{pmatrix}.
\]
Hence $\det(\phir)^TB_\beta=(q^{\frac{1}{2}}+q^{-\frac{1}{2}})^2$ and the rigid pairing is perfect if 
$q\neq -1$. 

To compare with \cite{CH}, the matrix corresponding
to change of basis from the ordered basis $\{T_{s_1}, T_{s_0}, T_1\}$ to the basis $\{C'_w\}$ 
followed by the involution ${}^\dagger(-)$ is 
\[
\diag(-q^{1/2},-q^{1/2},1)
\]
and the change of basis matrix from our basis to that of \cite{CH} is 
\[
\begin{pmatrix}
&[\pi^{G(\Oo)}]&[\pi^{K'}]&\St\\
\St &0&0&1\\
\pi^+ &1&-1&0\\
i_\emptyset &0&1&0
\end{pmatrix}.
\]
Composing appropriately we see that the two determinants match.
\subsubsection{$\PGL_2(F)$}
Now consider the extended affine Weyl group
\[
W=\genrel{s_0,s_1,\omega}{s_0^2=s_1^2=\omega^2=1, \omega s_0=s_1\omega}.
\]
This case corresponds to $G=\PGL_2(F)$, $\Gd=\SL_2(\C)$. In this case all the unitary principal series $\pi(\nu)$
are irreducible, but $G$ has a nontrivial unitary character given by the spinor norm and corresponding to 
the nontrivial element in its fundamental group. In this case we have $J_{1}=K(\pt/\Z/2\Z)$, where 
$\Z/2\Z=Z(\SL_2(\C))$. Thus we have that $\alpha=\sett{[\pi(\nu)], [\St], [-\St]}$ is a basis of $\RHr$, 
and we have
\[
\trace{\St}{t_1}=1=\trace{\St}{t_{\omega}},~\trace{-\St}{t_\omega,}=-1=-\trace{\St}{t_1},
\]
while $\trace{\pi(\nu)}{t_{s_0}}=1$. Hence the rigid pairing matrix for the basis 
$\beta=\sett{t_{s_1}, t_1, t_\omega}$ is 

\begin{center}
\begin{tabular}{|c|c|ccc|}
\hline 
elt.\textbackslash rep. & $a$ & \multicolumn{1}{c|}{$[\pi(\nu)]$} & $[\St]$ & \multicolumn{1}{:c|}{$[-\St]$} \\ 
\hline 
$[t_{s_0}]$ & 1 & \multicolumn{1}{c|}{1} &  &  \\ 
\hline
$[t_{1}]$ &  0&\multicolumn{1}{c|}{} & 1 & \multicolumn{1}{:c|}{1} \\ 
$[t_\omega]$ &  0& \multicolumn{1}{c|}{}& 1 & \multicolumn{1}{:c|}{-1} \\ 
\hline
\end{tabular} 
\end{center}

Adapting \cite{DSL2}, we compute that 
\[
\phi(C'_\omega)=\phi(T_\omega)=t_\omega+t_{\omega s_0}+t_{\omega s_1}.
\]
Thus
\[
\phir
=
\begin{pmatrix}
q^{1/2}+q^{-1/2} & 2 & 2 \\ 
 &  1&  \\ 
 &  &1 
\end{pmatrix}
\]
and the rigid pairing is perfect whenever $q\neq -1$.
\subsubsection{$\SO_7(F)$}
\label{subsection SO7}
In this example, we consider type $\tilde{B}_3$. Let $G=\SO_7(F)$, so that $\Gd=\Sp_6(\C)$ with 
conventions as in \cite{BDD}.
\paragraph{The tempered dual.}
The below table gives the Iwahori-spherical tempered representations of $G=\SO_7(F)$
in the form $i_P^G(\nu\otimes\sigma)$ with $\sigma$ a discrete series representation. We recall
the few cases of reducibility of these inductions immediately below. The rows list unipotent conjugacy 
classes in $\Gd=\Sp_6(\C)$ such that all  tempered standard modules $K(u,s,\rho)$ are in row $u$,
recording which summand of $J$ acts on each representation.

\begin{center}
\begin{tabular}{|c|c|c|c|c|c|c|c|}
\hline 
Cell\textbackslash Levi & $\GL_1^{\times 3}$ & $\GL_1^{\times 2}\times\SO_3$ & $\GL_1\times\GL_2$ & $\GL_1\times\SO_5$ & $\GL_2\times\SO_3$ & $\GL_3$ & $\SO_7$  \\ 
\hline 
$(1,\dots, 1)$ & $\pi(\nu)$ &  &  &  &  &  &  \\ 
\hline 
$(2,1,1,1, 1)$ &  & $\nu_1\otimes\nu_2\rtimes\pm\St_{\SO_3}$ &  &  &  &  &   \\ 
\hline 
$(2,2,1,1)$ &  &  & $\nu\otimes\xi\St_{\GL_2}$ & $\nu\rtimes\pm\tau_2(\SO_5)$ &  &  &   \\ 
 &  &  & $\xi^2=1$ &  &  &  &   \\ 
\hline 
$(2,2,2)$ &  &  &  &  & $\xi\St_{\GL_2}\rtimes\pm\St_{\SO_3}$ &  &   \\ 
&  &  &  &  & $\xi^2=1,\xi\neq 1$ &  &   \\ 
\hline 
$(4,1,1)$ &  &  &  & $\nu\rtimes\pm\St_{\SO_5}$ &  &  &   \\ 
\hline 
$(3,3)$ &  &  &  &  &  & $\xi\St_{\GL_3}$ &   \\ 
\hline 
$(4,2)$ &  &  &  &  &  &  & $\pm\tau_2(\SO_7)$  \\ 
&  &  &  &  &  &  & $\pm\tau_2'(\SO_7)$  \\ 
&  &  &  &  &  &  & $\pm\tau_3(\SO_7)$  \\ 
\hline
$(6)$ &  &  &  &  &  &  & $\pm\St_{\SO_7}$  \\ 
\hline 
\end{tabular} 
\end{center}
\vspace{5mm}

The discrete series representations of $G$ are given by Reeder \cite{Reeder}. Hence
we need only determine reducibility and compute traces of some elements $t_{\omega d,\rho}$ on a basis of the 
rigid quotient. We do so cell by cell.
 
For cells $(3,3)$ and $(4,1,1)$, by Theorems 11.4 and 11.2 of \cite{Tadic}, the 
respective tempered representations in the table are all simple. For these cells Lusztig's conjecture is 
true by \cite{BO}; each of the corresponding summands $J_u$ are matrix rings, and
$t_{d,\rho}=t_{d,\triv}=t_d$.
The same is well-known to be true for the lowest cell $(1,\dots, 1)$; the unitary
principal series are all irreducible as $\Gd$ is simply-connected. In each case the canonical one-sided cell
of each gives a canonical choice of distinguished involution, which we write simply as $t_d$.

For the cell $(2,1,1,1, 1)$, we have that $\Zur=\Sp_4\times\Z/2\Z$ is 
connected modulo the centre and has no projective representations.
Thus by \cite{BO} we have
\[
J_{(2,1,1,1, 1)}\simeq\Mat_{24\times 24}(R(\Sp_4(\C)\times\Z/2\Z)).
\]
In particular, $\mathrm{rank}(\pi(t_d))=1$ for all $d$ in $\cc$, and $t_{d,\rho}=t_{d,\triv}=t_d$.
Moreover, the corresponding tempered
representations are all simple. Indeed, using Goldberg's product formulas \cite{Goldberg} for the $R$-group, we have
\begin{align*}
R_{\GL_1\times\GL_1\times\SO_3}^{\SO_7}(i_P^G(\nu_1\otimes\nu_2\otimes\pm\St_{\SO_3}))&=R_{\GL_1\times\SO_3}^{\SO_5}(\nu_1\otimes\pm\St_{\SO_3}))
\times R_{\GL_1\times\SO_3}^{\SO_5}(\nu_2\otimes\pm\St_{\SO_3}))=\{1\},
\end{align*}
by \cite[Prop. 3.2]{Matic}, because $\nu_1,\nu_2$ are unitary.

The cell $(2,2,1,1)$ acts on representations belonging to more than one packet.
For the cell $(2,2,1,1)$, the tempered representations 
$i_{P}^G(\nu\otimes\xi\St_{\GL_2})$ are reducible if and only if $\xi^2=1$.
Indeed, either by noting that for $s\in\OO_2\subset \Sp_4$, 
\begin{equation}
\label{eq components of centralizers same 2211}
\pi_0(Z_{\OO_2}(s))=\pi_0(Z_{\SL_2\times\OO_2}(s))=\pi_0(\SL_2\times Z_{\OO_2}(s)),
\end{equation}
or using \cite{Goldberg}, we obtain
\[
R_{\GL_1\times\GL_2}^{\SO_7}(\nu\otimes\xi\St_{\GL_2})=R_{\GL_1}^{\SO_3}(\nu)\times R_{\GL_2}^{\SO_5}(\xi\St_{\GL_2})=R_{\GL_2}^{\SO_5}(\xi\St_{\GL_2}).
\]
By \cite[Prop. 3.3]{Matic}, this $R$-group is nontrivial if and only if $\xi^2=1$. Here we also used that unitary principal series of $\SO_3\simeq\PGL_2$ are all irreducible. Further, by \cite{Matic},
$
i_{GL_2\times\GL_1}^{\SO_5\times\GL_1}(\nu\otimes\xi\St_{\GL_2})
$
is a direct sum of two tempered representations. We denote their inductions to $\SO_7(F)$ by 
$\tau_\triv(\SO_5)$ and $\tau_{\mathrm{sgn}}(\SO_5)$. 

The tempered representations $i_P^G(\nu\otimes\pm\tau_2(\SO_5))$ are all irreducible, by 
\eqref{eq components of centralizers same 2211}. 

The structure of $J_{(2,2,1,1)}$ was computed in \cite{QiuIII}, where it was also shown that Lusztig's
conjecture holds for this cell. Matching the behaviour of the tempered representations to Lusztig's 
classification informs which elements to include in our basis.

The summand $J_{(2,2,2)}$ was analyzed in Example \ref{ex BDD central ext}. We can select
two elements $t_d$ to add to our basis by consulting the table in \cite{BDD}. We write 
the canonical distinguished involution (the last entry in the table in \textit{op. cit.})
just as $t_d$.

The subregular cell $(4,2)$ has six simple modules, comprising the entire Iwahori-spherical discrete 
series of $G(F)$ save the Steinberg representation and its twist. It also acts on representations belong 
to more than one packet. We have simply 
\[
J_{(4,2)}=\End(\tau_2(\SO_7))\oplus\End(-\tau_2(\SO_7))\oplus \End(\tau_2'(\SO_7))
\oplus \End(-\tau_2'(\SO_7))\oplus \End(\tau_3(\SO_7))\oplus \End(-\tau_3(\SO_7)).
\]
The summand $J_{(4,2)}$ is described in \cite[Section 12.3 (B)]{XiBook} along with its simple modules.
In the notation of \textit{loc. cit.}, in which $a_1a_2\in\pi_0(\Zur)$ is the nontrivial central element
of $Z(G^\vee)$, we have
$E_1=\tau_3$, $E_2=-\tau_3$, $E_3=-\tau_2$, $E_4=\tau_2$, $E_5=-\tau_2'$, and $E_6=\tau_2'$ (c.f.
the dimensions computed in \cite{XiBook} and \cite{Reeder}).
Using \textit{op. cit.}, we compute traces as in the below table.

For the cell $(6)=\{1,\omega\}$, the traces are obvious.
\paragraph{The rigid pairing.}
The above description of the tempered dual allows us to set
\begin{multline*}
\alpha=\left\{
[\nu_1\otimes\nu_2\rtimes\St_{\SO_3}] , [\nu_1\otimes\nu_2\rtimes-\St_{\SO_3}], 
[\tau_\mathrm{sgn}(\SO_5)] , [\tau_\triv(\SO_5)]  , {[\nu\rtimes\tau_2(\SO_5)]}, \right.
\\
\left.
{[\nu\rtimes-\tau_2(\SO_5)]}, [\tau_{\triv}] , {[\tau_\mathrm{sgn}]}, {[\nu\rtimes\St_{\SO_3}]}, {[\nu\rtimes-\St_{\SO_3}]} , [\xi\St_{\GL_3}]
\right\}
\end{multline*}
in the sense of Definition \ref{dfn alpha ordered basis}. According to the above, we chose
$\beta$ as in the below $20\times 20$ matrix $B_\beta$, split into:
\setlength{\tabcolsep}{4pt}

\vspace{2mm}
\hspace{-27mm}
\begin{tabular}{|c|c|c|c|*{11}{c|}}
  \hline
  Cell\textbackslash Rep. & $a$ & $\text{elt. of}~\beta$ & $[\pi(\nu)]$ & $[St_{\SO_3}]$ & $[-\St_{\SO_3}]$ & $[\tau_{\triv}(\SO_5)]$ & $[\tau_{\sgn}(\SO_5)]$ & $[\tau_2(\SO_5)]$ & $[-\tau_2(\SO_5)]$ & $[\tau_\sgn]$ & $[\tau_\triv]$ & $[\St_{\GL_3}]$ \\
  \hline
  $(1,\cdots,1)$ & 9 & $t_{w_0}$ & 1 & & & & & & & & & \\
  \hline
  $(2,1,\dots,1)$ &6 & $t_d$ & & 1 & 1 & & & & & & & \\
  & & $t_{\omega d}$ & & 1 & -1 & & & & & & & \\
  \hline
  $(2,2,1,1)$ & 4 & $t_{d_{\Gamma_{03}}, \triv}$ & &  &  & 1 & & & & & & \\
  & & $t_{d_{\Gamma_{03}},  \sgn}$ & & & & & 1 & & & & & \\
  & & $t_{\omega d_{\Gamma_{02}}}$ & &  & & 1& 1& 1 &-1 & & & \\
  & & $t_{d_{\Gamma_{02}}}$ & & & & 1& 1& 1& 1 & & & \\
  \hline
  $(2,2,2)$ & 3 & $t_{013}$ & & & & & & & &1 & & \\
   &  & $t_{d}$ & & & & & & & & &1 & \\
  \hline
  $(3,3)$ & 2 & $t_{d_{3,3}}$ & & & & & & & & & & 1\\
  \hline
\end{tabular}


\vspace{10mm}

\hspace{-27mm}
\begin{tabular}{|c|c|c|c|*{11}{c|}}
  \hline
 Cell\textbackslash Rep.  & $a$ & $\text{elt. of}~\beta$ & $[\St_{\SO_5}]$ & $[-\St_{\SO_5}]$ & $[\tau_3(\SO_7)]$ & 
  $[-\tau_3(\SO_7)]$ & $[-\tau_2(\SO_7)]$ & $[\tau_2(\SO_7)]$ & $[-\tau_2'(\SO_7)]$ & $[\tau_2'(\SO_7)]$ & 
  $[\St_{\SO_7}]$ & $[-\St_{\SO_7}]$ \\
  \hline
  $(4,1,1)$ & 2 & $t_{\omega d}$ & 1 &-1 & & & & & & & & \\
     &  & $t_{d}$ & 1 &1 & & & & & & & & \\
  \hline
  $(4,2)$ & 1 & $t_{s_0}$ & &  & 1 & 1 &1 &1 & & & & \\
  & & $t_{\omega s_0}$ & & & 1& -1& -1  &1 & & & & \\
  & & $t_{\omega s_3, \sgn}$ & &  & & & &  & -1 & 1 & & \\
  & & $t_{s_3,\sgn}$ & & & & & & & 1 & 1 & & \\
  & & $t_{\omega s_3,\triv}$ & & & & & -1& 1&  & & & \\
  & & $t_{s_3,\triv}$ & & & & & 1 & 1 &  & & & \\
  \hline
  $(6)$ & 0 & $t_{\omega}$ & & & & & & & & & 1& -1\\
    &  & $t_{1}$ & & & & & & & &&1 &1 \\
  \hline
\end{tabular}
We denote the unique distinguished involution in a one-sided cell $\Gamma$ by $d_\Gamma$,
with notation as in \cite{QiuIII}.

\appendix
\section{Appendix by Dmitriy Rumynin\footnote[2]{Department of Mathematics, University of Warwick, 
Coventry, CV4 7AL, UK, email: \texttt{D.Rumynin@warwick.ac.uk}}: Adjoint quotient for reductive group}
\newcounter{Coco}

\setcounter{lem}{0}
\setcounter{prop}{0}
\setcounter{cor}{0}
\renewcommand{\thelem}{\Alph{section}\arabic{lem}}
\renewcommand{\theprop}{\Alph{section}\arabic{prop}}
\renewcommand{\thecor}{\Alph{section}\arabic{cor}}

\begin{abstract}
We describe the GIT quotient $\GG$ for a disconnected reductive group $\cG$
  over an algebraically closed field of characteristic zero acting on itself by conjugation.
\end{abstract}

Let $\cG$ be a reductive group over an algebraically closed field $\bF$ of characteristic $0$.
The group $\cG$ being reductive means that the identity component $\cG^0$ is
a connected reductive group.

Johnston, Martin Duro and the author have recently
classified reductive groups with fixed $\Gamma = \pi_0 (\cG)$ and $\cG^0$ up to isomorphism \cite{JMR}.
Building on this work, we describe the adjoint quotient. Note that the same problem was considered
first by Segal for compact groups \cite{Segal}, then by Mohrdieck for semidirect products 
$\cG\cong \cG^0 \rtimes \Gamma$ \cite{diss}, an finally by Springer for semisimple, simply connected
$\cG^0$ \cite{Spr}.

The problem may be considered as an attempt to understand tensor products of simple $\cG$-modules. Indeed, their characters form a basis of the invariant functions on $\cG$ so that we can write
\[
\bF [\GG] \cong \bF[\cG]^{\cG}
\cong \Rep (\cG) \coloneqq \bF \otimes_{\bZ} K_0( \cG-\Rep) \, .
\]


%

\subsection{Subgroups and Components}
For convenience, we use the standard letters, adorned with the superscript $0$,
to denote the standard subgroups associated to the identity component $\cG^0$.
If $\cX\leq \cG$, then simply $\cX^0= \cX \cap \cG^0$,
which may or may be not the identity component of $\cX$.
For instance, a Borel subgroup, a torus, its normaliser and the Weyl group are
\[
\cB^0 \geq \cT^0 \unlhd  \cN^0 \coloneqq N_{\cG^0}(\cT^0), \quad \cW^0 \coloneqq \cN^0/\cT^0.
\]
These groups have ``disconnected'' counterparts that hoist no superscript:
\[
\cN \coloneqq N_{\cG}(\cT) \unrhd \cT \coloneqq N_{\cG}(\cB^0,\cT^0) \leq \cB \coloneqq N_{\cG}(\cT) \leq \cG, \quad \cW \coloneqq \cN/\cT.
\]
Note that the subgroups $\cB,\cG,\cN$ meet every connected component of $\cG$ \cite{JMR}. Besides, $\cT = N_{\cG}(\cT^0)$ \cite{JMR}.
Hence, we have six exact sequences of groups
\[
1 \rightarrow \cX^0 \rightarrow \cX \rightarrow \Gamma \rightarrow 1, \quad \cX\in \{\cA,\cB,\cG,\cN,\cT,\cW\} \, ,
\]
where 
the new, yet undefined subgroup $\cA$ is a finite subgroup of $\cT$ that meets every component of $\cG$
and $\cA^0 \coloneqq \cG^0 \cap \cA$ is central in $\cG^0$.
Such subgroup exists \cite{JMR}. We fix it once and for all.
Additionally, for an element $\gamma\in \Gamma$, we choose a fixed lifting $\dot{\gamma}\in \cA$.
The key variety is the component quotient $\VG \coloneqq (\cG^0 \dot{\gamma})\!\sslash\! \cG^0$, partially due to the following obvious observation.
\begin{lem} \label{lem1}
  The adjoint quotient can be taken in two steps
  \[
  \cG\!\sslash\! \cG^0  \cong \coprod_{\gamma \in \Gamma} \VG \ \mbox{ and }  \  
  \GG \cong (\cG\!\sslash\! \cG^0)\!\sslash\! \Gamma \cong \coprod_{\gamma \in \Gamma\!\sslash\!\Gamma} \VG\!\sslash\! \Gamma (\gamma),
    \]
where the second union goes over the conjugacy classes in $\Gamma$ and $\Gamma (\gamma)$ is the centraliser of $\gamma$.
\end{lem}  
Now we collect the facts that easily follows from the known results \cite{diss,Spr}.
Denote by $\cT_\gamma$ the identity component of the centraliser of $\dot{\gamma}$ in $\cT^0$. 
\begin{prop} \label{Prop2}
  Suppose that $\cG^0$ is semisimple. The following statements hold.
  \begin{enumerate}
  \item For each $\gamma$ there exist
    a finite group $\Delta_\gamma$ and its action on $\cT_\gamma \dot{\gamma}$ such that the natural embedding
    \[ \iota: \cT_\gamma \dot{\gamma} \hookrightarrow \cG^0 \dot{\gamma}, \quad x\dot{\gamma}\mapsto x\dot{\gamma} \]
    induces an isomorphism $\cT_\gamma \dot{\gamma} \!\sslash\! \Delta_{\gamma} \xrightarrow{\cong} V_\gamma = \cG^0 \dot{\gamma} \!\sslash\! \cG^{0}$.
  \item 
    If $\dot{\gamma}$ acts without non-identity fixed points on $\pi_1 (\cG^0)$, then
    the quotient map  $\cG^0 \dot{\gamma} \rightarrow V_\gamma$ admits a simultaneous resolution of singularities. 
\setcounter{Coco}{\value{enumi}}
\end{enumerate}
  Suppose further that $\cG^0$ is simply connected. Then the following statements hold.
\begin{enumerate}
\setcounter{enumi}{\value{Coco}}
\item The fixed point ring $\bF [\cT_{\gamma} \dot{\gamma}]^{\Delta_\gamma}$ is a polynomial ring.
\item The quotient space $V_{\gamma}$ is an affine space $\bA^{d(\gamma)}_{\bF}$ where $d(\gamma)=\dim \cT_\gamma$. 
\item The quotient map  $\cG^0 \dot{\gamma} \rightarrow V_\gamma$ admits a section.
\end{enumerate}
\end{prop}  
\begin{proof}
  Notice that the quotient $V_{\gamma}$ depends only on the adjoint automorphism $\AdRumynin (\dot{\gamma})$:
  \[ g (h \dot{\gamma}) g^{-1} = g h \dot{\gamma} g^{-1}\dot{\gamma}^{-1} \dot{\gamma} =
g h \AdRumynin (\dot{\gamma})(g^{-1}) \dot{\gamma} \, .
  \] 
  Our choice of $\dot{\gamma}$ ensures that  $\AdRumynin (\dot{\gamma})$ has finite order $n$ that divides $|\gamma|$.
  Let $C_n = \langle \tilde{\gamma} \rangle$ be the cyclic group of order $n$.
  Form a semidirect product $\widetilde{\cG} \coloneqq \cG^0 \rtimes C_n$ where the generator $\tilde{\gamma}$ acts via  $\AdRumynin (\dot{\gamma})$.
  Since $V_\gamma = \cG^0 \dot{\gamma} \!\sslash\! \cG^{0} \cong \cG^0 \tilde{\gamma} \!\sslash\! \cG^{0}$, the first statement is essentially \cite[Th. 1.1]{diss}.

  The remaining statements are equally known. Statement~2 is \cite[Th. 1.3]{diss}. 
  Statement~3 is \cite[Th. 1]{Spr}. Statement~4 is \cite[Cor. 2]{Spr}. 
  Statement~5 is \cite[Th. 1.2]{diss}. 
\end{proof}
\begin{cor}
  If $\cG^0$ is semisimple and simply connected, then
  \[
  \Rep (\cG) \cong \bigoplus_{\gamma \in \Gamma\!\sslash\!\Gamma}   \bF[x_1, \ldots , x_{d(\gamma)}]^{\Delta_\gamma} \, .
  \]
\end{cor}

\subsection{General reductive groups}
To deal with a general reductive $\cG$, we need the following observation.
\begin{lem} \label{seqREP}
Suppose we have a pair of reductive groups $\cG$ and $\wG$ that fit into an exact sequence with a finite abelian $\Phi$ 
\begin{equation} \label{seqGG}
1 \rightarrow \Phi \longrightarrow \wG \longrightarrow  \cG \rightarrow 1 \, . 
\end{equation}
Then there is an action of $\Phi$ on $\Rep (\wG)$ and $\Rep (\cG)\cong \Rep (\wG)^{\Phi}$.
\end{lem}
To a general $\cG$ one can associate the central torus $\cZ\coloneqq Z(\cG^0)^0$ and the semisimple,
simply connected group $\cH$,
defined as the universal cover of $[\cG^0,\cG^0]$. They come equipped with an exact sequence
\begin{equation} \label{seqGG2}
1 \rightarrow \Phi \longrightarrow \cZ \times \cH \xrightarrow{\ \ \mu \ \ }  \cG^0 \rightarrow 1
\end{equation}
where the map $\mu$ comes from multiplication in $\cG$ and the finite abelian group $\Phi$ is its kernel. Let us call $\cG$ {\em common}
if there exists a reductive group $\wG$ such that
$\wG^0 \cong \cZ \times \cH$ and the map $\mu$ extends to a map $\mu: \wG \rightarrow  \cG$,
extending the sequence~\eqref{seqGG2} to a sequence~\eqref{seqGG}.

Note that an obstruction to existence of $\wG$ is the Taylor cocycle, an element of $H^3 (\pi_0 (\cG), \Phi)$. It vanishes if and only if $\wG$ exists \cite{RVW,Tay}.

\begin{prop} \label{Prop5}
  Suppose that $\cG$ is common. Let the group $\wG$ be as above.
  The following statements hold.
  \begin{enumerate}
  \item For each $\gamma\in \Gamma = \pi_0 (\cG)$
    the quotient space $\widetilde{V_{\gamma}} = \wG^0 \dot{\gamma} \!\sslash\! \wG^{0}$
is a smooth affine variety $\cZ \times \bA^{d(\gamma)}_{\bF}$ where $d(\gamma)=\dim \cT(\cH)_\gamma$.
\item For each $\gamma\in \Gamma$
    the quotient space ${V_{\gamma}} = \cG^0 \dot{\gamma} \!\sslash\! \cG^{0}$
    is the quotient of the former $\widetilde{V_{\gamma}}\!\sslash\! \Phi$ by the finite abelian group $\Phi$.
  \item  The adjoint quotient can be written as
    \[
  \Spe (\Rep (\cG)) \cong \GG \cong (\cG\!\sslash\! \cG^0)\!\sslash\! \Gamma \cong \coprod_{\gamma \in \Gamma\!\sslash\!\Gamma} \VG\!\sslash\! \Gamma (\gamma),
  \]
\item The variety $\cG\!\sslash\! \cG$ has at worst finite quotient singularities.
\item The variety $\cG\!\sslash\! \cG$ is Cohen-Macaulay.
\end{enumerate}
\end{prop}
\begin{proof}
  The first statement follows from Proposition~\ref{Prop2} for $\cH$. The second is  obvious. The third one follows from Lemma~\ref{lem1}. 
  The last two statements are standard.
\end{proof}  

What can be done about an uncommon group $\cG$? Similarly to Proposition~\ref{Prop2}, the variety $V_{\gamma}$ depends only on the adjoint automorphism $\AdRumynin (\dot{\gamma})$. We can form the group $\wG$ ``locally'', i.e., consider $(\cZ \times \cH) \rtimes C_n$. Hence, Proposition~\ref{Prop5} yields us
\begin{cor} For any reductive $\cG$
the variety $\cG\!\sslash\! \cG$ is Cohen-Macaulay with at worst finite quotient singularities.
\end{cor}

Of course, a careful reader would observe that we did not need any theory for the first conclusion. 
The variety $\cG\!\sslash\! \cG$ is Cohen-Macaulay immediately by Hochster-Roberts Theorem \cite{HR}.
In our defence, our result also describes singularities, while our method allows to prove that 
$\cG\!\sslash\! \cG$ is Cohen-Macaulay in a finite, not very small characteristic $p$. Namely,
if $p$ does not divide $|\Gamma|$ and the determinant of the Cartan matrix of $\cG$. 

Finally, let us reexamine the exact sequence~\eqref{seqGG}. Let $\Phi^\vee$ be the group of the linear characters of $\Phi$.
For $\alpha\in\Phi^\vee$ by $\Rep^\alpha (\wG)$ we understand the corresponding space of $\alpha$-semiinvariants.
\begin{lem}
\label{lem3}
 In the context of Lemma~\ref{seqREP}
  we have a decomposition
  \begin{equation}
    \Rep (\wG) \;=\; \bigoplus_{\alpha\in\Phi^{\vee}} \Rep^{\alpha} (\wG)
\end{equation}
  where each $\Rep^{\alpha} (\wG)$ is a maximal Cohen-Macaulay $\Rep (\cG)$-module.
\end{lem}

\subsection{Cartan Subgroups}
\label{subsection appendix Cartan subgroups}
Recall that an algebraic subgroup $\cC\leq \cG$ is called a Cartan subgroup if the following
properties hold:
\begin{itemize}
\item  $\cC$ is diagonalisable (commutative and all elements are semisimple),
\item $\cC$ has finite index in its normaliser,
\item the component group $\pi_0 (\cC)$ is cyclic.
\end{itemize}
Both Segal \cite{Segal} and  Mohrdieck \cite{diss} pay considerable attention to Cartan subgroups since they are a key technical ingredient for tackling
the adjoint quotient. Here we only summarise the main facts. Their proofs are straightforward adaptations of their proofs \cite{diss,Segal}.

\begin{prop} 
The following statements for Cartan subgroups of $\cG$ hold. 
  \begin{enumerate}
  \item Every semisimple $g\in \cG$ is contained in a Cartan subgroup.
  \item If $\cC$ is a Cartan subgroup, then $\cC^0=\cG^0\cap \cC$ is the identity component of $\cC$ and $\cC^0$ is a torus in $\cG^0$.
  \item Let $g\in \cG^0 \dot{\gamma}$ be a semisimple element. If $\cC$ is a Cartan subgroup such that $g\in \cC$ and 
    the image of ${g}$ is a generator of $\pi_0 (\cC)$, then every semisimple element of $\cG^0 \dot{\gamma}$ is conjugate to an element of $\cC^0 g$.
  \item The map
    \[
    \{\mbox{\rm{Cartan subs of }} \cG \} \rightarrow \{\mbox{\rm{Cyclic subs of }} \pi_0 (\cG) \}, \quad
    \cC \mapsto \cC / (\cC \cap \cG^0)
    \]
    is a bijection on the equivalence classes under conjugation by $\cG$. 
  \end{enumerate}  
\end{prop}

\bibliography{rigid-determinant-biblio.bib}

\end{document}